\documentclass[11pt,twoside]{amsart}
\usepackage{hyperref}
\usepackage[active]{srcltx}
\usepackage[curve]{xypic}
\usepackage{graphicx}
\usepackage{longtable}
\usepackage{rotating}
\usepackage{multirow}
\usepackage{titletoc} 
\usepackage{tikz}

\usepackage[active]{srcltx}
\usepackage{epsfig,amsmath}
\usepackage[english]{babel}
\usepackage{amsmath,amsfonts,amssymb,amscd,color,epsfig,amsthm}

\newtheorem{newthm}{Theorem}
\newtheorem{theorem}{Theorem}[section]
\newtheorem{lemma}[theorem]{Lemma}
\newtheorem{proposition}[theorem]{Proposition}
\newtheorem{corollary}[theorem]{Corollary}

\newtheorem{definition}[theorem]{Definition}

\theoremstyle{remark}
\newtheorem{example}{\bf Example}[section]

\theoremstyle{plain}

\newtheorem{fact}[theorem]{\bf Fact}

\newtheorem{assumption}[theorem]{\bf Assumption}

\newtheorem{prop}[theorem]{\bf Proposition}
\newtheorem{coro}[theorem]{\bf Corollary}

\newtheorem{remark}[theorem]{Remark}

\numberwithin{equation}{section}

\newcommand{\bess}{\begin{eqnarray*}}
\newcommand{\eess}{\end{eqnarray*}}

\newcommand{\C}{\mathbb{C}}
\newcommand{\EC}{\widehat{\mathbb{C}}}

\newcommand{\wt}{\widetilde}
\newcommand{\mb}{\mathbb}
\newcommand{\mc}{\mathcal}

\newcommand{\tu}{\textup}
\newcommand{\ol}{\overline}
\newcommand{\es}{\emptyset}

\newcommand{\tb}{\textbf}

\newcommand{\wh}{\widehat}
\newcommand{\olC}{\widehat{\mathbb{C}}}
\newcommand{\olD}{\overline{\mathbb{D}}}

\def\ee{{\mathbf e}}
\def\c{{\mathbf c}}

\def\C{\mbox{$\mathbb C$}}

\def\lv{ \left(\begin{matrix} }
	\def\rv{\end{matrix}\right)}

\def\dw{{\dw}}

\newcommand{\mylabel}[1]{\label{#1}}

\newcommand{\REFEQN}[1] { \begin{equation}\mylabel{#1} }
\newcommand{\ENDEQN}{\end{equation}}
\newcommand{\REFTHM}[1] { \begin{theorem}\mylabel{#1} }
	\newcommand{\ENDTHM}{\end{theorem}}
\newcommand{\REFNTH}[1] { \begin{newthm}\mylabel{#1} }
	\newcommand{\ENDNTH}{\end{newthm}}
\newcommand{\REFPROP}[1]{\begin{proposition}\mylabel{#1} }
	\newcommand{\ENDPROP}{\end{proposition} }
\newcommand{\REFLEM}[1]{\begin{lemma}\mylabel{#1} }
	\newcommand{\ENDLEM}{\end{lemma} }
\newcommand{\REFCOR}[1]{\begin{corollary}\mylabel{#1} }
	\newcommand{\ENDCOR}{\end{corollary} }

\begin{document}

\title{Dynamics of Newton maps}


\begin{abstract}
 In this paper, we study the dynamics of Newton maps for arbitrary polynomials.   Let $p$ be an arbitrary polynomial
with at least three distinct roots,  and
   $f$ be its Newton map.  
   It is shown that the boundary $\partial B$ of any immediate root basin $B$ of $f$ is locally connected.
   Moreover, $\partial B$ is a Jordan curve if and only if ${\rm deg}(f|_B)=2$. 
   This implies that the boundaries of all components  of root basins, for all polynomials' Newton maps, from the viewpoint of topology, are tame.
	\end{abstract}
	

\author{Xiaoguang Wang}
\address{Xiaoguang Wang, School of Mathematical Sciences, Zhejiang University, Hangzhou, 310027, P.R. China}
\email{wxg688@163.com}
\author{Yongcheng Yin}
\address{Yongcheng Yin, School of Mathematical Sciences, Zhejiang University, Hangzhou, 310027, P.R. China}
\email{yin@zju.edu.cn}
\author{Jinsong Zeng}
\address{Jinsong Zeng, School of Mathematics and Information Science, Guangzhou University, 510006, P.R. China.}
\email{zeng.jinsong@foxmail.com}


\date{\today}

	\maketitle


\section{Introduction}

Newton's method, also known as the {\it Newton-Raphson method} named after Isaac Newton (1642-1727) and Joseph Raphson (1648-1715), is probably the oldest and most famous iterative process to be found in mathematics. 
 The method was
first proposed to find successively better approximations to the  roots (or zeros) of a real-valued function $p(z)$.
 Picking an initial point $z_0$ near a root of $p$,
Newton's method produces an $n$-th approximation of the root via the formula $z_{n+1}=f_p(z_n)$, where
 $$f_p(z)=z-\frac{p(z)}{p'(z)}$$
 is called the {\it Newton map} of $p$.
Replacing $z_n$ by $z_{n+1}$ generates a sequence of approximations $\{z_n\}$ which may or
may not converge to a root of  $p$.

A brief history of Newton's method, following \cite{A}, is as follows.
Versions of Newton's method had been in existence for centuries previous to Newton and Raphson. Anticipations of Newton's method are found in an
ancient Babylonian iterative method of approximating the square root of $a$,
 $$z_{n+1}=\frac{1}{2}\Big(z_n+\frac{a}{z_n}\Big),$$
which is equivalent to Newton's method for the function $f(z)=z^2-a$. The modern formulation of the method is also attributed to Thomas Simpson (1710-1761) and
Joseph Fourier (1768-1830).

By the mid-1800's, several mathematicians had already examined the
convergence of Newton's method towards the real roots of an equation $p(z)=0$, but
the investigations of Ernst Schr\"oder (1841-1902) and Arthur Cayley (1821-1895) are distinguished from their predecessors 
in their consideration of the convergence of Newton's method to the complex roots
of $p(z)=0$.
%
Schr\"oder and Cayley each studied the convergence of Newton's method for the
quadratic polynomials, and both showed that on either side of the perpendicular
bisector of the roots, Newton's method converges to the root on that particular
side. 
 However, in 1879, Cayley \cite{C} first  noticed the difficulties in generalizing Newton's method to
 cubic polynomials, or general polynomials
with at least three distinct roots. 
In \cite{C}, Cayley wrote


\vspace{4 pt}

\textit{``The solution is easy and elegant in the case of a  quadric equation, but (Newton's method for) the next succeeding case of the cubic equation  appears to present considerable difficulty."}

\vspace{4 pt}

 The study of Newton's method led to the theory of iterations of holomorphic functions, as initiated by Pierre Fatou and Gaston Julia around 1920s.  Since then, the study of Newton maps became 
one  of the  major themes  with general interest, both in discrete dynamical system (pure mathematics), and in root-finding algorithm (applied mathematics), see for example \cite{AR, BFJK1, BFJK2, Ba, Be, HSS, P, Ro07, Ro08, RWY17,  Sh09, Tan97}....



Let $p$ be a polynomial with at least two distinct roots\footnote{the discussion is trivial when $p$ has only one (possibly multiple) root}, and let $\zeta\in\C$ be a root of $p$. For its Newton map $f_p$, the {\it attracting basin} or {\it root basin} of $\zeta$, denoted by $B(\zeta)$,
consists of points $z$ on the Riemann sphere $\EC$ whose orbit $\{f_p^n(z);n\in\mathbb N\}$ (here $g^n$ means the 
$n$-th iterate of $g$) converges to $\zeta$:
$$B(\zeta)=\big\{z\in \EC; f^n_p(z)\rightarrow \zeta \text{ as } n\rightarrow +\infty\big\}.$$
It is well known that $B(\zeta)$ is an open set of $\EC$. 
In the case that $p$ has two distinct (possibly multiple) roots, by quasi-conformal surgery, one can show that $B(\zeta)$ is a quasi-disk and the Julia set
$J(f_p)$ is a quasi-circle.
So this case is easy. 


We  say that a polynomial $p$ is {\it non-trivial} (in the sense of Cayley) if $p$ has at least three distinct roots.  A non-trivial polynomial takes the form
$$p(z)=a(z-a_1)^{n_1}\cdots(z-a_d)^{n_d}$$
where $a\in \mathbb C-\{0\}$, $d\geq 3$, and $a_1,\cdots, a_d\in\mathbb C$ are  distinct roots of $p$, with multiplicities $n_1,\cdots,n_d\geq 1$, respectively. This is the general case, and the
attracting basin $B(\zeta)$  consists of countably many connected components. The one containing $\zeta$ is called the {\it immediate attracting basin} or {\it immediate root basin}, and is denoted by $B^0(\zeta)$.  Przytycki \cite{P} showed that $B^0(\zeta)$ is a topological disk when $p$ is a non-trivial cubic polynomial. 
By means of quasi-conformal surgery, Shishikura \cite{Sh09} proved that the Julia set of the Newton map for any non-trivial polynomial is connected. 
This result is further generalized to Newton's method for entire functions  by Baranski, Fagella, Jarque and Karpinska  \cite{BFJK1}\cite{BFJK2}. 
 This implies, in particular, each component of $B(\zeta)$ is a topological disk.

Although $B=B^0(\zeta)$ has a simple topology, its boundary $\partial B$ exhibits  rich topological structures.
The reason is that the Newton map $f_p$ can have unpredictable dynamics and complicated  bifurcations on  $\partial B$. Therefore for  Newton maps, understanding the topology of $\partial B$  makes a fundamental and challenging
problem from the view point of dynamical system.

Little progress has been made towards the problem, until the ground-breaking  work of Roesch.
In \cite{Ro08}, Roesch proved,  building on previous works of 
Head \cite{He87} and Tan Lei \cite{Tan97}, that $\partial B$  is always a Jordan curve, when 
 $p$ is a non-trivial cubic polynomial and  ${\rm deg}(f_p|_B)=2$. The proof is the  first successful application of the Brannar-Hubbard-Yoccoz puzzle theory to rational maps.  The puzzle theory has also been developed by Roesch, Wang, Yin  \cite{RWY17} to study the local connectivity and rigidity phenomenon in 
 parameter space.


%
%



The main result of the paper, is to give a complete characterization of $\partial B$ for all polynomials'  Newton maps:
%
%

\begin{theorem}\label{main} Let  $f_p$ be the Newton map for any non-trivial polynomial $p$.
Then the boundary $\partial B$ of any immediate root basin $B$ 
is locally connected.
  Moreover, $\partial B$ is a Jordan curve if and only if ${\rm deg}(f_p|_B)=2$.
\end{theorem}

The theorem implies that the boundary of each component of the root basins is locally connected. 
Therefore, {\it the boundaries of all components of root basins, for all polynomials' Newton maps, from the viewpoint of topology, are tame}.  Our argument also has a byproduct: the Julia set of a non-renormalizable 
Newton map is always locally connected, which generalizes Yoccoz's famous theorem  to Newton maps.

Our work  extends Roesch's Theorem \cite[Theorem 6]{Ro08} for cubic Newton maps 
to Newton maps of arbitrary polynomials.
It is distinguished from Roesch's work \cite{Ro08} in two folds. Firstly, 
the invariant graph is different from those in \cite{Ro08}. In our work, we construct only one graph adapted to the puzzle theory:
the one generated by the
{\it channel graph}, 
 while in \cite{Ro08}, countably many candidate graphs are provided, and each of them involves very technical construction.
 Secondly,  each cubic Newton map  has only one free critical point, so the puzzle theory in \cite{Ro08} is same as the quadratic case; 
however, the Newton maps for higher degree non-trivial polynomials can have  more free critical points, and the quadratic puzzle theory does not work here. To deal with this general case, we take advantage of recent developments  \cite{KL1,KL09, KSS07} in multi-critical polynomial dynamics. 

\begin{figure}
	\center
	\begin{tikzpicture}
	\node at (0,0) {\includegraphics[width=6cm]{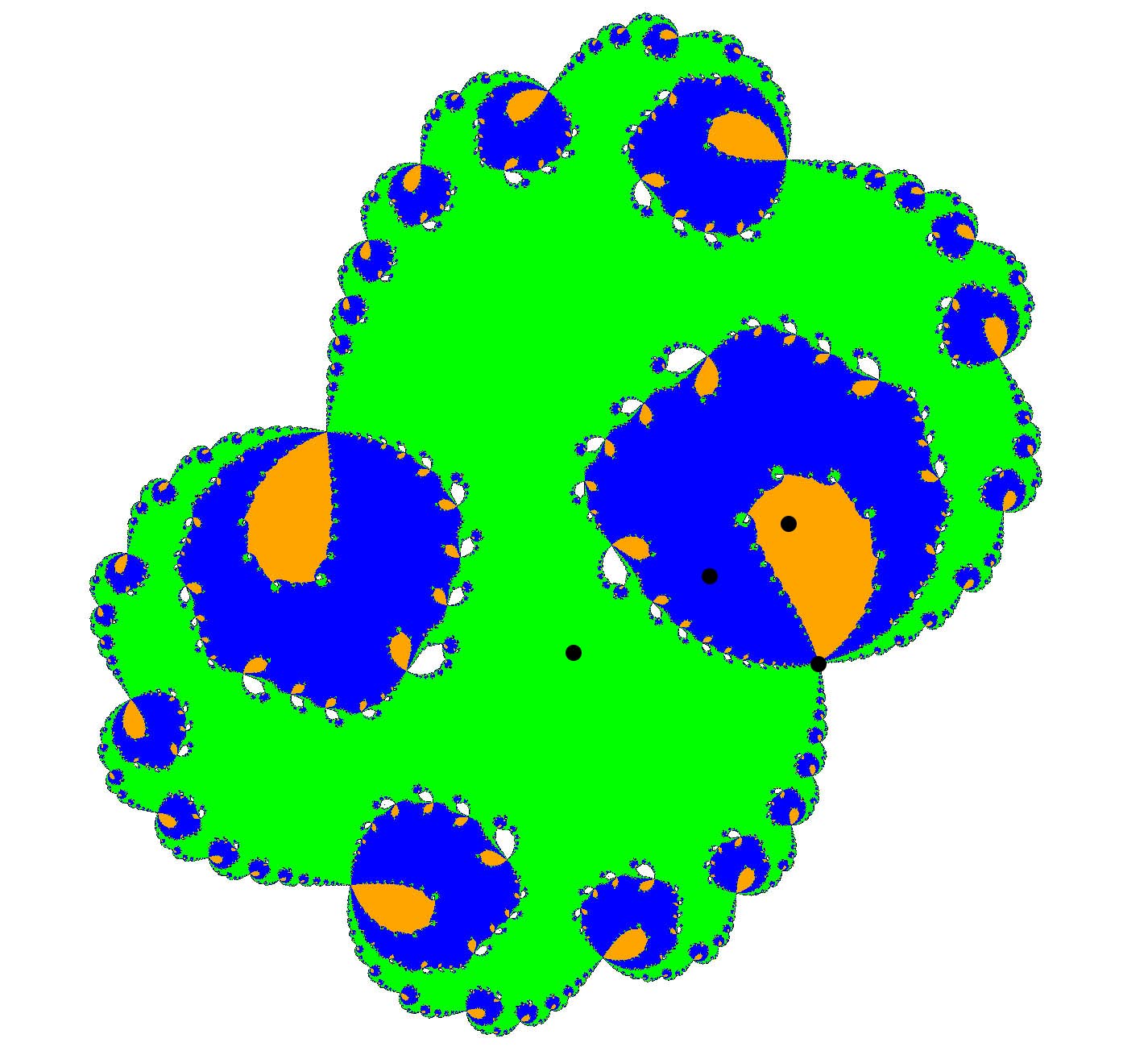}};
	\node at (0.9,-0.55){$a$};
	\node at (1.55,-0.25){$b$};
	\node at (0.05,-1.20){$-1$};
	\node at (2.05,-1.05){$\infty$};
	\end{tikzpicture}
	\caption{The image of  the Julia set $J(f)$ under the action of  the M\"{o}bius map $h(z)=\frac{z}{z-1}$, where $f$ is the Newton map for the polynomial $p(z)=(z^2-1)(z-a)(z-b)$
	with $a=-1.142-2.0477\,i$ and $b=0.1667-3.15485\,i$.%
%
	}
	\label{fig:newton}
\end{figure}
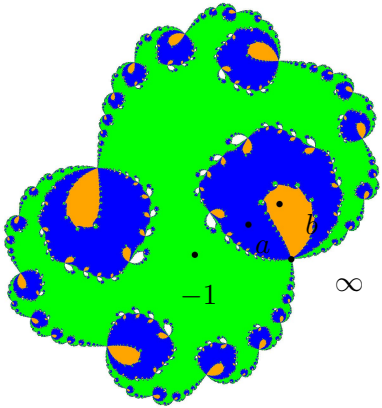

\vspace{6 pt} 

 \noindent \textbf{Organization of the paper.}  The paper is organized as follows: 

In Section \ref{sec_pre}, we present some basic facts for Newton maps.

In Section \ref{sec:fixe_point_thm},  we develop a method to count the number of poles (counting a suitable multiplicity) for Newton maps in certain domains arising from dynamics.
This allows us to construct an invariant graph for Newton maps  by an inductive procedure (in Section \ref{sec_graphs}).



In Section \ref{sec_graphs}, we will construct an invariant graph for Newton maps.
This graph is used to develop the puzzle theory.



In Section \ref{sec:puzzle}, we introduce the Branner-Hubbard-Yoccoz puzzle theory and sketch the idea of the proof, whose 
details are carried out in the forthcoming sections.
The strategy is deeply inspired by the work of Roesch-Yin \cite{RY08}.

%
%
 
 To prove the local connectivity of $\partial B$, for each
 $z\in \partial B$, we define its {\it end} $\ee(z)$ as
 the intersection of infinitely many nested puzzle pieces containing $z$. 
  The main point is to show that 
 $\ee(z)\cap \partial B=\{z\}$. 
 For this purpose, we need treat two cases: the wandering case and renormalizable case.


In Section \ref{wandering}, we will show that each wandering end is a singleton. This is based on the dichotomy: a wandering end $\ee$ either satisfies the {\it bounded degree} property, or 
		its combinatorial limit set  $\omega(\ee)$ contains a {\it persistently recurrent} critical end.
		The treatments for these two cases are different: the former needs to control the number of critical points in long orbits of puzzle pieces, while the latter makes essential use of recent developments   in multi-critical polynomial dynamics, especially the {\it principle nest } construction and its properties \cite{KL1,KL09, KSS07}.

%
%
%


In Section \ref{pre-periodic}, we handle the {\it renormalizable case}.
We will show that if $\ee(z)$ is periodic and non-trivial, then  $\ee(z)\cap \partial B=\{z\}$.
 The main idea 
  is to construct an invariant curve which separates the end $\ee(z)$ from $B$. The construction is natural and  less technical (compare \cite{Ro08}). 
The idea is new and can be applied to study other rational maps.
%
%
%

In Section \ref{proof-main}, we complete the proof of the main theorem. 
 
%

\vspace{6 pt} 

 \noindent \textbf{Notations} 
Throughout the paper, we will use the following notations:

\begin{itemize}
    \item  $\olC$, ${\mb{C}}$, $\mb{D}$ are the Riemann sphere, the complex plane, the unit disk, respectively.
    The boundary of $\mb{D}$ is denoted by $\mb{S}$.
	\item Let $A$ be a set in $\olC$. The closure and the boundary of $A$ are denoted by $\ol{A}$, $\partial A$, respectively. We denote by $\tu{Comp}(A)$  the collection of all connected components of $A$. The cardinality of $A$ is $\#A$.
	\item Two sets in $\olC$ satisfying $A\Subset B$ means that $\ol{A}$ is contained in the interior of $B$.
    \item The Julia set and Fatou set of a rational map $f$ are denoted by $ J(f)$ and $ F(f)$, respectively.
     \end{itemize}

%
%
%
%
%
  
  \vspace{4 pt}
  \noindent \textbf{Acknowledgement} 
The research is supported by NSFC.

\section{Preliminaries}\label{sec_pre}

This section collects some basic facts and introduces some notations for Newton maps.  

Let $p$ be a complex polynomial, factored as
$$p(z)=a(z-a_1)^{n_1}\cdots(z-a_d)^{n_d}$$
where $a\neq 0$ and $a_1,\cdots, a_d\in \mb{C}$ are  distinct roots of $p$, with multiplicities $n_1,\cdots,n_d\geq 1$, respectively. 
In our discussion, we may assume $d\geq 2$.

Its Newton map $f_p$ fixes each root $a_k$ with multiplier 
$$f_p'(a_k)=\frac{p(z)p''(z)}{p'(z)^2}\Big|_{z=a_k}=\frac{n_k-1}{n_k}.$$ 
Therefore, each root $a_k$ of $p$ corresponds to  an   attracting  fixed point of $f_p$ with multiplier $1-1/n_k$.
It follows from the equation 
$$\frac{1}{f_p(z)-z}=-\sum_{k=1}^d\frac{n_k}{z-a_k}$$
that the degree of $f_p$ equals $d$, the number of  distinct roots of $p$.
One may also verifies that $\infty$ is a repelling fixed point of $f_p$ with multiplier 
$$\lambda_\infty=\frac{\sum_{k=1}^d n_k}{\sum_{k=1}^d n_k-1}=\frac{\deg(p)}{\deg(p)-1}.$$

From above discussion, we see that a degree-$d$ Newton map has $d+1$ distinct fixed points with specific multipliers.
On the other hand, a well-known theorem of Head states that the fixed points together with the specific multipliers can determine a
unique Newton map:
\begin{theorem}[Head \cite{He87}] \label{char} A rational map $f$ of degree $d\geq2$
is the Newton map of a polynomial $p$ if and only if
$f$ has $d+1$ distinct fixed points
$$a_1, a_2, \cdots, a_d, \infty,$$
such that for each fixed point $a_k$, 
 the multiplier takes the form 
$$f'(a_k)=1-1/n_k \text{ with } n_k\in \mathbb N,  \ 1\leq k\leq d.$$
 In this case, the polynomial $p$ has the form $a(z-a_1)^{n_1}\cdots(z-a_d)^{n_d}, \ a\neq0$.
\end{theorem}


Now, for the Newton map $f=f_p$ of $p$, let $B(a_k)$ be the root basin of $a_k$, and
$B_k$ be the immediate root basin of $a_k$. Recall that
$$B(a_k)=\big\{z\in \EC; f^n(z)\rightarrow a_k \text{ as } n\rightarrow +\infty\big\}.$$
The attracting basin for all roots is 
$${B}_f=B(a_1)\cup\cdots\cup B(a_d).$$
We say that $f$ is \emph{post-critically finite} in ${B}_f$, if there are only finitely many post-critical points in $B_f$,
or equivalently, each critical point in ${B}_f$ will eventually be iterated to one of $a_k$'s. 

 
 According to Shishikura \cite{Sh09}, the Julia set of a Newton map $f$ is always connected, or equivalently, all Fatou components of $f$ are simply connected (see Figure \ref{fig:newton}).
  By means of quasi-conformal surgery, one can show that $f$ is quasi-conformally conjugate, in a neighborhood of
  of $\mathbb{\widehat{C}}-B_f$,
   to a Newton map $g$ which is post-critically finite in its root basin $B_g$. 
%
%
Since the topology of the Julia set $J(f)$ does not change under this conjugacy,  throughout the paper,
we pose the following

 \begin{assumption}\label{assum} The Newton map $f$ is post-critically finite in $B_f$.
 \end{assumption}
 
Under Assumption \ref{assum}, if the degree $d$ of $f$ is two, then $f$ is affinely conjugate to $z^2$.
  In this case,  the collection ${\rm Comp}(B_f)$ of all components of $B_f$ consists of only two elements. In other situations, ${\rm Comp}(B_f)$ consists of infinitely many elements. 
   
A virtue of Assumption \ref{assum} is that one can give a natural dynamical parameterization of root basins (see \cite{Mi06}):

%

%
%

\begin{lemma}
\label{lem:coordinates} Assume $f$ is post-critically finite in $B_f$, then 
	there exist, so-called B\"ottcher maps, \ $\{\Phi_B\}_{B\in\tu{Comp}({B}_f)}$, such that for each $B\in \tu{Comp}({B}_f)$,
	\begin{itemize}
		\item[(1)] $\Phi_B:B\to \mb{D}$ is a conformal map;
		\item[(2)] $\Phi_{f(B)}\circ f\circ\Phi_B^{-1}(z)=z^{d_{B}},  \forall z\in\mb{D}$, where $d_B={\rm deg}(f|_{B})$. 
	\end{itemize}
\end{lemma}
In general, for each $B\in {\rm Comp}({B}_f)$, the B\"ottcher map $\Phi_B$ is not unique. There are $d_B-1$ choices of $\Phi_B$ when $f(B)=B$, and 
$d_B$ choices of $\Phi_B$ when $f(B)\neq B$ and $\Phi_{f(B)}$ is determined.
Once we fix a choice of B\"ottcher maps $\{\Phi_B\}_{B\in\tu{Comp}({B}_f)}$,  we may define the {\it internal rays}, as follows

For each $B\in {\rm Comp}({B}_f)$,  the point $\Phi_B^{-1}(0)$ is called the \emph{center} of $B$, and the Jordan arc 
$$R_B(\theta):=\Phi_B^{-1}(\{re^{2\pi i\theta}:0<r<1\})$$
 is called the \emph{internal ray} of angle $\theta$ in $B$. According to a well-known landing theorem \cite[Theorem 18.10]{Mi06}, when $\theta$ is rational, the internal ray $R_B(\theta)$ always lands (i.e. the limit
 $\lim_{r\rightarrow 1^-} \Phi_B^{-1}(re^{2\pi i\theta})$ exists). A number $r\in(0,1)$ and two rational angles $\theta_1,\theta_2$
 induce a \emph{sector}:
	$$S_B(\theta_1,\theta_2;r):=\Phi_B^{-1}\big(\{t\,e^{2\pi i\theta}:r<t<1,\theta_1<\theta<\theta_2\}\big),$$
	here $\theta_1<\theta<\theta_2$ means that the angles $\theta_1, \theta, \theta_2$ sit in the circle in the counter clock-wise order.

\section{Counting number of poles}\label{sec:fixe_point_thm}

In this section, we develop a method to count the number of poles (counting suitable multiplicity) for Newton maps $f$ in certain domains (which arise from dynamics).
We will show that in the domains we consider, the number of poles is strictly less than the number of Jordan curves which bound the domain.
This fact allows us to construct an invariant graph for Newton maps  by an inductive procedure (see next section).

\subsection{Counting number of fixed points}

By a \emph{graph} we mean  a connected and compact subset of $\olC$, written as the disjoint union of
finitely many points (called {\it vertices}) and
 finitely many open Jordan arcs (called {\it edges}),   any two of which touch only at vertices. A graph can contain a loop.
 
 Let $G$ be a graph.
For any $z\in G$, let $\nu(G, z)$ be the number of components of $G\setminus\{z\}$.
We call $z$ a \emph{cut point} of $G$ if $\nu(G, z)\geq 2$ ($\Longleftrightarrow G\setminus\{z\}$ is disconnected), a \emph{non-cut point}  of $G$ if $\nu(G, z)=1$ ($\Longleftrightarrow G\setminus\{z\}$ is connected).
Observe that all components in $\olC\setminus G$ are Jordan disks if and only if all 
 $z\in G$ are non-cut points.

In our discussion, by a \emph{Jordan domain} or \emph{Jordan disk}, we means an open subset of $\olC$ whose boundary is a 
Jordan curve.  A \emph{pre-Jordan domain} $W$ means a connected component of $g^{-1}(D)$, where $D$ is a Jordan disk and $g$ is a rational map. Here the boundary $\partial D$ may or may not contain critical values of $g$. 
If $\partial D$ contains no critical value of $g$, then each boundary component of $W$ is a Jordan curve; if 
 $\partial D$ contains at least one critical value of $g$,  then each component of $\partial W$ can be written as a 
  union of finitely many Jordan curves, touching  at critical points.
In either case, for any component $\gamma$ of $\partial W$, the map $g|_{\gamma}: \gamma\rightarrow 
\partial D$ has a well-defined degree, denoted by ${\rm deg}(g|_\gamma)$.

One may observe that {\it for any pre-Jordan domain $W$, any  component $V$ of  $h^{-1}(W)$ 
is again a pre-Jordan domain, here $h$ is a rational map}. To see this, note that $W$ is a component of $g^{-1}(D)$ ($D$ is a Jordan disk), and 
$V$ is a component of $(g\circ h)^{-1}(D)$, where $g\circ h$ is a rational map.

  \begin{figure}[h]
\begin{center}
\includegraphics[height=5.5cm]{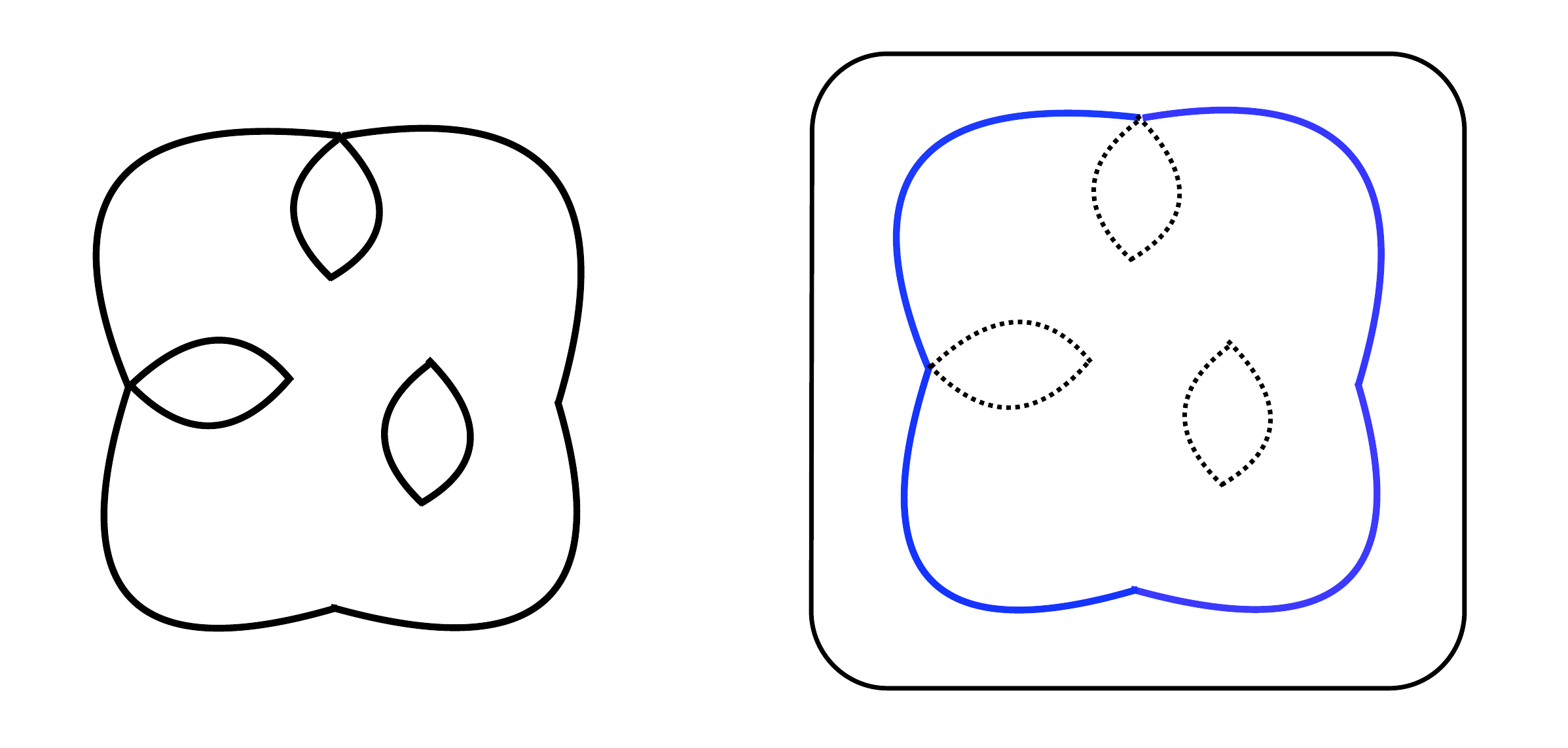} 
\put(-267,55){$U$} \put(-107,55){$\wh{U}_D$} \put(-34,22){$D$}
 \caption{An example of pre-Jordan domain $U$ (left), and its filled closure $\wh{U}_D$ with respect to $D$ (right).
Clearly $\wh{U}_D$ is a Jordan disk bounded by a blue curve.
 }
 \label{filled-closure}
\end{center}
\end{figure}


	 Let $U$ be a pre-Jordan domain, and $D$ be a Jordan disk  in $\olC$ such that $U\subseteq D$.
	  The \emph{filled closure} of $U$ with respect to $D$, denoted by $\wh{U}_D$, is 
	$$\wh{U}_D=\ol{U}\cup\cup_V\ol{V},$$
	where $V$ ranges over all components of $\olC\setminus\ol{U}$ with $V\subseteq D$. See Figure \ref{filled-closure}.

It's easy to verify the following facts:
\begin{itemize}
    \item    $\ol{U}\subseteq\wh{U}_D\subseteq \ol{D}$.
    \item   The filled closure $\wh{U}_D$ is always a Jordan disk\footnote{To see this, if  $\wh{U}_D$ is not a Jordan disk, then $\partial\wh{U}_D$
	     is not a Jordan curve, and can be written as 
	     a union of finitely many Jordan curves, say $\alpha_1, \cdots, \alpha_k$ with $k\geq2$, such that the intersection of any two curves is a  finite set. These curves enclose mutually disjoint Jordan disks, $D_1, \cdots, D_k$, in $\olC-\overline{U}$. 
	     Note that $\partial D\subset \overline{D}_j$ for some $j$, however  this will contradict the definition of
	     $\wh{U}_D$. 
}.
	\item   $\wh{U}_D=\ol{U}$ if and only if $U$ is a Jordan disk.
	     \end{itemize}
	     
%
%
%
%


The following fixed point theorem appears in \cite[Theorem 4.8]{RS}. 


\begin{lemma}
\label{lem:fixed_points}
	Let $D\subseteq \olC$ be a Jordan disk, and $g$ be a rational map of degree at least two.	
Suppose that $g^{-1}(D)$ has a component $U\subseteq D$. If $\partial D\cap\partial U$ contains a fixed point $q$, we further require that $q$
 is repelling and 
$g(N_q\cap\partial U)\supseteq N_{q}\cap\partial U$ in a neighborhood $N_{q}$ of $q$. Then
	$$\#\tu{Fix}(g|_{\ol{U}})=\tu{deg}(g|_{\partial U}).$$
\end{lemma}


Here the number of fixed points is counted with multiplicity. Recall that the
 \emph{multiplicity} of a fixed point $z_0\in\mathbb C$ 
is defined to be the unique integer $m\geq 1$ such that, near $z_0$, 
$$g(z)-z=a_m(z-{z}_0)^m+a_{m+1}(z-{z}_0)^{m+1}+\cdots$$ with $a_m\neq 0$. 
The number $m$ is independent of the choice of coordinates.

 As a consequence of Lemma \ref{lem:fixed_points}, one has

\begin{coro}\label{cor:counting_number}
Let $D$ be a Jordan disk in $\olC$, and $g$ be a rational map of degree at least two.
Suppose that $g^{-1}(D)$ has a component $U\subseteq D$. If $\partial D\cap\partial U$ contains a fixed point $q$, we further require that $q$
 is repelling and 
$g(N_q\cap\partial U)\supseteq N_{q}\cap\partial U$ in a neighborhood $N_{q}$ of $q$.	
 Then

	$$\#\tu{Fix}(g|_{\wh{U}_D})=\sum_{V}\tu{deg}(g|_{\partial V}),$$
	where $V$ runs over all components of $g^{-1}(D)$ such that $V\subseteq \wh{U}_D$. 
	
	In particular, if $\wh{U}_D$ contains only one fixed point (counting multiplicity), then $U$ is a Jordan disk ($\Longleftrightarrow\ol{U}=\wh{U}_D$), and $g:U\to D$ is a homeomorphism.
\end{coro}
\begin{proof}
	Let $V_1=U, V_2,\cdots,V_n$ (resp. $V'_1,\cdots,V_m'$) be all the components of $g^{-1}(D)$ (resp. $g^{-1}(\olC\setminus \overline{D})$) in the filled closure $\wh{U}_D$. Then by definition,
	$$\wh{U}_D=\ol{V}_1\cup\cdots\cup\ol{V}_n\cup V_1'\cdots\cup V_m'.$$ These $V_k'$'s are clearly disjoint from fixed points. For distinct $\ol{V}_i,\ol{V}_j$, 
	the intersection $\ol{V}_i\cap \ol{V}_j$ is a finite set because it is contained in the critical set of $g$.
	Further, if $\ol{V}_i\cap \ol{V}_j$ contains a fixed point, say $q$, of $g$, then $q$ is 
	a critical point, hence a superattracting fixed point.
	Moreover, we have $q\in \partial V_i\cap\partial D$, and this implies that $q$ is also 
	on $\partial U\cap\partial D$. However, this contradicts our assumption.
	Therefore, there is no fixed point on $\partial V_i\cap \partial V_j$.
	
	  It follows from Lemma \ref{lem:fixed_points} that
	$$\#\tu{Fix}(g|_{\wh{U}_D})=\sum_{1\leq k\leq n}\#\tu{Fix}(g|_{\ol{V}_k})=\sum_{1\leq k\leq n}\tu{deg}(g|_{\partial V_k}).$$
	
This equality implies that, if $\#\tu{Fix}(g|_{\wh{U}_D})=1$, then 
	$U$ is the unique component of $g^{-1}(D)$ in $D$ and ${\rm deg}(g|_U)=1$, proving the statement.
\end{proof}

\begin{remark}[What is multiplicity?]\label{d-multiplicity}
 In Corollary \ref{cor:counting_number}, the sum 
 $$\sum_{1\leq k\leq n}\tu{deg}(g|_{\partial V_k})$$
is the cardinality $\#(g^{-1}(q)\cap \wh{U}_D)$, for (any) $q\in \partial{D}$, counting multiplicity.

The multiplicity $m(p, \wh{U}_D)$ of $p\in g^{-1}(q)\cap \wh{U}_D$ is an integer between $1$ and the local degree $\tu{deg}(g,p)$.
A natural definition is  as follows (see Figure \ref{multiplicity}):

Let $\mathcal{F}(p, \wh{U}_D)$ consist of those $V\in \{V_1,\cdots, V_n\}$ so that $p\in \partial V$. For each $V\in \mathcal{F}(p, \wh{U}_D)$, let $ T(p,\partial V)\geq 1$ be the number of meeting times when one moves along the boundary 
$\partial V$ once. The multiplicity $m(p, \wh{U}_D)$ is defined by
$$m(p, \wh{U}_D)=\sum_{V\in \mathcal{F}(p, \wh{U}_D)} T(p,\partial V).$$

  \begin{figure}[h]
\begin{center}
\includegraphics[height=5.5cm]{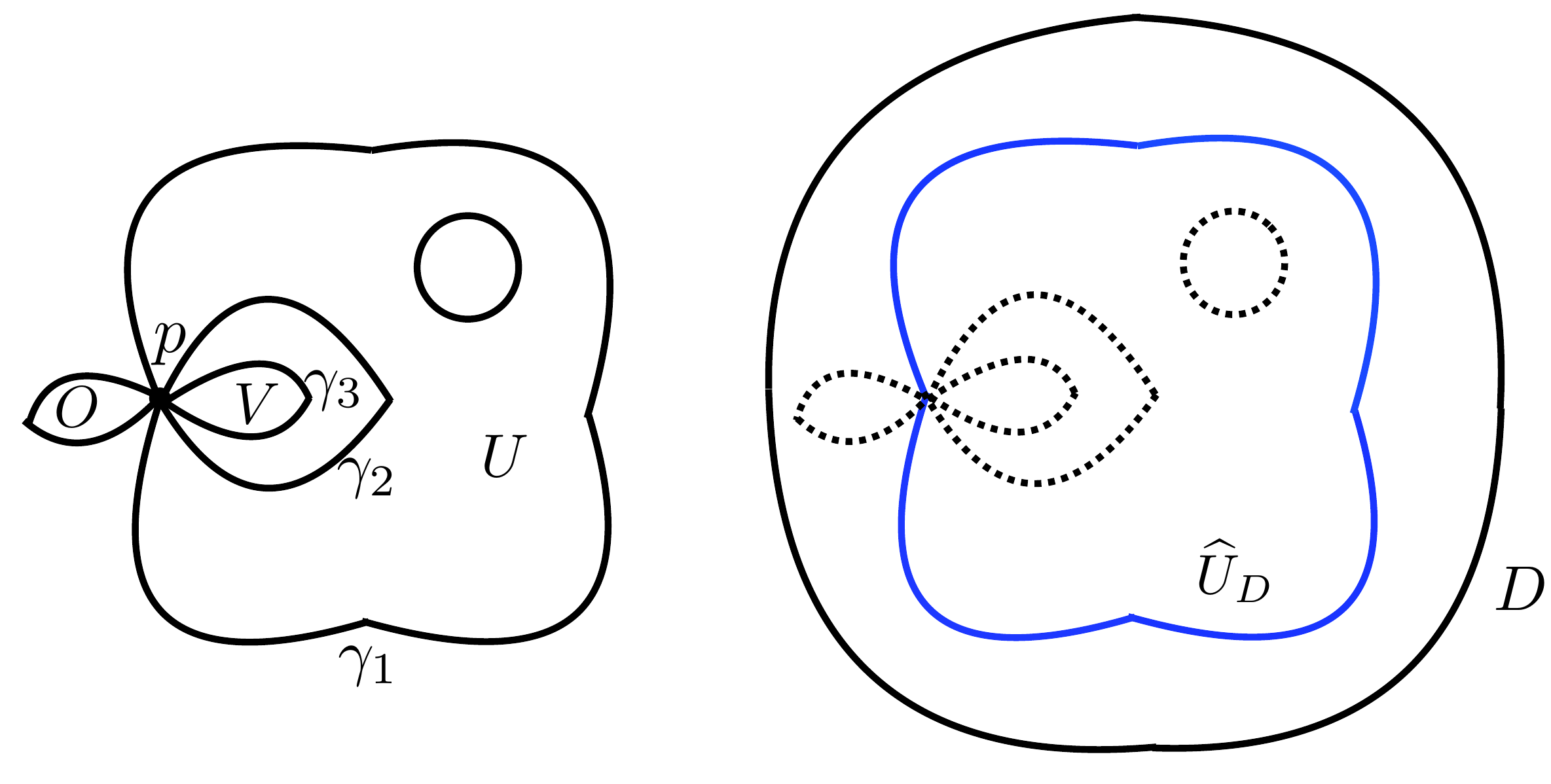} 
 \caption{In this example, there are three components of $g^{-1}(D)$ contained in $D$. They are  $U, V$ and $O$. Their boundaries touch at $p$.  The filled closure $\wh{U}_D$ contains $U,V$.
 Clearly, $\partial U=\gamma_1\cup \gamma_2$, $\partial V=\gamma_3$.  Moreover, $\partial U$ meets $p$ twice and $\partial V$ meets $p$ once, hence $m(p, \wh{U}_D)=2+1=3$. Note also $\alpha(p)=\gamma_1\cup\gamma_2\cup \gamma_3$ and $\nu(\alpha(p), p)=3=m(p, \wh{U}_D)<\tu{deg}(g, p)=4$.
 }
 \label{multiplicity}
\end{center}
\end{figure}

Let $\alpha(p)$ be the component of $g^{-1}(\partial D)\cap \wh{U}_D$ 
containing $p$.
 One may show
 $$m(p, \wh{U}_D)=\nu(\alpha(p), p).$$

Hence,  we have the identity 
$$\sum_{p\in g^{-1}(q)\cap \wh{U}_D}\nu(\alpha(p), p)=\sum_{p\in g^{-1}(q)\cap \wh{U}_D}m(p, \wh{U}_D)
 =\sum_{1\leq k\leq n}\tu{deg}(g|_{\partial V_k}).$$
%
%
\end{remark}


\subsection{The inverse image of a Jordan curve} Let $\gamma$ be a Jordan curve in 
$\olC$. Its complement $\olC-\gamma$ has two components, one is called the {\it interior part} of $\gamma$, denoted by $\tu{Int}(\gamma)$, while the other is called the {\it exterior part} of $\gamma$, denoted by $\tu{Ext}(\gamma)$. The designation of interior or exterior  part is arbitrary at this moment.

Let $g$ be a rational map.
Suppose there is a component $U$ of $g^{-1}(\tu{Ext}(\gamma))$ contained in $\tu{Ext}(\gamma)$.
Let $\wh{U}$ be the filled closure of $U$ with respect to $\tu{Ext}(\gamma)$. The {\it inverse image}  $\gamma^{-1}$ of $\gamma$ with respect to $g$,  is the Jordan curve
$$\gamma^{-1}=\partial \wh{U}.$$ 
One may verify that $g(\gamma^{-1})=\gamma$ and $\gamma^{-1}$ is  contained in (possibly equal to) a 
    connected component, say $\alpha$, of $g^{-1}(\gamma)$. 
    Moreover,
    the  degrees of $g|_{\gamma^{-1}}$ and $g|_{\alpha}$ are well-defined, and satisfy
    $$\tu{deg}(g|_{\gamma^{-1}})\leq \tu{deg}(g|_{\alpha})\leq \tu{deg}(g).$$
 
 The equality $\tu{deg}(g|_{\gamma^{-1}})=\tu{deg}(g|_{\alpha})$ holds if and only if $\gamma^{-1}=\alpha$.
 Applying the same operation to the new curve $\gamma^{-1}$, one get $\gamma^{-2}=(\gamma^{-1})^{-1}$. 
 Precisely, 
    let $V$ be a  component $g^{-1}(\tu{Ext}(\gamma^{-1}))$ contained in $\tu{Ext}(\gamma^{-1})$, and $\wh{V}$ the filled closure of $V$ with respect to $\tu{Ext}(\gamma^{-1})$,
    we set $\gamma^{-2}=(\gamma^{-1})^{-1}=\partial \wh{V}$.
     Similarly, for any integer $n\geq 1$, the curve $\gamma^{-n}$  can be defined inductively:  
     $$\gamma^{-n}=(\gamma^{-n+1})^{-1},$$
with the  property $\tu{Ext}(\gamma^{-n})\subseteq \tu{Ext}(\gamma^{-n+1})\subseteq \cdots\subseteq
\tu{Ext}(\gamma^{-1})$.
    
    We remark that the only ambiguity in the definition of $\gamma^{-1}$ occurs when we are choosing the component  $U$ of $g^{-1}(\tu{Ext}(\gamma))$. There might be several components of $g^{-1}(\tu{Ext}(\gamma))$ contained in $\tu{Ext}(\gamma)$, and $U$ is not unique.
   Even though, the readers needn't worry about that, because in the following discussion, we actually choose some 
    specific component  $U$ of $g^{-1}(\tu{Ext}(\gamma))$, and there will be no ambiguity then.


\subsection{Counting number of poles}

We say that the Jordan curves $\gamma_1,\cdots,\gamma_n$ with $n\geq 2$ in $\olC$ are {\it independent}, if
\begin{itemize}
    \item    $\gamma_i\cap \gamma_j$ is a finite set (possibly empty), for $i\neq j$.
    \item  For any $k$,  there is a component $\tu{Int}(\gamma_k)$ of $\olC\setminus\gamma_k$, designated as  the \emph{interior part} of $\gamma_k$,  such that the Jordan disks $\tu{Int}(\gamma_1),\cdots,\tu{Int}(\gamma_n)$ are mutually disjoint (see Figure \ref{f3-b}). 
	     \end{itemize}

  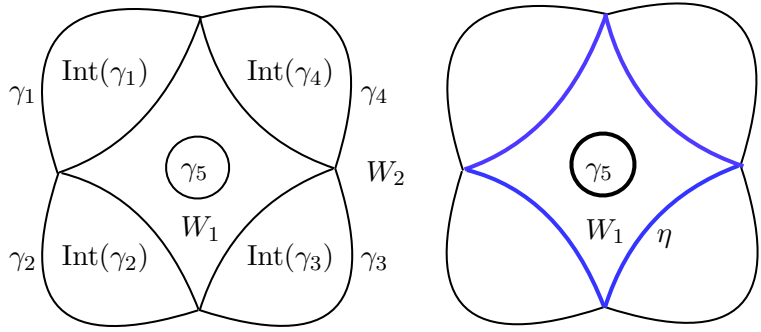
\begin{figure}[h]
\begin{center}{
  \includegraphics[height=5cm]{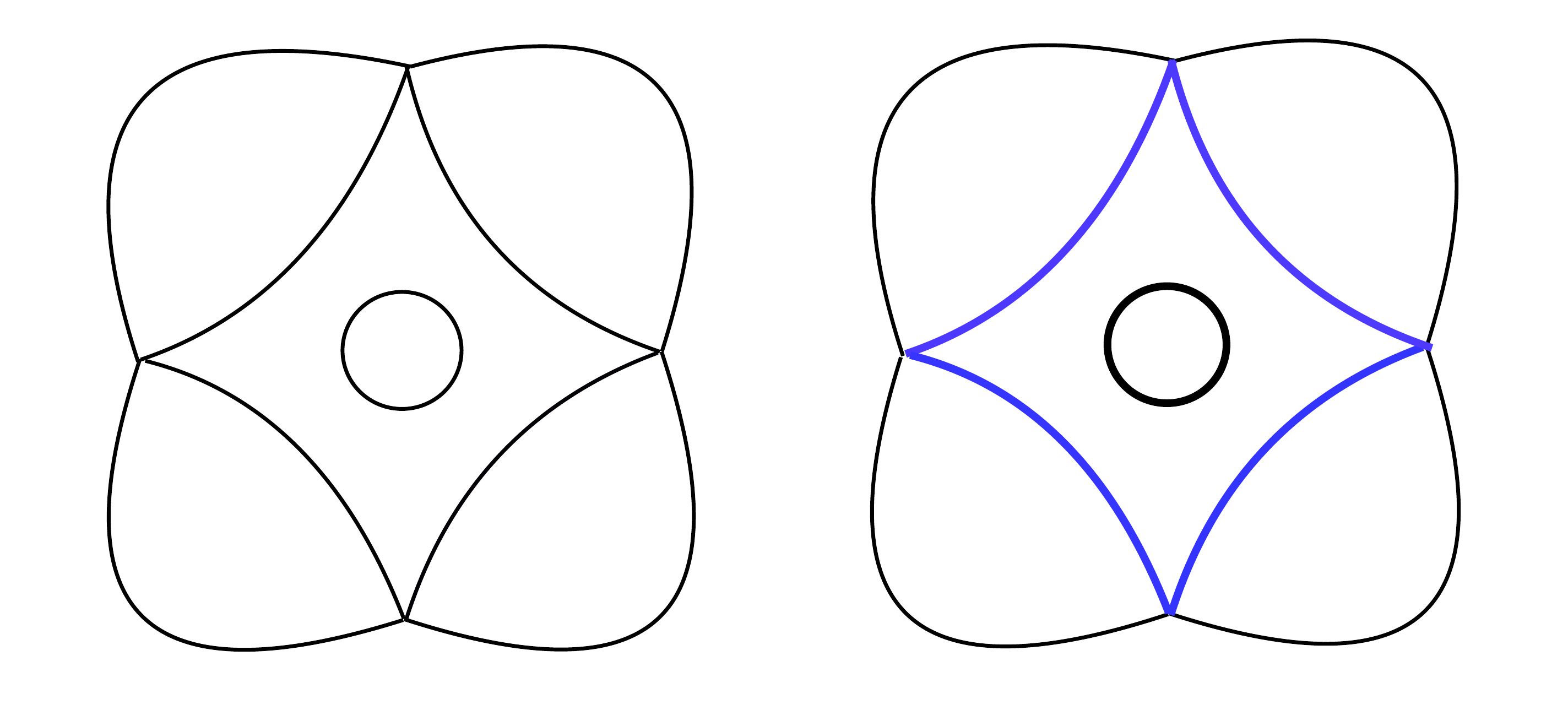} 
\put(-305,35){$\gamma_2$}\put(-285,35){$\tu{Int}(\gamma_2)$}
 \put(-305,98){$\gamma_1$} \put(-285,105){$\tu{Int}(\gamma_1)$} 
 \put(-172,98){$\gamma_4$} \put(-215,105){$\tu{Int}(\gamma_4)$} 
 \put(-172,35){$\gamma_3$} \put(-215,35){$\tu{Int}(\gamma_3)$} 
 \put(-240,69){$\gamma_5$} 
  \put(-170,67){$W_2$} \put(-240,46){$W_1$} 
 \put(-87,69){$\gamma_5$}\put(-87,45){$W_1$} \put(-60,44){$\eta$}}
 \caption{An example of independent Jordan curves $\gamma_1,\cdots, \gamma_5$ (left), 
 here $A(\gamma_1,\cdots,\gamma_5)=W_1\cup W_2$, and 
 $W_1$ can be written as $A(\gamma_5, \eta)$, where $\gamma_5,\eta$ are independent (right).
 }
 \label{f3-b}
\end{center}\end{figure}

	   Note that when we are saying that the curves $\gamma_1,\cdots,\gamma_n$ are independent, their interiors $\tu{Int}(\gamma_k)$'s are determined. The other component of $\olC\setminus\gamma_k$ is the \emph{exterior part} of $\gamma_k$, denoted by $\tu{Ext}(\gamma_k)$. Let 
	   $$A(\gamma_1,\cdots,\gamma_n)=\bigcap_{1\leq k\leq n}\tu{Ext}(\gamma_k)=\olC-\bigcup_{1\leq k\leq n}\ol{\tu{Int}(\gamma_k)}.$$	

Clearly $A(\gamma_1,\cdots,\gamma_n)$ is an open set and has finitely many connected components. It is worth observing that for any component $W$ of $A(\gamma_1,\cdots,\gamma_n)$, which is not a Jordan disk, 
there are independent Jordan curves $\eta_1,\cdots,\eta_m$ for some $m\geq 2$ such that $W=A(\eta_1,\cdots,\eta_m)$, see Figure \ref{f3-b}.

\begin{prop}\label{prop:key_prop}
	Let $g$ be a rational map with $\infty$ a repelling fixed point. Let $\gamma_1, \cdots,\gamma_n$ be independent Jordan curves in $\olC$ satisfying  that
	
(a).	$\gamma_i\cap \gamma_j=\{\infty\}$ for any $i\neq j$;

(b).   All fixed points of $g$ in $\mathbb{C}$ are contained in $\tu{Int}(\gamma_1)\cup\cdots\cup \tu{Int}(\gamma_n)$;
		
(c).  In a neighborhood $N(\infty)$ of $\infty$, one has 
		$$N(\infty)\cap \gamma_k\subseteq g(N(\infty)\cap\gamma_k), \ \forall \ 1\leq k\leq n;$$ 
	
(d). The unbounded component of $g^{-1}(\gamma_k)$ is contained in $\ol{\tu{Ext}}(\gamma_k)$.

	 Then  the unbounded component  $U_k$ of $g^{-1}(\tu{Ext}(\gamma_k))$ satisfies 
	$$U_k\subseteq \tu{Ext}(\gamma_k).$$
	
	Further, let $\wh{U}_k$ be the filled closure of $U_k$ with respect to $\tu{Ext}(\gamma_k)$ and let 
	$\gamma_k^{-1}=\partial\wh{U}_k$. Then
	
	1. $\gamma_1^{-1},\cdots,\gamma_n^{-1}$ are independent Jordan curves with $\ol{\tu{Ext}}(\gamma_k^{-1})=\wh{U}_k$;
	
	2. In each $\wh{U}_k$, 
		the number of poles (counting multiplicity, see Remark \ref{d-multiplicity}) equals that of fixed points;
	
	3. $g^{-1}(\tu{Ext}(\gamma_k))$ is disjoint from $\tu{Ext}(\gamma_k)\setminus \tu{Ext}(\gamma_k^{-1})$;
	
	4. The unbounded component of $g^{-1}(\ol{A(\gamma_1,\cdots,\gamma_n)})$ is contained in 
		$$\ol{A(\gamma_1^{-1},\cdots,\gamma_n^{-1})}.$$
\end{prop}

  \begin{figure}[h]
\begin{center}
\includegraphics[height=6cm]{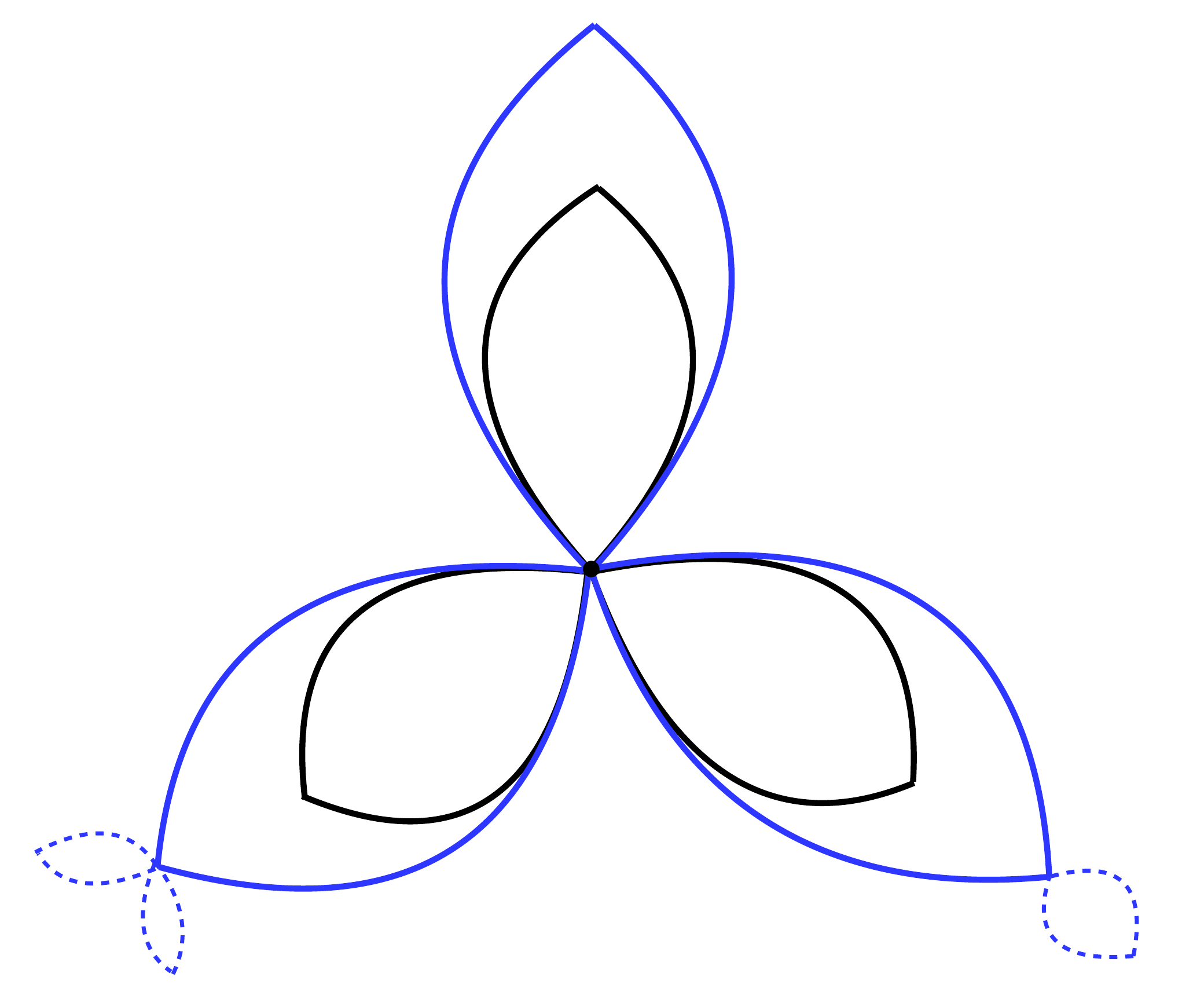} 
\put(-118,66){$\infty$}
 \put(-109,126){$\gamma_1$} \put(-120,105){$\tu{Int}(\gamma_1)$} \put(-76,125){$\gamma_1^{-1}$} 
  \put(-185,60){$\gamma_2^{-1}$}  \put(-144,60){$\gamma_2$} \put(-147,45){$\tu{Int}(\gamma_2)$} 
  \put(-87,66){$\gamma_3$} \put(-40,66){$\gamma_3^{-1}$}\put(-85,45){$\tu{Int}(\gamma_3)$}
 \caption{Inverse images of independent curves.}
\end{center}\label{f5}
\end{figure}
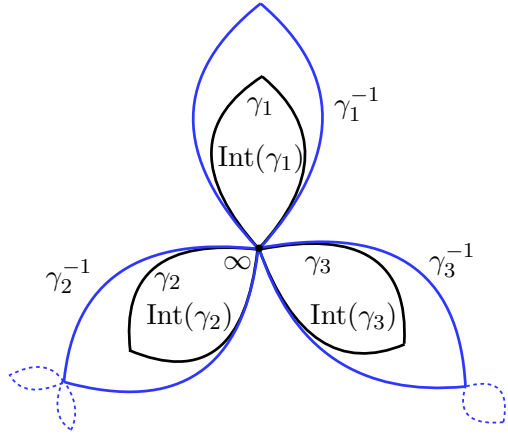

\begin{proof}
	Let $\alpha_k$ be the unbounded component of $g^{-1}(\gamma_k)$.
	The set $\alpha_k$ is a union of finitely many Jordan curves, touching at finitely many points.
	Clearly, $\olC\setminus\alpha_k$ has finitely many components, and 
	each component of $g^{-1}(\tu{Ext}(\gamma_k))$ (resp. $g^{-1}(\tu{Int}(\gamma_k))$) is contained in one of them. Write 
	$$\tu{Comp}(\olC\setminus\alpha_k)=\big\{C_{k,\infty},C_{k,1},\cdots,C_{k,l},C'_{k,\infty},C'_{k,1},\cdots,C_{k,l'}'\big\},$$
	where the notations are labeled so that
	\begin{itemize}
		\item $C_{k,\infty},C_{k,\infty'}$ are the only two unbounded components;
		\item Each $C\in\{C_{k,\infty}, C_{k,1},\cdots, C_{k,l}\}$ (resp. $\{C'_{k,\infty}, C'_{k,1},\cdots, C'_{k,l'}\}$) contains a component $V$ of $g^{-1}(\tu{Ext}(\gamma_k))$ (resp. $g^{-1}(\tu{Int}(\gamma_k))$) such that $\partial C\subseteq \partial V$.
	\end{itemize}
	
	The unbounded component $U_k$ of $g^{-1}(\tu{Ext}(\gamma_k))$ is contained in $C_{k,\infty}$. By condition (d), either $C_{k,\infty}\subseteq \tu{Ext}(\gamma_k)$ or $\tu{Int}(\gamma_k)
\subseteq C_{k,\infty}$.  The latter cannot happen,
because locally near $\infty$, $g$ behaves like
$N(\infty)\cap \gamma_k\subseteq g(N(\infty)\cap\gamma_k)$,
  and globally, $g$ is orientation preserving. Thus $U_k\subseteq C_{k,\infty}\subseteq \tu{Ext}(\gamma_k)$.
	
	We will prove the properties (1)-(4), based on the following
	$$\textit{Claim }: \ C_{k,1}\cup\cdots\cup C_{k,l}\subseteq\wh{U}_k.$$
	
	In fact, if the claim is not true, we have $C_{k,i}\subseteq\olC\setminus \wh{U}_k$ for some $i$. By condition (d), the filled closure $\wh{C}_{k,i}$ of $C_{k,i}$ with respect to $\olC\setminus\wh{U}_k$ is disjoint from $\tu{Int}(\gamma_k)$, and $\tu{Int}(\gamma_k)\subseteq C_{k,\infty}'$. 
  Let $D=\tu{Ext}(\gamma_k)$ and $U$ be a component of $g^{-1}(\tu{Ext}(\gamma_k))$ contained in $C_{k,i}$. 
  Clearly $U\subseteq D$.
  Applying Corollary \ref{cor:counting_number} to the pair $(D, U)$,  we have that $\wh{C}_{k,i}$ contains at least one fixed point of $g$. This contradicts condition (b), completing the proof of the claim.
 

1. The following observation 
	$$\gamma_i^{-1}\cap\gamma_j^{-1}\subseteq g^{-1}(\gamma_i)\cap g^{-1}(\gamma_j)\subseteq g^{-1}(\gamma_i\cap\gamma_j)=g^{-1}(\infty),  \ i\neq j$$
	implies that $\gamma_i^{-1}\cap\gamma_j^{-1}$ is a finite set, as it consists of finitely many poles.
	
	Note that $\gamma_i\subseteq \ol{\tu{Ext}}(\gamma_k)$ for $i\neq k$, the unbounded component $\alpha_i$ of $g^{-1}(\gamma_i)$ is contained in the unbounded component of $g^{-1}(\ol{\tu{Ext}}(\gamma_k))$.
Therefore,  $\gamma_i^{-1}\subseteq \alpha_i\subseteq \wh{U}_k$, and there are mutually disjoint interior parts ${\tu{Int}}(\gamma_k^{-1})$'s of the curves $\gamma_k^{-1}$'s.
This verifies that the curves $\gamma_k^{-1}$'s are independent.


2. It is an immediate consequence of Corollary \ref{cor:counting_number}.
	
3. It follows from the  claim above.


4. Since $$\ol{A(\gamma_1,\cdots,\gamma_n)}\subseteq \ol{\tu{Ext}}(\gamma_i)\tu{\,\, for\,\, }1\leq i\leq n,$$ the unbounded component of $g^{-1}(\ol{A(\gamma_1,\cdots,\gamma_n)})$, denoted by $E$, is contained in that of $g^{-1}(\ol{\tu{Ext}}(\gamma_i))$, and therefore  $E\subseteq\wh{U}_i$ by the claim above. Thus $E\subseteq \wh{U}_1\cap\cdots\cap \wh{U}_n=\ol{A(\gamma_1^{-1},\cdots,\gamma_n^{-1})}$. 
\end{proof}


\begin{prop}
\label{lem:common_poles}
Let $g$ be a rational map with $\infty$ a repelling fixed point. Let $\gamma_1, \cdots,\gamma_n$ be 
independent Jordan curves in $\olC$ such that

1. $\infty\in \gamma_1\cap\cdots\cap \gamma_n$; 

2.  All fixed points of $g$ in $\mathbb{C}$ are contained in $\tu{Int}(\gamma_1)\cup\cdots\cup \tu{Int}(\gamma_n)$;

3.  In each $\ol{\tu{Ext}}(\gamma_k)\setminus\{\infty\}$, the number of poles equals that of fixed points.
	
Then the number of poles in $\ol{A(\gamma_1,\cdots,\gamma_n)}$  is $n-1$, strictly less than $n$.
\end{prop}
Here, the number of poles is counted with multiplicity (see Remark \ref{d-multiplicity}). 

\begin{proof}
	Let $a_k$ (resp. $b_k$) be the number of poles (resp.  fixed points) in  $\ol{\tu{Ext}}(\gamma_k)\setminus\{\infty\}$. Let $\wt{a}_k$ (resp. $\wt{b}_k$) be the number of poles (resp.  fixed points) in $\tu{Int}(\gamma_k)=\olC-\ol{\tu{Ext}}(\gamma_k)$. Let $a$ be the number of poles in $\ol{H}\setminus\{\infty\}$, where $H=A(\gamma_1,\cdots,\gamma_n)$.  All these numbers are counted with multiplicities.
	
	The independent curves $\gamma_1,\cdots,\gamma_n$ decompose  $\mb C$ into several parts.
	These parts satisfy the following relations: 
	\begin{itemize}
		\item $\ol{\tu{Ext}}(\gamma_k)\setminus\{\infty\} =(\ol{H}\setminus\{\infty\})\cup\cup_{i\neq k}\tu{Int}(\gamma_i)$;
		\item $\mb{C}=(\ol{H}\setminus\{\infty\})\cup\cup_{1\leq i\leq n}\tu{Int}(\gamma_1)$,
	\end{itemize}
	By counting the number of poles, we have the following identity,
	\bess \sum_{1\leq k\leq n} a_k&=&\sum_{1\leq k\leq n}\Big(a+\sum_{i\neq k}\wt{a}_i\Big)=an+(n-1)\sum_{1\leq i\leq n}\wt{a}_i\\
	&=&a+(n-1)\Big(a+\sum_{1\leq i\leq n}\wt{a}_i\Big)=a+(n-1)(d-1),
\eess	
	where $d$ is the degree of $g$.		
	Note that $\ol{H}\setminus\{\infty\}$ is disjoint from the fixed points of $g$. 
	By counting the number of fixed points in $\mathbb C$, we have
	$$\sum_{1\leq k\leq n} b_k=\sum_{1\leq k\leq n}\Big(0+\sum_{i\neq k}\wt{b}_i\Big)=(n-1)\sum_{1\leq i\leq n}\wt{b}_i=(n-1)d.$$
%

By assumption,  one has $a_k=b_k$ for all $k$, implying that $\sum a_k=\sum b_k$. Therefore we have $a=n-1$. 
The proof is completed.
\end{proof}

\section{Invariant graph}\label{sec_graphs}

Let $f$ be a Newton map of degree $d\geq 3$, post-critically finite on $ B_f$.
The aim of this section is to prove the existence of invariant graph for $f$.
Here, a graph $G$ is said {\it invariant} for $f$ if it satisfies 
$$f(G)\subseteq G, \text{ and }  f^{-1}(G) \text{ is connected}.$$

In fact, the existence of invariant graph is first proven by Drach, Mikulich,  R\"{u}ckert  and  Schleicher \cite{MRS15, DMRS}.
Our work is distinguished from theirs in two aspects.
First, our idea of proof is essentially different from theirs: our proof is constructive while the proof in \cite{DMRS} is more conceptual (they use a proof by contradiction). Secondly, our graph is different from theirs: our graph has very good properties
(each point in the graph except some strictly pre-periodic Fatou centers are non-cut points)
 and is well adapted to construct puzzles, while the one in \cite{DMRS} is abstract and does not have such properties.
That is the reason why we develop a different proof and construct a different graph.

\subsection{Channel graph} \label{c-d}
For any immediate root basin $B$ of $f$, there are
exactly $d_B-1$ fixed internal rays in $B$:
 $$R_{B}({j}/{(d_B-1)}), 0\leq j\leq d_B-2, \text{ where } d_B=\tu{deg}(f|_{B}).$$
 Each of these fixed internal rays must land at a fixed point on $\partial B$.
 Since $\infty$ is the unique fixed point of $f$ on its Julia set, all fixed internal rays land at the
 common point $\infty$.

 The \emph{channel graph} of $f$, denoted by $\Delta_0$, is defined by 
		$$\Delta_0=\bigcup_{B}\bigcup_{j=1}^{d_B-1}\overline{ R_{B}\big({j}/(d_B-1)\big)},$$
		where $B$ ranges over all immediate root basins in $\{B_1,\cdots, B_d\}$. 		
	 Clearly $f(\Delta_0)=\Delta_0$.
	Figure \ref{channel-graph} 
illustrates all possible channel graphs when $d=4$. Some graphs which look like channel graphs but in fact  are fake ones  are given in Figure \ref{fake-channel}.

  \begin{figure}[h]
\begin{center}
\includegraphics[height=4cm]{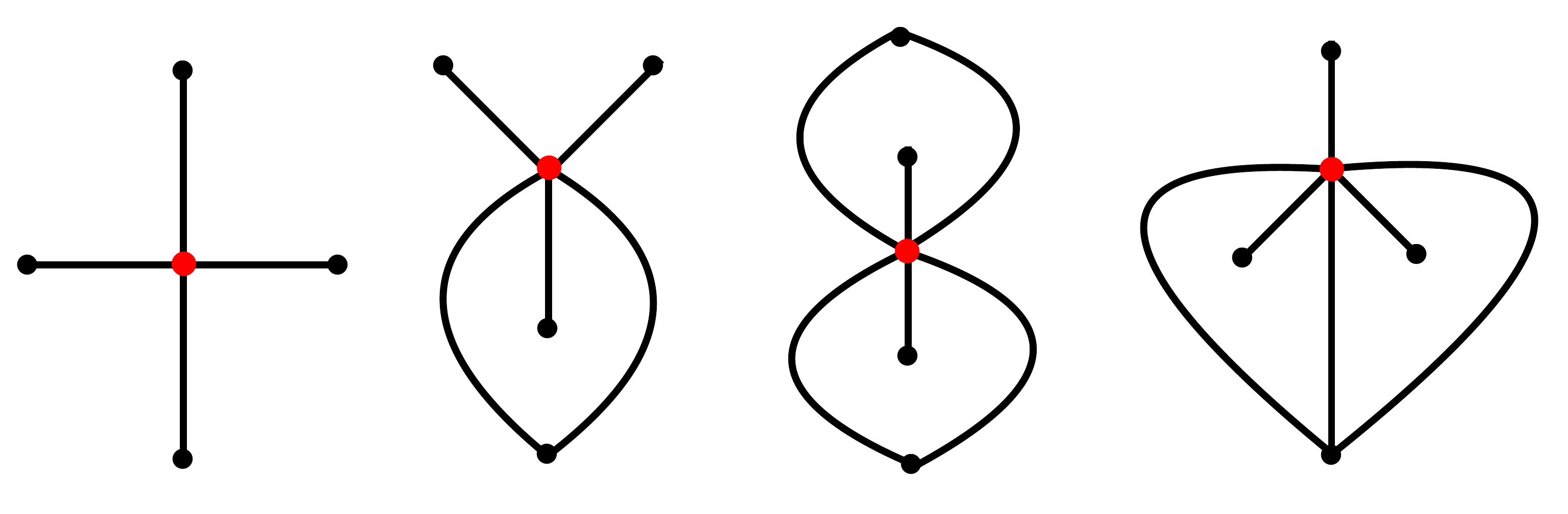} 
 \caption{All possible channel graphs when $d=4$.
 The red dot is $\infty$ and the black dots are the centers of immediate root basins.}
 \label{channel-graph}
\end{center}
\end{figure}

 \begin{figure}[h]
\begin{center}
\includegraphics[height=4cm]{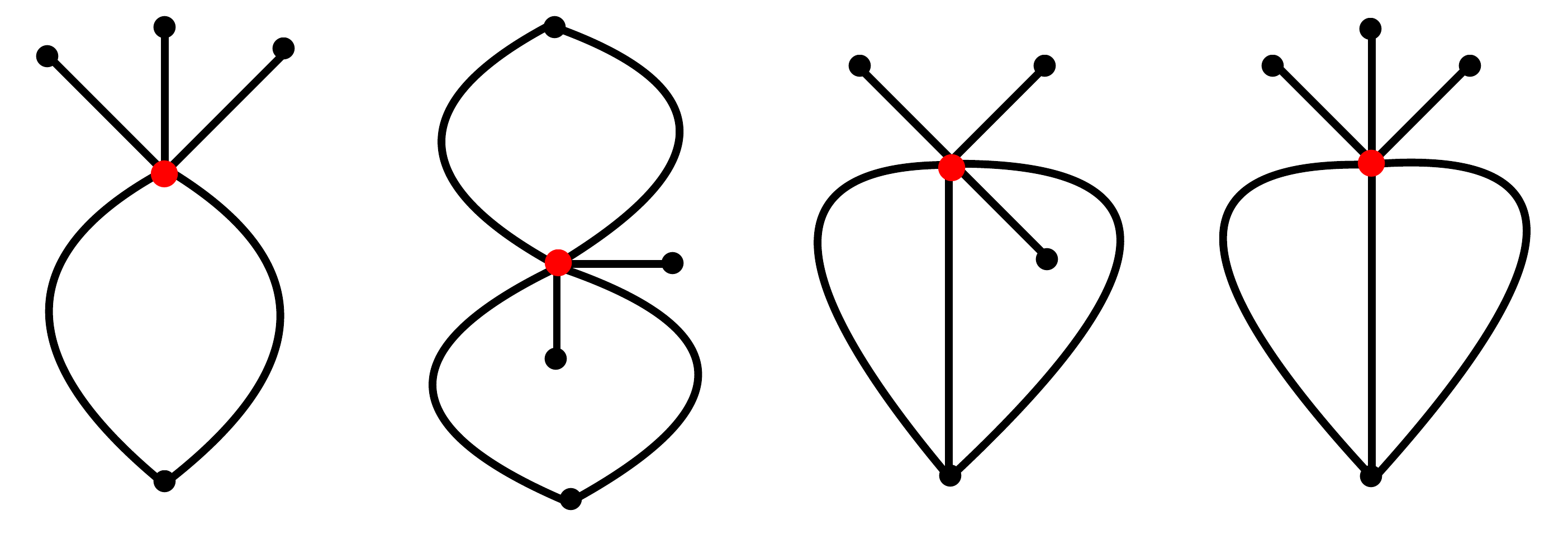} 
 \caption{Caution: these graphs are not channel graphs for $d=4$, they are fake ones.
 Want to know why? See Fact \ref{trivial-unbounded}. }
 \label{fake-channel}
\end{center}
\end{figure}

%
%
%

\subsection{Invariant graph}
The main result in this section is the following 

\begin{theorem}\label{in-graph}
	Let $f$ be a Newton map which is post-critically finite on $ B_f$. Then there exists an invariant graph $G$ such that
	
	1. $ f^N(G)=\Delta_0$ for some integer $N\geq 1$.
    
    2. $\infty$ is a non-cut point with respect to $G$.
	 
%
%
\end{theorem}

\noindent\textbf{The idea of the proof.}	
	Let's sketch the idea, so that the readers can have a rough picture of the proof.
	For each $k\geq 1$, let $\mc C_k=f^{-k}(\Delta_0)\setminus f^{-k+1}(\Delta_0)$.
From  Theorem \ref{in-graph}(1), one may easily imagine that $G$ is actually a union of  some suitable iterated pre-images of $\Delta_0$. 

These iterated pre-images are chosen in an inductive fashion.
First, we extend the graph $\Delta_0$ to a larger one $\Delta_1$, by adding a suitable subset of $\mc C_1$. 
Inductively, at step $k$, we will get an extension of the graph $\Delta_k$ from $\Delta_{k-1}$ by 
adding a subset of $\mc C_1\cup\cdots\cup \mc C_k$. 
%

The choice of the subset of $\mc C_1\cup\cdots\cup \mc C_k$ is delicate, we actually choose a subtable subset such that either its endpoint is a pole, or some iterated pre-image's endpoint is a pole.
This dichotomy is guaranteed by the Shrinking Lemma (see Lemma \ref{lem:shinking}). The heart part of the proof is to show  that any  subset of this kind can touch another one at some pole.
This will be based on the counting number of poles (Propositions \ref{prop:key_prop} and \ref{lem:common_poles}) in the preceding section. 
However one can't apply these results directly. 

To compensate the situation, we need make a modification $G_k$ of the graph $\Delta_k$ in each step.
For these $G_k$'s, we can apply
 Propositions \ref{prop:key_prop} and \ref{lem:common_poles} 
successfully.
For this technical reason, in our discussion, we actually focus on the construction of $G_k$ (whose modification yields $\Delta_k$), and the graphs $\Delta_k$'s don't appear directly in the proof. 

%
%
%
%
Then  Theorem \ref{in-graph}(2) can guide each step of the proof.
	 In order to construct a graph $G$ so that $\infty$ is  a non-cut point, we construct a  sequence of modified graphs
	  $G_k$'s so that
$$\infty\in G_0\subseteq G_1\subseteq G_2\subseteq\cdots,$$
here, the graph $G_0$  is a modified version of the channel graph $\Delta_0$, and $G_{k+1}$ is  constructed inductively so that the difference set $G_{k+1}\setminus G_k$ is the union of  finitely many Jordan arcs, and that  
$$\nu(G_{k+1}, \infty) < \nu(G_k, \infty) \ \text{ if } \  \nu(G_k, \infty)\geq 2.$$
The property $d=\nu(G_0, \infty)>\nu(G_1, \infty)>\nu(G_2, \infty)> \cdots$ implies that after finitely many steps, the procedure will terminate at a graph $G_{\ell}$ with $\nu(G_{\ell}, \infty)=1$, which is equivalent to say that  $\infty$ is a non-cut point for $G_{\ell}$.  Finally, a suitable modification of 
$G_{\ell}$ yields the required graph $G$.

\

\noindent\textbf{The proof of Theorem \ref{in-graph}}	 The proof proceeds in six steps, as follows:




	\begin{proof}[Step 1: from $\Delta_0$ to $G_0$]
	
	The aim of this step is to modify $\Delta_0$ to a new graph $G_0$, such 
	that $G_0$ is disjoint from the $d$ attracting fixed points.

	Consider an immediate root basin $B$ of $f$. Recall that $\Phi_B:B\to\mb{D}$ is a B\"{o}ttcher map, satisfying that $\Phi_B(z)^{d_B}=\Phi_B(f(z))$. Fix a number  $r\in (0,1)$.

	 \begin{figure}[h] 
	 \begin{center}
\includegraphics[height=4.5cm]{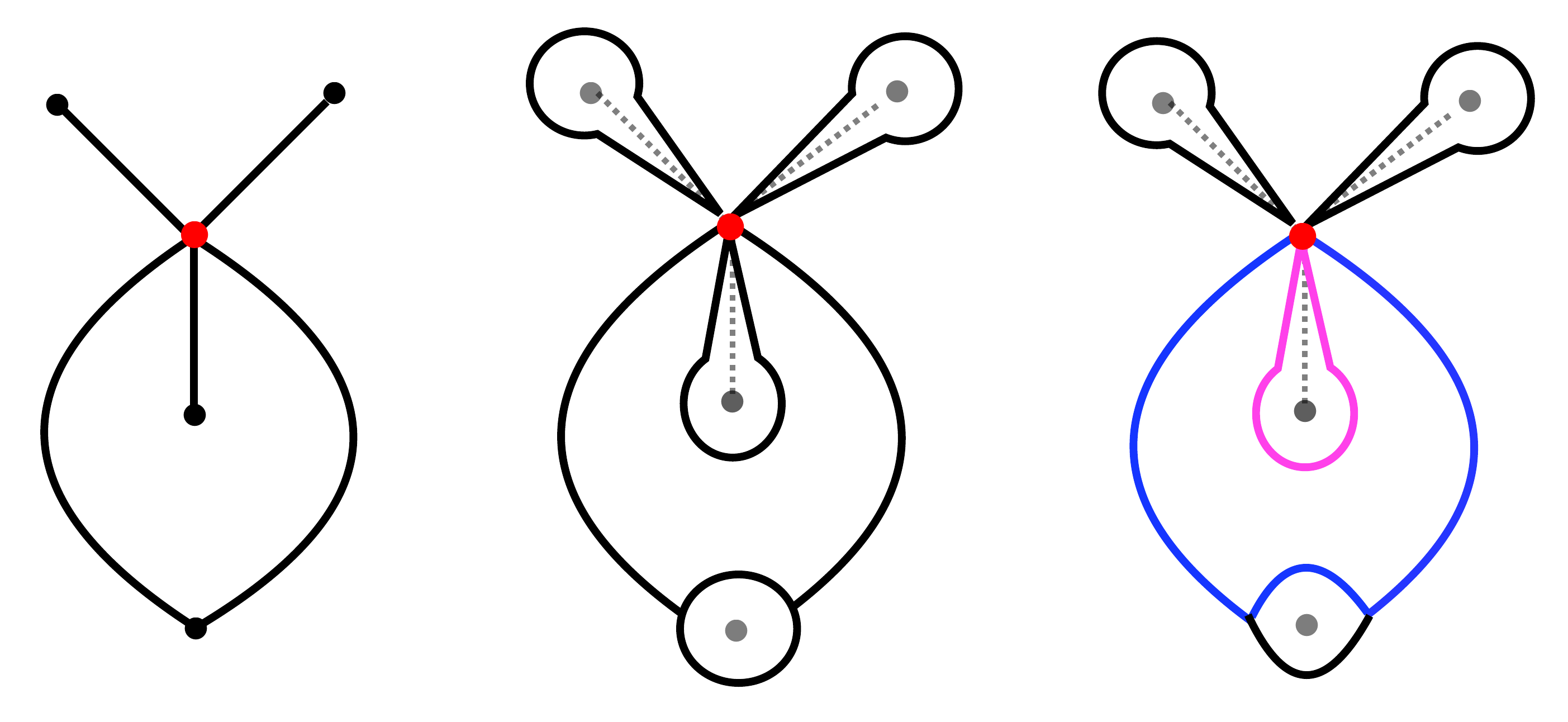} 
 \put(-50,32){$U$} \put(-15,44){$\gamma_1$} \put(-38,57){$\gamma_2$}
 \caption{The channel graph $\Delta_0$ (left), the modified graph $G_0$ (middle), and a non-trivial component
$U$ of $\olC\setminus G_0$ (right). This $U$ can be written as 
$A(\gamma_1, \gamma_2)$, where $\gamma_1$ is the blue curve and $\gamma_2$ is the purple one.}
\label{modified-channel}
\end{center}
\end{figure}

	If $d_B\geq3$,  let 
	$$\Delta_B=\{\infty\}\cup \Phi_B^{-1}\Big(\big\{[r,1) e^{2\pi i k/(d_B-1)}; 0\leq k\leq d_B-2\big\}\Big)\cup\Phi_B^{-1}(r\mathbb S) $$


%
%
	If $d_B=2$, take a small angle $\theta_0\in(0,1/2)$ and define two arcs $\alpha_{\pm}$ in $B$ by
	 $$\alpha_{\pm}=\Phi_B^{-1}(\{e^{s(\log r\pm 2\pi i\theta_0}); 0<s<1\}).$$
	Clearly, $\alpha_{\pm}$ connect $\infty$ to $\Phi_B^{-1}(re^{\pm 2\pi i \theta_0})$, and $\alpha_{\pm}\subseteq f(\alpha_{\pm})$.
%
We set
	$$\Delta_B=\{\infty\}\cup\alpha_+\cup\alpha_-\cup \Phi_B^{-1}(\{r\,e^{2\pi i t};  \theta_0\leq t\leq -\theta_0\}).$$
	
	Finally, let
	$$G_0=\bigcup_{B}\Delta_B,$$
	where the union is taken over all immediate root basins $B$ of $f$. Clearly $G_0$ avoids all centers of the immediate root basins. See Figure \ref{modified-channel}.
	\end{proof}

	\begin{proof}[Step 2: from $G_0$ to $G_1$] 
	For a finite graph $\Gamma\subseteq\olC$ with $\infty \in \Gamma$, its complement $\olC\setminus \Gamma$ has finitely many components.
	There are two kinds of unbounded ones. An unbounded component $U$ of  $\olC\setminus \Gamma$ is called {\it trivial} if
	$\nu(\partial U, \infty)=1$ (equivalently, $\infty$ is a non-cut point of $\partial U$); {\it non-trivial} if 
	$\nu(\partial U, \infty)\geq 2$ (i.e, $\infty$ is a cut point of $\partial U$).
	
	For the graph $G_0$ given by Step 1, the following fact is non-trivial.

	\begin{fact} \label{trivial-unbounded} An unbounded component $U$ of $\olC\setminus G_0$ is trivial if and only if  
	$\partial U=\Delta_B$ with $d_B=2$.
	\end{fact}
	\begin{proof} The  `$\Longleftarrow$' part is obvious.  We need show the `$\Longrightarrow$'
	part.  If it is not true, then a trivial  unbounded component $U$ of $\olC\setminus G_0$ is necessarily a component of $\olC-\Delta_B$ with $d_B\geq 3$ (see Figure \ref{impossible-graph}).

	 \begin{figure}[h] 
	 \begin{center}
\includegraphics[height=5cm]{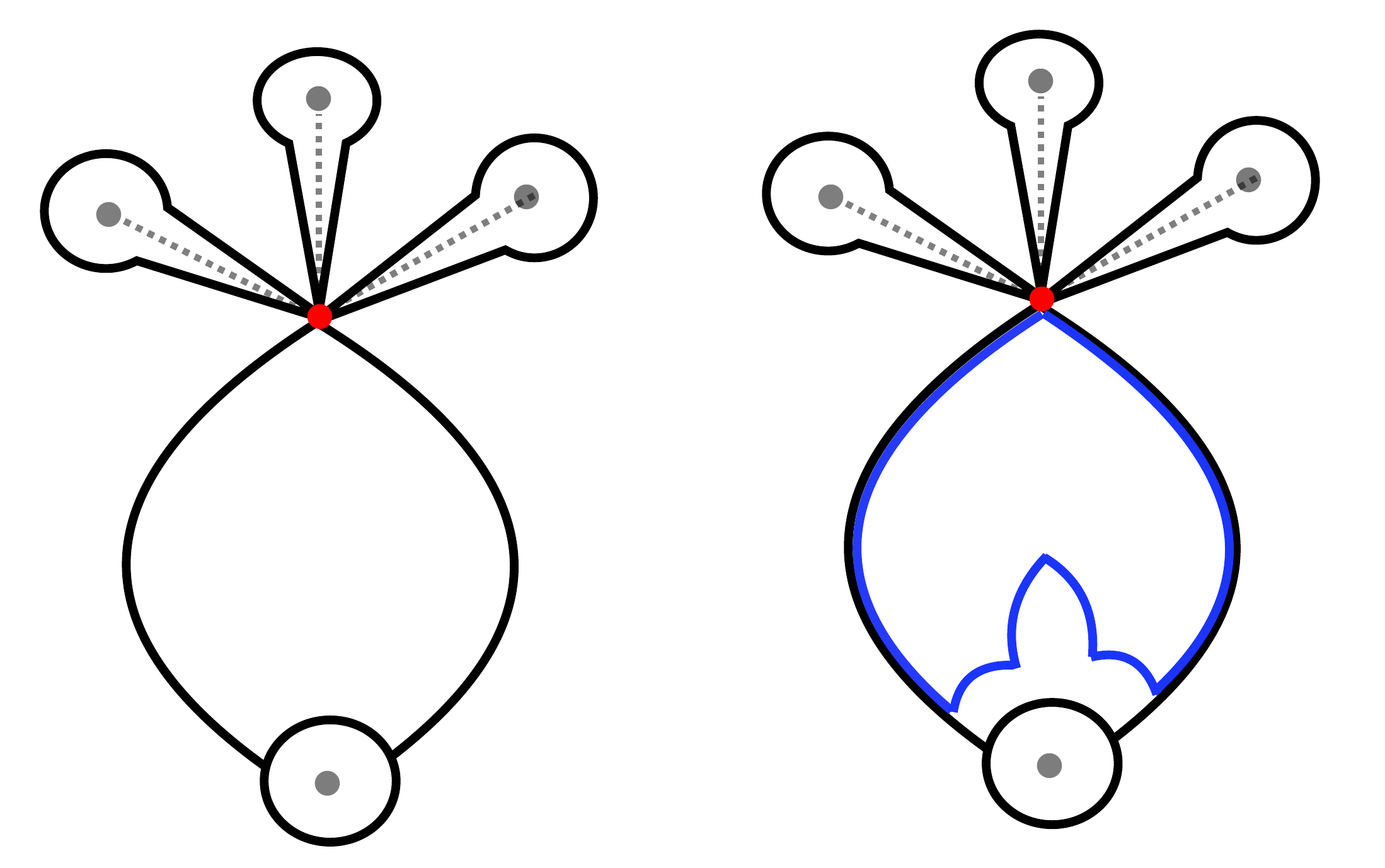} 
 \put(-59,60){$V$} \put(-179, 50){$U$} 
 \caption{This figure shows why the first graph in  Figure \ref{fake-channel} is not a channel graph, and it is used in the proof of 
 Fact \ref{trivial-unbounded}. Here $V$ is a  connected component of $f^{-1}(U)$ (in general, $V$ is not necessarily a topological disk).}
\label{impossible-graph}
\end{center}
\end{figure}

	Note that $U\subseteq f(U)$ and the image $f(U\cap B)$ covers $U\cap B$  twice.  One may also observe that there is a component $V$ of $f^{-1}(U)$, contained in $U$, such that
	$\partial V$ contains two sections of fixed internal rays. Note that $\ol{V}$ contains only one fixed point, namely $\infty$. On the other hand, by Lemma \ref{lem:fixed_points}, one has
	$$\#\tu{Fix}(f|_{\ol{V}})=\tu{deg}(f|_{\partial V}).$$
	This gives a contradiction, because $\#\tu{Fix}(f|_{\ol{V}})=1$ and $\tu{deg}(f|_{\partial V})\geq 2$.
\end{proof}

Fact \ref{trivial-unbounded} has the following interesting corollary:

\begin{fact} \label{structure-cor} Under the assumption $d={\rm deg}(f)\geq 3$, there are 
at least two immediate root basins $B$ with $d_B=2$.
As a consequence,  trivial and non-trivial components both exist in  ${\rm Comp}(\olC\setminus G_0)$. 
	\end{fact}
\begin{proof}  This fact is obvious if  $d_B=2$ for all immediate root basins $B$. 
So we may assume $d_B\geq3$ for some $B$. In this case,
$\olC-\Delta_B$ has at $d_B-1\geq 2$  unbounded components. To prove the fact, we will show
each unbounded component of $\olC-\Delta_B$ contains an immediate root basin $B''$ with $d_{B''}=2$.

 In fact, if some unbounded component, say $U$, contains no immediate root basin $B'$ with
$d_{B'}=2$, then there is
an immediate root basin 
$B''\subseteq U$ such that $d_{{B}''}\geq3$ and some unbounded component $V$ of $\olC-\Delta_{B''}$ is also a component $\olC\setminus G_0$. This $V$ is trivial. 
  However,
   this contradicts Fact \ref{trivial-unbounded}.
%
%
\end{proof}
	
%


The idea of this step is to take pullbacks of the boundaries of non-trivial unbounded  components of $\olC\setminus G_0$.

	
	
	Let $Q_0$ be a non-trivial unbounded component of $\olC\setminus G_0$.
	Such a component can be written as $Q_0=A(\gamma_1,\cdots,\gamma_n)$ with $n\geq 2$  and
	 $\gamma_1,\cdots,\gamma_n$ are independent Jordan curves. One may verify that the curves $\gamma_1,\cdots,\gamma_n$ satisfy the conditions (a)-(d) in Proposition \ref{prop:key_prop}. Indeed, by the construction in Step 1, the conditions (a)(b)(c) are satisfied,  
	   we only need to check that the unbounded component $\alpha_k$ of $f^{-1}(\gamma_k)$ is contained in $\ol{\tu{Ext}}(\gamma_k)$. To see this,  note that  either $\alpha_k\subseteq\ol{\tu{Int}}(\gamma_k)$ or $\alpha_k\subseteq\ol{\tu{Ext}}(\gamma_k)$;
	   the former cannot happen because the behavior of $f$ on each immediate root basin $B$
	  is conjugate to the power map $z\mapsto z^{d_B}$ on $\mb{D}$. 
	
	 We then apply Proposition \ref{prop:key_prop} to the  independent curves $\gamma_1,\cdots,\gamma_n$, and 
	obtain the new independent Jordan curves $\gamma_1^{-1},\cdots,\gamma_n^{-1}$. 
	For these new curves, observe that
	\begin{itemize}
    \item    For each $k$, the curve $\gamma_k^{-1}$ contains at least one pole  of $f$ in $\mathbb{C}$.
    \item     One has 
    $\gamma_i^{-1}\cap\gamma_j^{-1}\subseteq f^{-1}(\infty)$, for $i\neq j$. 
	     \end{itemize}

	In the following, we will show that at least two curves of $\gamma_1^{-1},\cdots,\gamma_n^{-1}$ have a common pole.
	In fact, if this is not true, then the set $\ol{A(\gamma_1^{-1},\cdots,\gamma_n^{-1})}$ contains at least $n$ distinct poles.
On the other hand, by Proposition \ref{prop:key_prop}, the curves $\gamma_1^{-1},\cdots,\gamma_n^{-1}$ satisfy the assumptions
in  Proposition \ref{lem:common_poles}. Then by Proposition \ref{lem:common_poles}, the number of poles in $\ol{A(\gamma_1^{-1},\cdots,\gamma_n^{-1})}$ is exactly $n-1$ (counting multiplicity).  This is a contradiction.
	
	
	Finally, let's define three curve families $\Gamma_0,  \Gamma^*_1, \Gamma_1$ and a new graph $G_1$ by
	$$\Gamma_0=\bigcup_{Q_0}\big\{\gamma_1,\cdots,\gamma_n\big\},\  \Gamma^*_1=\Gamma_1=\bigcup_{Q_0}\big\{\gamma_1^{-1},\cdots,\gamma_n^{-1}\big\},\  G_1=\bigcup_{\gamma\in \Gamma^*_1}\gamma.$$	
	 where $Q_0$ ranges over all non-trivial unbounded components of $\olC\setminus G_0$. The existence of common poles for the curves $\gamma_k^{-1}$'s,  implies that 
	$$\nu(G_1,\infty)<\nu(G_0,\infty)=d.$$
	Note that we have the inclusion 
	$$f(\Gamma_1):=\{f(\gamma);\gamma\in\Gamma_1\}\subseteq \Gamma_0, \ \ f(G_1)\subseteq G_0.$$
	\end{proof}
	
	\begin{proof}[Step 3: from $G_1$ to $G_2$]
	The idea of the proof is similar to that of Step 2: taking pullbacks of the boundaries of non-trivial unbounded components
	 of $\olC\setminus G_1$, until some pullback hits a pole. We remark that this step actually reveals
	  the general case of the pullback procedure. The Shrinking Lemma is involved to deal with the difficulty arising here.
	 
	 Note that if $\nu(G_1,\infty)=1$, then $\infty$ is a non-cut point of $G_1$, hence there is nothing to do in this step.
	 So we may assume that  $\nu(G_1,\infty)\geq 2$, and this case happens if and only if 
	 there exists a non-trivial unbounded component, say $Q_1$, of $\olC\setminus G_1$. 
	 
%
%
	
	Note that $Q_1$ is contained in some  $Q_0=A(\gamma_1,\cdots,\gamma_n)$ in Step 2, and that $Q_0$ is decomposed by the curves $\gamma^{-1}_1,\cdots,\gamma_n^{-1}$ into several parts. Since $Q_1$ is non-trivial, it can be written as
	$$Q_1=A(\alpha_1,\cdots,\alpha_m),$$
	 where $\alpha_1,\cdots,\alpha_m$ are independent Jordan curves. 
	 The set $\{\alpha_1,\cdots,\alpha_m\}$ can be decomposed into two disjoint subsets 
	 $\Gamma_1(Q_1)$ and $\Xi(Q_1)$, such that
	 
%
%
	 \begin{itemize}
    \item    Each curve $\lambda\in \Gamma_1(Q_1)$ comes from $\Gamma_1$, namely $\Gamma_1(Q_1)\subseteq 
    \Gamma_1$;
    \item    Each curve $\eta\in \Xi(Q_1)$ is {\it new}, i.e.,  composed of several sections, each section is a part of a curve in $\Gamma_1$.  
	     \end{itemize}

Note that each curve $\lambda\in \Gamma_1(Q_1)\subseteq 
    \Gamma_1$ must contain a pole in $\mb{C}$.  Each curve $\eta \in \Xi(Q_1)$ must also contain a pole
    in $\mb{C}$, because if two curves in $\Gamma_1$ intersect at a point other than $\infty$, then this point is a pole.

\begin{fact} \label{new-curve} $\Xi(Q_1)\neq \emptyset$. In other words, at least one curve among $\alpha_k$'s is new.
	\end{fact}
	\begin{proof} If not, then $\Gamma_1(Q_1)=\{\alpha_1,\cdots,\alpha_m\}\subseteq\Gamma_1$, and 
	$\alpha_i\cap \alpha_j=\{\infty\}$ for $i\neq j$. Therefore the number of poles 
	in $\ol{Q}_1=\ol{A(\alpha_1,\cdots,\alpha_m)}$ is at least $m$.
	
	On the other hand, applying Proposition \ref{prop:key_prop} to the curves 
	$f(\alpha_1), \cdots, f(\alpha_m)$, we see that the curves $\alpha_1,\cdots,\alpha_m$ satisfy the assumptions
in  Proposition \ref{lem:common_poles}. Then by Proposition \ref{lem:common_poles}, the number of poles in $\ol{A(\alpha_1,\cdots,\alpha_m)}$ is exactly $m-1$ (counting multiplicity).  This is a contradiction.
%
%
\end{proof}
	
	We may write 
	$$\Gamma_1(Q_1)=\{\lambda_1,\cdots,\lambda_r\}, \ \Xi(Q_1)=\{\eta_1,\cdots, \eta_s\}.$$
	Consider the Jordan curves
	$$f(\lambda_1),\cdots,f(\lambda_r), \eta_1,\cdots,\eta_s.$$ 
	One may verify that these curves are independent, and satisfy the conditions of Proposition \ref{prop:key_prop}.
	Applying Proposition \ref{prop:key_prop} to these curves, for each $\eta_j$, one get $\eta_j^{-1}$. Moreover, the curves 
	 $\lambda_1,\cdots,\lambda_r,\eta^{-1}_1,\cdots, \eta^{-1}_s$ are independent. 
	

	If one of the resulting curves $\eta^{-1}_j$'s, say $\eta_k^{-1}$, is disjoint from poles in $\mb{C}$, then it is exactly the unbounded component of $f^{-1}(\eta_k)$, and
	
	 \begin{itemize}
    \item   $\eta_k^{-1}$ intersects each of $\lambda_1,\cdots,\lambda_r,\eta^{-1}_j, j\neq k,$ only at $\infty$.
    \item   $f:\eta_k^{-1}\to \eta_k$ is one-to-one. 
	     \end{itemize}

For any integer $l\geq 1$, 
one may define $\eta_k^{-l-1}$ inductively by
$$\eta_k^{-l-1}=(\eta_k^{-l})^{-1}$$
as long as the curves
$$f(\lambda_1),\cdots,f(\lambda_r),\eta_1,\cdots, \eta_{k-1}, \eta_k^{-l}, \eta_{k+1}, \cdots,\eta_s,$$
are independent, and $\eta_k^{-1},\cdots, \eta_k^{-l}$ are disjoint from  poles in $\mb{C}$.
In this case, the curves $\lambda_1,\cdots,\lambda_r,\eta_1^{-1},\cdots, \eta_{k-1}^{-1},
		\eta_k^{-l-1},\eta_{k+1}^{-1},\cdots,\eta_s^{-1}$ are independent.


To continue our discussion, we need the following crucial fact:

\begin{lemma}\label{integer-existence}
	For each curve $\eta\in\Xi(Q_1)=\{\eta_1,\cdots, \eta_s\}$, there is a minimal integer $N=N_{\eta}\geq 1$, such that
	$\eta^{-N}$ contains a pole of $f$ in $\mb{C}$.
\end{lemma}	

	\begin{proof}
		If it is not true for $\eta=\eta_k$, then for any $j\geq 1$, the curves $\lambda_1,\cdots,\lambda_r$, $\eta_1^{-1}$, $\cdots, \eta_{k-1}^{-1},
		\eta_k^{-j},\eta_{k+1}^{-1},\cdots,\eta_s^{-1}$ are independent, and
		the domains
		$$H_j=A(\lambda_1,\cdots,\lambda_r,\eta_1^{-1},\cdots, \eta_{k-1}^{-1},
		\eta_k^{-j},\eta_{k+1}^{-1},\cdots,\eta_s^{-1})$$
		satisfy
		$$H_1\supseteq H_2\supseteq \cdots \supseteq H_j\supseteq \cdots.$$
		In particular, we have
		$$\tu{Int}(\eta_k^{-1})\subseteq \tu{Int}(\eta_k^{-j}) \ \text{ and } \ \tu{Int}(\lambda_1)\cup\cdots\cup 
		\tu{Int}(\lambda_r) \subseteq \tu{Ext}(\eta^{-j}_k).$$ 
		This implies that the spherical diameters $\tu{diam}(\eta_k^{-j})$ with $j\geq 1$ are uniformly bounded from below and above.
		
		To get a contradiction, we will show $\tu{diam}(\eta_k^{-j})\to 0$ as $j\to\infty$. 	
		Note that all Jordan curves $\eta_k^{-j}$'s traverse two distinct immediate root basins, say 
		$B', B''$.
		We may decompose $\eta_k^{-j}$ into three segments $\beta_j,\beta'_j, \beta''_j$:
		 \begin{itemize}
    \item   $\beta_j'$ (resp. $\beta_j''$) is the intersection of $\eta_k^{-j}$ with the closure of some fixed internal ray of $B'$
    (resp. $B''$);
    \item   $\beta_j=\eta_k^{-j}\setminus(\beta_j'\cup\beta_j'')$. 
	     \end{itemize}
%
%
The observation $\bigcap_{j} \beta'_j=\bigcap_{j}\beta_j''=\{\infty\}$ implies that $\tu{diam}(\beta_j')\to 0$ and $\tu{diam}(\beta_j'')\to 0$ as $j\to\infty$.
		   It remains to prove 
		$$\tu{diam}(\beta_j)\to 0\tu{ as }j\to\infty.$$

	Note that	$f:\beta_{j+1}\rightarrow\beta_j$ is a homeomorphism. By the construction of $\beta_j$, 
		 there is a large integer $n_0>0$ such that $\beta_{0}\cap\beta_{n_0}=\emptyset$.
		 It follows that $\beta_{jn_0}\cap\beta_{(j+1)n_0}=\emptyset$ for all $j\geq 0$. By Proposition \ref{prop:key_prop} (3), there exists a sequence of open sets $\{U_j\}$ with $\beta_{jn_0}\Subset U_j$ such that $f^{n_0}(U_{j+1})=U_j$ and $U_{j+1}\cap U_0=\emptyset$ for $j\geq 0$. By the Shrinking Lemma (see Lemma \ref{lem:shinking}), 
	one has	$\tu{diam}(\beta_{jn_0})\to 0$ as $j\to\infty$. 
		The shrinking property $H_j\supseteq H_{j+1}$ implies that $\tu{diam}(\beta_j)\to 0$ as $j\to\infty$.  This gives a contradiction. 
	\end{proof}

	 \begin{figure}[h] 
	 \begin{center}
\includegraphics[height=6cm]{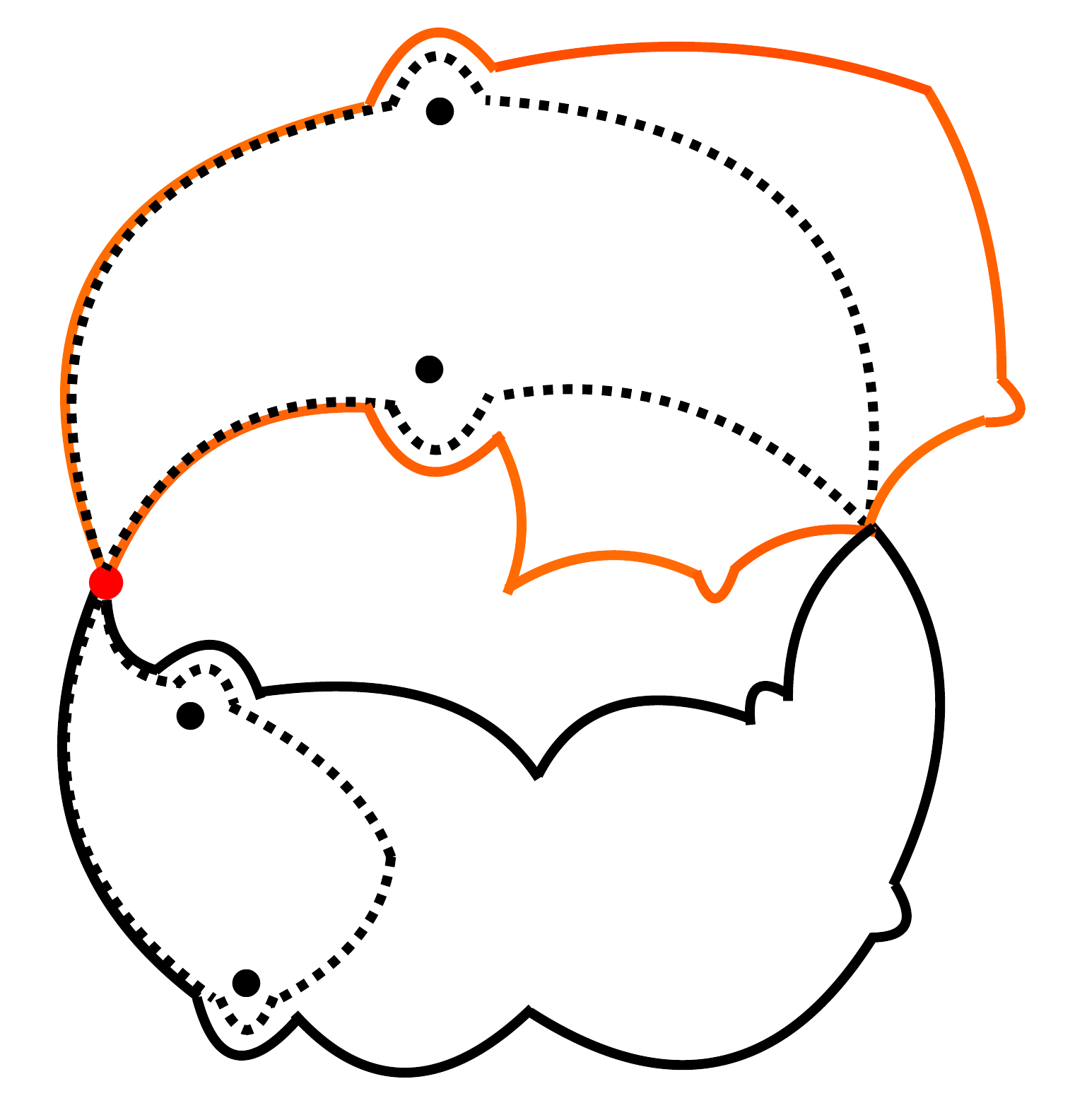} 
 \put(-42,50){$\eta_1^{-1}$} \put(-125, 48){$\eta_1$} \put(-110, 38){$\bullet \ p_1$} \put(-36,87){$\bullet\ p_2$}
 \put(-50,122){$\eta_2$}  \put(-10,130){$\eta_2^{-1}$}
 \caption{In this example, the black dashed curves $\eta_1, \eta_2$ contain the poles $p_1, p_2$, respectively. $\eta_1^{-1}$ is the black curve containing the new pole $p_2$, and
  $\eta_2^{-1}$ is the orange curve containing the same pole $p_2$ as $\eta_2$ does.
  The pole $p_2$ is also a critical point.
 }
\label{pre-pole}
\end{center}
\end{figure}
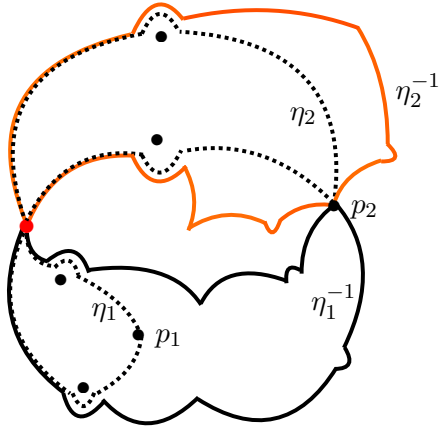

	\begin{remark}\label{common-preimage-pole} 
	In Lemma \ref{integer-existence}, it may happen that 
	$$\eta^{-N}\cap f^{-1}(\infty)=\eta\cap  f^{-1}(\infty).$$ 
	In other words, the poles in $\eta^{-N}$ are already contained in $\eta$, hence not new. Figure \ref{pre-pole} gives such an example.
	\end{remark}

	By Lemma \ref{integer-existence}, for each $\eta_k$, there exists a minimal integer $N_k\geq 1$
	 such that $\eta_k^{-N_k}$ contains a pole.  One may verify further that 
	 the Jordan curves $$f(\lambda_1), \cdots, f(\lambda_r),\eta_1^{-l_1},\cdots,\eta_s^{-l_s}$$ with $0\leq l_1< N_1,\cdots,0\leq l_s< N_s$ are independent and
	  satisfy the conditions (a)-(d) in Proposition \ref{prop:key_prop}. 
	Applying Proposition \ref{prop:key_prop} to  the curves $$f(\lambda_1),\cdots,f(\lambda_r),\eta_1^{-N_1+1},\cdots,\eta_s^{-N_s+1},$$
	we get the following independent curves
	 $$\lambda_1,\cdots,\lambda_r,\eta_1^{-N_1},\cdots,\eta_s^{-N_s},$$
	 each of which contains a pole in $\mb{C}$. Again Proposition \ref{lem:common_poles} implies that at least two of these curves contains a common pole in $\mb{C}$. We remark that each $\eta_k$ passes through exactly two immediate root basins $B',B''$, and so do the curves $\eta_k^{-j}, 1\leq j\leq N_k$; these Jordan curves overlap on an invariant subarc in $\{\infty\}\cup B'\cup B''$.

	Let's define two families of Jordan curves
	$$\Gamma^*_2=\bigcup_{Q_1}\big\{\eta_1^{-1},\cdots, \eta_1^{-N_1},\cdots,\eta_s^{-1},\cdots, \eta_s^{-N_s}\big\},\Gamma_2=\bigcup_{Q_1}\big\{\eta_1^{-N_1},\cdots,\eta_s^{-N_s}\big\}$$
	where $Q_1$ ranges over all non-trivial unbounded components of $\olC\setminus G_1$. 
	Now we get a new graph $G_2$, which is an extension of $G_1$:
	$$G_2=G_1\bigcup \bigcup_{\gamma\in \Gamma^*_2}\gamma.$$	
%
	 
Observe that $f(G_2)\subseteq G_0\cup G_2$.	 The construction and the existence of common poles for the curves 
	  $\lambda_1,\cdots,\lambda_r$, $\eta_1^{-N_1},\cdots,\eta_s^{-N_s}$ 
	   imply that
	 $$ \nu(G_2,\infty)<\nu(G_1,\infty).$$
	\end{proof}
	
	\begin{proof}[Step 4: from $G_k$ to $G_{k+1}$, an induction procedure]
%
%
Suppose  for some $k\geq 2$, we have constructed the graphs $G_1, \cdots, G_k$ and the curve families
$\Gamma^*_1, \Gamma_1$
$\cdots$, $\Gamma^*_k, \Gamma_k$, inductively in the following way
$$\Gamma^*_l=\bigcup_{Q_{l-1}}\big\{\eta_1^{-1},\cdots, \eta_1^{-N_1},\cdots,\eta_s^{-1},\cdots, \eta_s^{-N_s}\big\},$$
$$\Gamma_l=\bigcup_{Q_{l-1}}\big\{\eta_1^{-N_1},\cdots,\eta_s^{-N_s}\big\}, \ G_{l}=G_{l-1}\bigcup \bigcup_{\gamma\in \Gamma^*_l}\gamma,$$
where $Q_{l-1}$ is taken over all non-trivial unbounded components of $\olC\setminus G_{l-1}$, and that
 $f(G_l)\subseteq G_0\cup G_l$ and $\nu(G_l,\infty)<\nu(G_{l-1},\infty)$, for $2\leq l\leq k$.

If $\nu(G_k,\infty)=1$, then the step is done. If $\nu(G_k,\infty)\geq 2$,
	  we consider each non-trivial unbounded component $Q_{k}$ of  $\olC\setminus G_k$.
	  Write $Q_k$ as $A(\delta_1,\cdots, \delta_t)$ and
	   compare the curves $\delta\in\{\delta_1,\cdots, \delta_t\}$ with the curves in $\Gamma_1\cup\cdots\cup\Gamma_k$,
	   there are two possibilities: either 	 
	 \begin{itemize}
    \item    $\delta\in\Gamma_1\cup\cdots \cup\Gamma_k$, or
    \item    $\delta$ is {\it new}, i.e. $\delta\notin\Gamma_1\cup\cdots \cup\Gamma_k$. In this case, $\delta$ is composed of several sections, each section
    is a part of a curve in $\Gamma_1\cup\cdots \Gamma_k$.  
	     \end{itemize}
   Let $\Xi(Q_k)$ be the collection of new curves. For each $\eta\in \Xi(Q_k)$,
   by the same argument as Lemma \ref{integer-existence}, there is a minimal integer  $N_\eta\geq1$  such that $\eta^{-N_\eta}$ meets a pole in $\mb{C}$. 	   
	   By Proposition \ref{lem:common_poles}, at least two curves of $\{\eta^{-N_\eta}; \eta\in \Xi(Q_k)\}$ share a common pole. 
	  Similarly as above, we get a new graph $G_{k+1}$ and two curve families $\Gamma_{k+1}\subseteq \Gamma_{k+1}^*$:
	  $$\Gamma^*_{k+1}=\bigcup_{Q_{k}} \bigcup_{\eta\in \Xi(Q_k)}\big\{\eta^{-1},\cdots,\eta^{-N_\eta}\big\}, 
	  \ \Gamma_{k+1}=\bigcup_{Q_{k}} \bigcup_{\eta\in \Xi(Q_k)}\big\{\eta^{-N_\eta}\big\},$$
	  $$G_{k+1}=G_{k}\bigcup \bigcup_{\gamma\in \Gamma^*_{k+1}}\gamma,$$
where $Q_{k}$ is taken over all the non-trivial unbounded component $Q_{k}$ of  $\olC\setminus G_k$.	  
	  
	   The resulting graph $G_{k+1}$
	   satisfies 
	   $$\nu(G_{k+1},\infty)<\nu(G_k,\infty), \  f(G_{k+1})\subseteq G_0\cup G_{k+1}.$$
	   
	    After finitely many steps, we have $\nu(G_\ell,\infty)=1$ for some minimal integer $\ell\geq 1$. Then $\infty$ is a non-cut point for the graph $G_\ell$, and $f(G_\ell)\subseteq  G_0\cup G_\ell$.
	 \end{proof}
	
	\begin{proof}[Step 5: from $G_\ell$ to $G$, a natural modification]
	
	By construction, 
	  all points in $G_\ell\cap  J(f)$ are iterated pre-images of $\infty$, and
	$$f(G_\ell\cap  J(f))\subseteq G_\ell\cap  J(f), \ f^N(G_\ell\cap  J(f))=\{\infty\}.$$
	
	Note that for any $0\leq k\leq \ell$, the graph $G_k$ is a union of some curves in 
	$$\Gamma=\Gamma_0\cup\Gamma^*_1\cup \Gamma^*_2\cup \Gamma^*_3\cup  \cdots\cup \Gamma^*_\ell.$$

	To give a natural modification of $G_{\ell}$, 
	it suffices to define the modification 
	of each curve $\delta\in \Gamma$. This goes in the following way.
	
	 \begin{figure}[h] 
	 \begin{center}
\includegraphics[height=4.5cm]{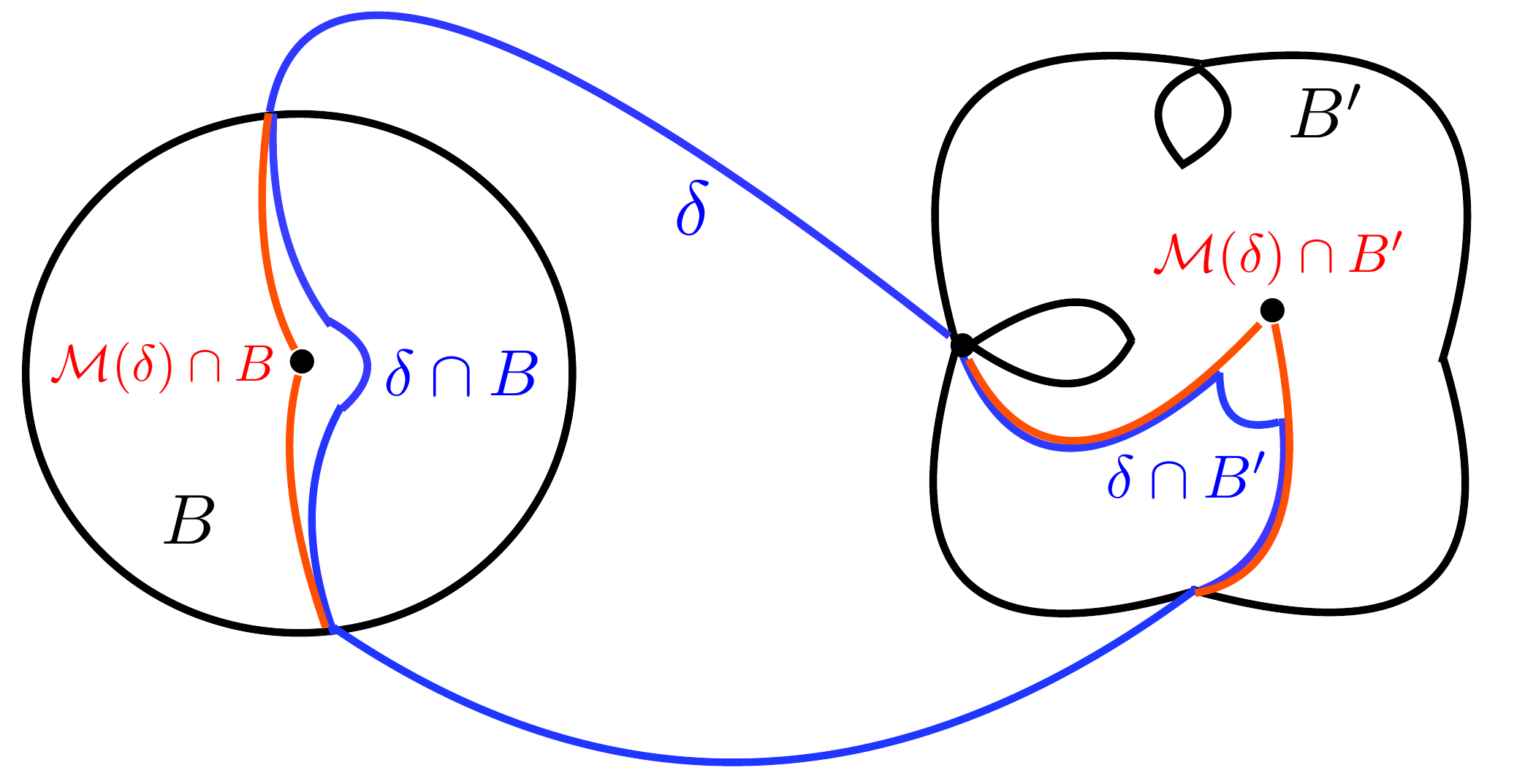} 
 \caption{This figure shows how to modify a curve arc by arc.
 $B\cap \delta$ is tangent to two internal rays near $\partial B$, while $B'\cap \delta$  is equal to two
 internal rays near $\partial B'$. 
 }
\label{moodification}
\end{center}
\end{figure}

Let $B\in\tu{Comp}(B_f)$ with $B\cap \delta\neq\emptyset$, then $B\cap \delta$ consists of finitely many
 components. Suppose $B$ is eventually iterated to the immediate root basin $B_0$.
 Note that each component $\sigma$ of $B\cap \delta$  is an open arc, and near the boundary  $\partial B$, $\sigma$ is either  tangent to (if $d_{B_0}=2$) or equal to (if $d_{B_0}\geq 3$) two internal rays (see Figure \ref{moodification}), say
 $R_B(\alpha), R_B(\beta)$. We define the modification $\mathcal{M}(\sigma)$ of $\sigma$ by
 $$\mathcal{M}(\sigma)=R_B(\alpha)\cup R_B(\beta)\cup\{c_B\},$$
  where $c_B$ is the center of $B$ (it is possible that $\alpha=\beta$).
 We then set
 $$\mathcal{M}(\delta)=\ol{\bigcup_B\bigcup_\sigma  \mathcal{M}(\sigma)},$$
 where $B$ ranges over all components $B\in\tu{Comp}(B_f)$ with $B\cap \delta\neq\emptyset$ 
 and $\sigma$ is taken over all components of $B\cap \delta$.
 
  \begin{figure}[h] 
	 \begin{center}
\includegraphics[height=5.5cm]{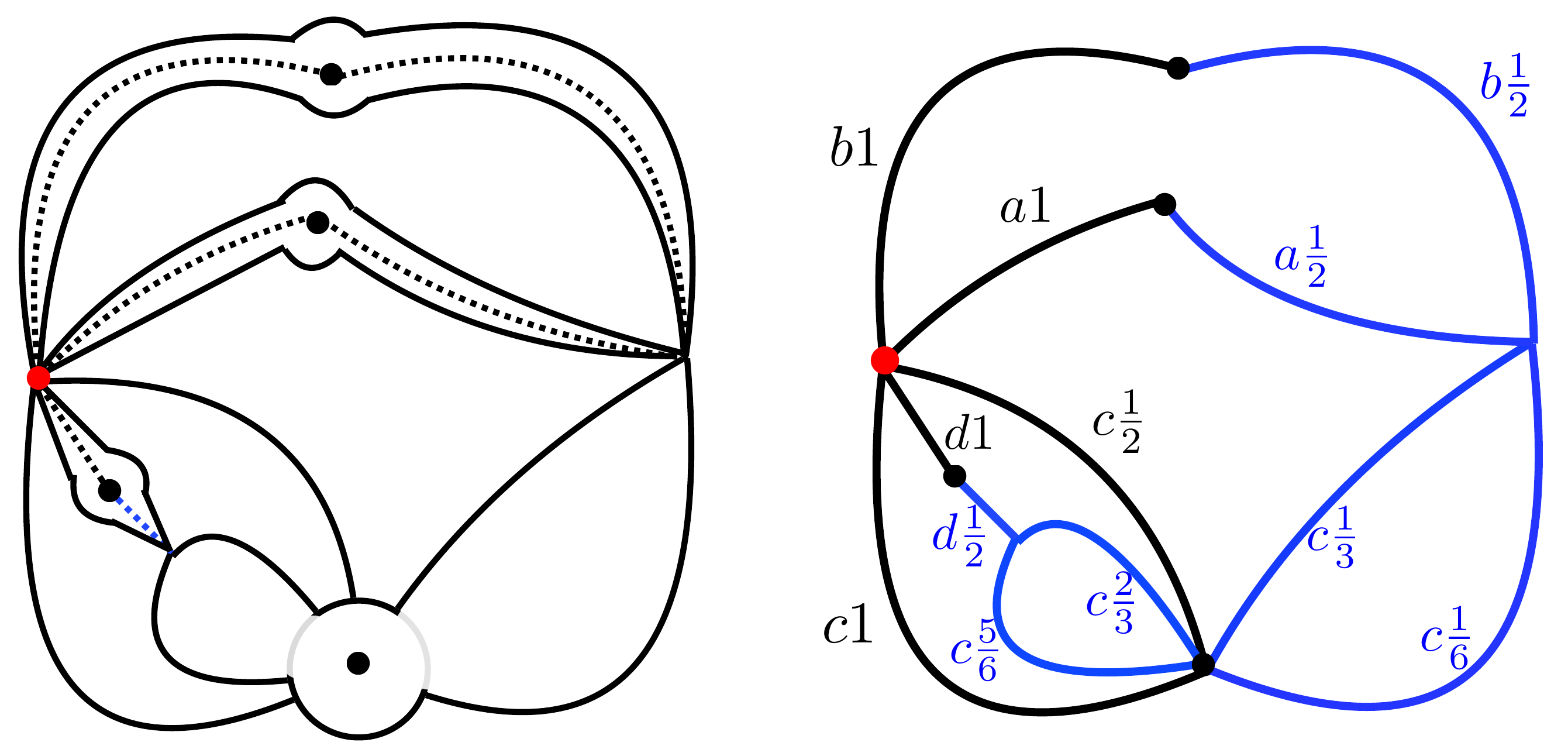} 
 \caption{This figure shows how to get the graph $G$ (right) by a natural modification of $G_\ell$ (left, $\ell=2$) in Step 5.
 }
\label{graph-final}
\end{center}
\end{figure}

%
By the law $\mathcal M(\delta_1\cup \delta_2)=\mathcal M(\delta_1)\cup\mathcal M(\delta_2)$, we obtain  
the modification of the graphs $G_k$'s, which satisfy
 $$\Delta_0=\mathcal{M}(G_0)\subseteq \mathcal{M}(G_1) \subseteq \cdots\subseteq \mathcal{M}(G_\ell)\subseteq f^{-N}(\Delta_0).$$
 
 Let $G=\mathcal{M}(G_\ell)$. Clearly one has $f^{N}(G)=\Delta_0$. Moreover,
 $$f(G)=\mathcal{M}(f(G_\ell))\subseteq \mathcal{M}(G_0\cup G_\ell)=\mathcal{M}(G_\ell)=G.$$
 $$ \nu(G,\infty)=\nu(G_\ell,\infty)=1.$$

	\end{proof}

	\begin{proof}[Step 6: $f^{-1}(G)$ is connected]  To prove the connectivity of $f^{-1}(G)$, we need investigate some properties of $G$ and $G_\ell$ (given in Step 4) first.

	\begin{fact} \label{Gl-property} Each component of $\wh{\mb{C}}\setminus G_\ell$
is a Jordan disk.	
		\end{fact}
	\begin{proof}	It is equivalent to show that $\nu(G_\ell, z)=1$ for all $z\in G_{\ell}$. Clearly this is true for $z=\infty$ by the construction of $G_\ell$.
	For $z\in G_{\ell}-\{\infty\}$,
	note that $G_\ell=\bigcup_{\delta\in \Gamma\setminus\Gamma_0}\delta$ (here $\Gamma,\Gamma_0$ are defined in Step 5) and  $\infty\in \bigcap_{\delta\in \Gamma\setminus\Gamma_0}\delta$.	
	The observation 
	$G_{\ell}\setminus\{z\}=\bigcup_{\delta\in \Gamma\setminus\Gamma_0}(\delta\setminus\{z\})$ and
	$\infty\in \bigcap_{\delta\in \Gamma\setminus\Gamma_0}(\delta\setminus\{z\})$ imply that 
	 $G_{\ell}\setminus\{z\}$ is connected, hence $z$ is not a cut point of $G_\ell$.
	\end{proof}
	
	\begin{prop}\label{cor:non-cutpoints_graph} The graph $G$ satisfies  

1. Any point in $G\cap J(f)$ is not a cut point of $G$.

2. The center of any immediate root basin is not a cut point of $G$.

3. For any immediate root basin $B$, the intersection $G\cap \overline{B}$ is connected.
 In other words, any Julia point in $\overline{B}\cap J(f)$ is linked to the center of $B$ by an
internal ray in $G$.

%
%
\end{prop}

\begin{proof}
%

 It is worth observing that
  $$G=\mathcal{M}(G_\ell)=\bigcup_{\delta\in \Gamma\setminus\Gamma_0}\mathcal{M}(\delta).$$
 
1. For any $z\in  G\cap J(f)$ and $z\neq\infty$, the facts 
$$G\setminus\{z\}=\bigcup_{\delta\in \Gamma\setminus\Gamma_0}(\mathcal{M}(\delta)\setminus\{z\}), \ \infty\in \bigcap_{\delta\in \Gamma\setminus\Gamma_0}(\mathcal{M}(\delta)\setminus\{z\})$$ imply that 
$G\setminus\{z\}$ is  connected, hence $z$ is not a cut point of $G$.
%
%
	
2.   Recall that each curve
$\delta\in\Gamma^*_2\cup \Gamma^*_3\cup  \cdots\cup \Gamma^*_\ell$
starts at an immediate root basin $B'$ and
terminates at a different one $B''$. Each curve $\delta\in \Gamma^*_1$ will be connected to another immediate root basin by another curve $\beta\in \Gamma_1\cup\Gamma_2\cup\cdots\cup \Gamma_\ell$.
It follows that after the modification, 
 the centers of immediate root basins are not cut points with respect to $G$. 

%
	
3.  Note that each curve
$\delta\in\Gamma^*_2\cup \Gamma^*_3\cup  \cdots\cup \Gamma^*_\ell$
meets exactly two different immediate root basins. By construction, if $\delta\cap {B}\neq\emptyset$ for some
 immediate root basin $B$, then  $\mc{M}(\delta)\cap \ol{B}$ is the closure of the union of two  internal rays.  (This implies, in particular, that $\mc{M}(\delta)\cap \ol{B}$ has no isolated point)

Note also for $\delta\in\Gamma_1$, the intersection $\mc{M}(\delta)\cap \ol{B}$ is the closure of the union of   two (if $d_{B}=2$) or  four (if $d_{B}>2$, see Figure \ref{impossible-graph})  internal rays.
Therefore,
$$G\cap \ol{B}=
 \bigcup_{\delta\in \Gamma\setminus\Gamma_0} (\mc{M}(\delta)\cap \ol{B})$$
is the closure of the union of finitely many internal rays, hence connected.
%
%
\end{proof}

%
%

	  \begin{figure}[h] 
	 \begin{center}
\includegraphics[height=6cm]{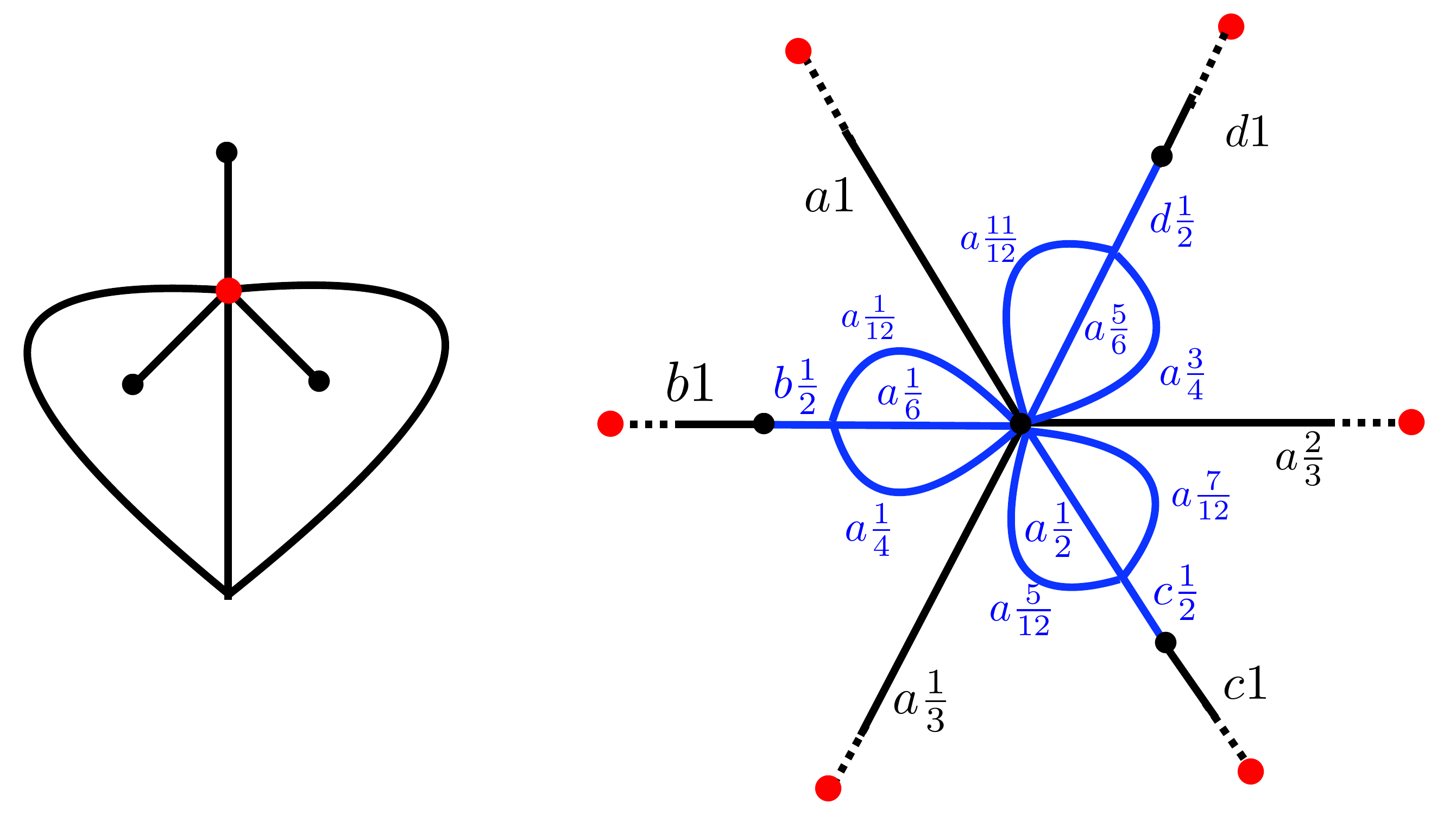} 
 \caption{A type of channel graph $\Delta_0$ (left) and its resulting invariant graph $G$ (right). 
 Here the red dots represent the same point $\infty$. The notation `$w t$' means the internal ray of angle 
 $t$ starting from the center $w\in\{a,b,c,d\}$. 
 }
\label{graph-final1}
\end{center}
\end{figure}


%
%
%

	  \begin{figure}[h] 
	 \begin{center}
\includegraphics[height=5.4cm]{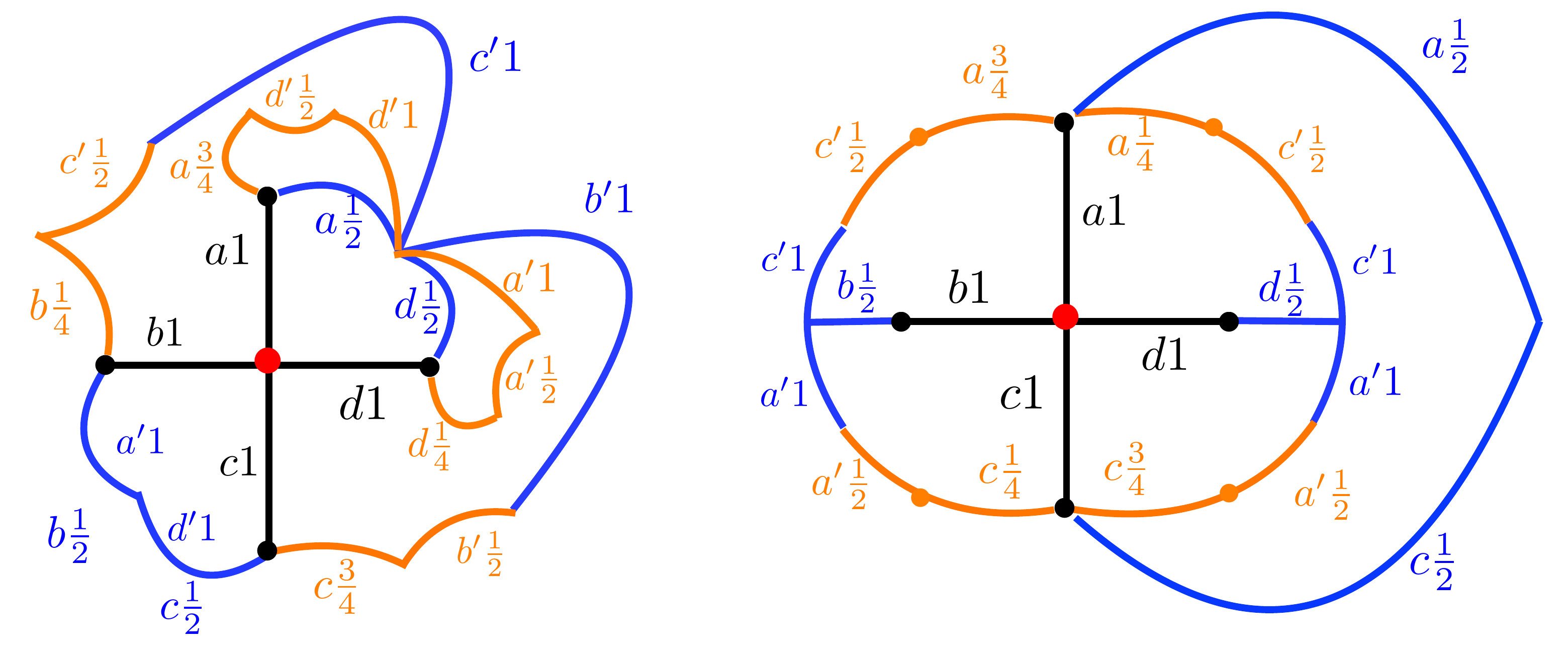} 
 \caption{They are two invariant graphs, for two different maps.
 Their channel graphs are same, but the combinations are different.
  The notation `$w t$' means the internal ray of angle 
 $t$ starting from the center $w\in\{a,b,c,d\}\cup \{a',b',c',d'\}$, here 
 $x'$ is a pre-image of $x$.
 The blue/orange parts are the subsets of  first/second pre-images of the channel graph.
 }
\label{2graphs}
\end{center}
\end{figure}


\begin{remark}\label{rmk:fixed}
Proposition \ref{cor:non-cutpoints_graph} implies that 
the only possible cut points in $G$ are the centers of 
strictly pre-periodic components $B\in {\rm Comp}(B_f)$.
%
\end{remark}

To prove the connectivity of $f^{-1}(G)$, it is equivalent to show that each component of 
$f^{-1}(\wh{\mb{C}}\setminus G)$ is simply connected. 

To this end, let  $X=\olC\setminus\bigcup_{B}\overline{\Phi_{B}^{-1}(\mb{D}_{1/2})}$,
where the union is taken over all  $B\in {\rm Comp}(B_f)$ such that $B\cap G\neq \emptyset$, and
$\Phi_B:B\to \mb{D}$
 is the B\"ottcher map of $B$.
 By Proposition \ref{cor:non-cutpoints_graph}, each component of $X\setminus G$ is a Jordan disk. 
To show  that each component of 
$f^{-1}(\wh{\mb{C}}\setminus G)$ is simply connected, it is equivalent to show that  each component of 
$f^{-1}(X\setminus G)$ is simply connected. This is the main task of Proposition \ref{map-puzzle},
which is stated as an independent important property for puzzle pieces (here we deal with $f^{-1}(X\setminus G)$ instead of $f^{-1}(\wh{\mb{C}}\setminus G)$ for the technical reason that we want to apply 
Corollary \ref{cor:counting_number}).

This completes proof of Step 6, hence the whole proof of Theorem \ref{in-graph}. 
 \end{proof}

\begin{remark}\label{graph-N} One may assume the number $N$ in  Theorem \ref{in-graph} is minimal, in the sense that 
$G\subseteq f^{-N}(\Delta_0)$ and $G \nsubseteq  f^{-N+1}(\Delta_0)$.
This minimal $N$ can not be controlled by the degree of $f$, even in the cubic case.

In fact, we can show: For any integer $n\geq 1$, there is a post-critically finite cubic Newton map $f$, for which the invariant graph $G$ constructed in Theorem \ref{in-graph} 
satisfies that 
	$$G\subseteq f^{-n}(\Delta_0), \ G \nsubseteq  f^{-n+1}(\Delta_0).$$
The proof is based on the deeper understanding of the parameter space \cite{RWY17}.
Since we will not use this fact in the paper, we skip its proof.
\end{remark}

\subsection{Appendix: Shrinking Lemma revisited}

At the end of this section, we prove a version of Shrinking Lemma (see \cite{LM} for its original form), which plays an important role
 in the proof of  Lemma \ref{integer-existence}.

\begin{lemma}\label{lem:shinking}
	Let $f$ be a rational map. Let $\{(E_n,U_n)\}_{n\geq 0}$ be a sequence of subsets in $\olC$, such that for all $n\geq 0$,
	\begin{enumerate}
		\item $E_n\Subset U_n$ with $E_n$'s full  \footnote{ A set is said \emph{full} if its complement is connected.
} continua and $U_n$'s open sets;	\item $f(E_{n+1})=E_n, f(U_{n+1})=U_n$;
		\item $U_{n+1}\cap U_0=\es$.
	\end{enumerate}
	Then the spherical diameter of $E_n$ converges to zero as $n\to\infty$.
\end{lemma}	

\begin{proof}
	First observe that the sets $U_n$'s are pair-wisely disjoint. If not, assume $U_{n_1}\cap U_{n_2}\neq\es$ for some $0\leq n_1<n_2$. Then we have $\es\neq f^{n_1}(U_{n_1}\cap U_{n_2})\subseteq U_{0}\cap U_{n_2-n_1}$, which contradicts (3). 
	Thus by ignoring finitely many pairs $(E_n,U_n)$, one may assume that
	$\bigcup_{n\geq 0}U_n$ does not contain the critical values of $f$. Since $E_0$ is full and $E_0\Subset U_0$, we can choose a topological disk $D_0$ such that $E_0\Subset D_0\Subset U_0$. Then 
	for each $n$, the unique component $D_n$ of $(f^{n}|_{U_n})^{-1}(D_0)$, which contains $E_n$, is a topological disk. Moreover, the map $f^n:D_n\to D_0$ is conformal, whose inverse is denoted by $g_n$. Then $\{g_n\}$ forms a normal family.
    
 We  claim that the limit map $g_\infty$ of any convergent subsequence $\{g_{n_k}\}$ is a constant map. If not, then  $g_{\infty}(D_0)$ is an open subset of $\olC$. Therefore, for any sufficiently large integers $k\neq k'$, the images $g_{n_{k}}(D_0)(=D_{n_k})$ and $g_{n_{k'}}(D_0)(=D_{n_{k'}})$ will overlap, which is impossible. 
   
	Finally, if $\lim_n\tu{diam}(E_n)\rightarrow 0$ is not true, then 
	there is a constant $\epsilon>0$ and a subsequence  $\{E_{l_k}\}$with  $\tu{diam}(E_{l_k})\geq \epsilon$. This is impossible,
  because by passing to a further subsequence, the maps $g_{l_k}$'s  converge uniformly on $E_0$ to a constant.
	\end{proof}

\section{ Brannar-Hubbard-Yoccoz puzzle}\label{sec:puzzle}

In this section, we develop the Brannar-Hubbard-Yoccoz puzzle theory for Newton maps,  using the invariant graph given by the preceding section. 

\subsection{Puzzles and ends}
Let $G$ be the graph given by Theorem \ref{in-graph}. 
Recall that $\Phi_B:B\to \mb{D}$
 is the B\"ottcher map of $B\in {\rm Comp}(B_f)$. Let
$$X=\olC\setminus\bigcup_{B}\overline{\Phi_{B}^{-1}(\mb{D}_{1/2})}$$
where the union is taken over all  $B\in {\rm Comp}(B_f)$ such that $B\cap G\neq \emptyset$. 
Clearly $f^{-1}(X)\subseteq X$. 
For any  integer $n\geq 0$, let $\mc{P}_n$ be the collection of all connected components of $f^{-n}(X\setminus G)$.
An element  $P\in \mc{P}_n$ is called a {\it puzzle piece} of depth (or level) $n\geq 0$.
Note that two distinct puzzle pieces $P,Q$ are either disjoint (i.e. $P\cap Q=\emptyset$) or nested (i.e. $P\subseteq Q$ or $Q\subseteq P$).

An important fact about  puzzle pieces is as follows:
\begin{prop}\label{map-puzzle} Let $P,Q$ be two  puzzle pieces with 
$Q=f(P)$.
	 Then we have the following two implications:
	 
	 1. $Q$ is a Jordan disk $\Longrightarrow P$ is a Jordan disk.
	 
	 2. $P\subseteq Q \Longrightarrow  \infty\in \partial P\cap \partial Q$ and 
	   $f: P\rightarrow Q$ is conformal.
%
\end{prop}

\begin{proof} 
Let $l\geq1$ be the depth of $P$. Note that there is a unique puzzle piece of depth $l-1$, say $S$, containing $P$.
 The filled closure $\wh{P}_S$ of $P$ with respect to $S$ contains at most one fixed point, which can only be $\infty$ on its boundary.

	      
	    To prove the two implications,  we discuss the relation of $Q$ and $S$:
	    
	    \textit{Case 1: $Q=S$ or equivalently $P\subseteq Q$.}
	  Applying Corollary \ref{cor:counting_number} to the case  $(D,U)=(Q, P)$,
	  we have
	  $$1\geq\#\tu{Fix}(f|_{\wh{P}_Q})=\sum_{V\subseteq \wh{P}_Q, f(V)=Q}\tu{deg}(f|_{\partial V})\geq \tu{deg}(f|_{\partial P})\geq 1.$$
	   This  implies that $\infty\in \partial P\cap \partial Q$, 
		$\wh{P}_S=\ol{P}$, and $f: P\rightarrow Q$ is conformal. In this case,
		we also have the first implication.

%
%
%
%

\textit{Case 2: $Q\neq S$ or equivalently $Q\cap S=\emptyset$.} In this case, we only need prove 1.  
	 Assume that $Q$ is a Jordan disk.  If $P$ is not a Jordan disk, then 
	  $\wh{P}_S\setminus \ol{P}$ is non-empty, furthermore, it contains at least a component $V$ of  $f^{-1}(W)$ with $W:=\olC\setminus \overline{Q}$. Clearly $V\subseteq W$.
	  Applying Corollary \ref{cor:counting_number} to the case  $(D,U)=(W, V)$, we know that the filled closure $\wh{V}_W (\subseteq \wh{P}_S\subseteq S)$ contains fixed points, which must be $\infty$. 
	  Therefore we have $$\infty\in\partial V\cap \partial P\cap\partial S\cap\partial Q.$$
	 On the other hand,  the local behavior of $f$ near $\infty$ implies that in a neighborhood $N(\infty)$ of $\infty$, we have $P\cap N(\infty)\subseteq f(P\cap N(\infty))$. It follows that $Q=S$. This is a contradiction.
	 \end{proof}

\begin{lemma}\label{lem:puzzles}
The puzzle pieces satisfy the following properties: 

1. Each puzzle piece is a Jordan disk;

2. For any puzzle piece $P$ and any immediate root basin $B$,  the  intersection $\overline{P}\cap \partial{B}$  is connected (caution: if $B\in{\rm Comp}(B_f)$ is not an immediate root basin,
then $\overline{P}\cap \partial{B}$ might be disconnected);

3.  For any puzzle piece $P$, the  intersection $\overline{P}\cap J(f)$  is connected.

\end{lemma}
\begin{proof}
	1. By Proposition \ref{cor:non-cutpoints_graph}, each puzzle piece of depth $0$ is a Jordan disk.
	By Proposition \ref{map-puzzle} and induction, all puzzle pieces are Jordan disks.
	
	2. By Proposition \ref{cor:non-cutpoints_graph}, the set $B\cap P$ is some sector $S_B(\theta,\theta';r)$, and   $\overline{P}\cap \partial{B}=\bigcap_{0<s< 1}\ol{S_B(\theta,\theta'; s)}$, which is connected.

3. By Proposition \ref{cor:non-cutpoints_graph},  if $B\in{\rm Comp}(B_f)$ satisfies $B\cap P\neq \emptyset$ and
$B\nsubseteq P$, then $B\cap P$ is the union of finitely many sectors $S_B(\theta,\theta';r)$ (The reason is that the center $c_B$ of $B$ might be a cut point. In this case, $\overline{P}\cap \partial{B}$ is a union of finitely many connected set).
Note that $ J(f)\cap\ol{S_B(\theta,\theta';r)}$ is connected, because $ J(f)\cap\ol{S_B(\theta,\theta';r)}=\bigcap_{0<s< 1}\ol{S_B(\theta,\theta'; s)}$. We aim to the show that any two points $z_1,z_2\in\ol{P}\cap J(f)$ are contained in a connected subset $C\subseteq J(f)\cap \ol{P}$. Let $\gamma$ be a Jordan arc in ${P}$ connecting $z_1$ and $z_2$.
Then $\gamma\cap F(f)$ consists of countably many open segments $\{\gamma_i\}_{i\in\Lambda}$.
For each $\gamma_i$,  if there is $B_0\in{\rm Comp}(B_f)$ so that
 $\gamma_i\subseteq B_0\subseteq P$, we set $C_i=\partial B$; otherwise, $\gamma_i$ is contained in some sector $S_B(\theta,\theta';r)$, and we set $C_i= J(f)\cap \ol{S_B(\theta,\theta';r)}$. The set $C=(\gamma\cap  J(f))\cup (\cup_{i\in\Lambda}\,C_i)$ is a connected subset of $ J(f)\cap \ol{P}$
 connecting $z_1$ and $z_2$.	\end{proof}

It's worth observing that the number of unbounded puzzle pieces of depth $n$ is independent of $n$.
This number is $d_0=\sum_{B}(\tu{deg}(f|_B)-1)$, where the sum is taken over all immediate root basins $B$'s.
 Let 
 $\mc{P}_n^\infty=\big\{P_{n,1}^{\infty},\cdots,P_{n,d_0}^{\infty}\big\}$ be the set of all unbounded puzzle pieces of depth $n$, numbered in the way that  $P_{n+1,k}^\infty\subseteq P_{n,k}^\infty$, for any $n\geq 0$ and $1\leq k\leq d_0$.
Clearly, the sets 
$$Y_n(\infty)=\ol{P_{n,1}^{\infty}}\cup\cdots\cup \ol{P_{n,d_0}^{\infty}},\ n\geq0$$
 are closed neighborhoods of $\infty$. 
The {\it grand orbit} of $\infty$ is  denoted by
$$\Omega_f=\bigcup_{k\geq 0}f^{-k}\{\infty\}.$$
For any $z\in \Omega_f$, let's define
 $$\mc{P}_n^z=\big\{P\in \mc{P}_n; z\in \ol{P}\big\},  \ Y_n(z)=\bigcup_{P\in \mc{P}_n^z}\ol{P}. $$ 
 
For any point $z\in \olC-B_f\cup \Omega_f$, its orbit avoids the graph $G$, therefore the puzzle piece of depth $k\geq0$ containing $z$ is well-defined, and is denoted by $P_{k}(z)$. 
For $z\in \Omega_f$, let $P_k(z)$ be the interior of $Y_k(z)$. In this way, for all $z\in \olC-B_f$ and all $k\geq0$, the piece 
$P_k(z)$ is well-defined.

For any $z\in \olC-B_f$, the {\it end} of $z$, denoted by 
${\mathbf e}(z)$, is defined by 

$${\mathbf e}(z)=\bigcap_{k\geq 0}\ol{P_k(z)}.$$
%

\begin{prop}  \label{prop:local_connect_infty} 
For any $z\in \olC-B_f$ and any integer $k\geq 0$,  there is an integer $n_k=n_k(z)>0$ with the property:
 $$P_{k+n_k}(z) \Subset P_k(z).$$
 This implies, in particular, that  $\mathbf{e}(z)=\{z\}$ for any $z\in \Omega_f$.

\end{prop}

%
%
%
\begin{proof} 
We first consider $z\in \Omega_f$. In this case, there is an integer $N\geq 0$ with $f^N(\mathbf{e}(z))=\mathbf{e}(\infty)$. To show the statement, it suffices to show $\mathbf{e}(\infty)=\{\infty\}$.

By Proposition \ref{map-puzzle}, for each $n\geq 1$, the map $f^n: P_n(\infty)\rightarrow P_0(\infty)$ is conformal, 
and the boundaries $\partial Y_{n}(\infty), \partial Y_{0}(\infty)$  are Jordan curves.
Therefore $f^n: \partial Y_{n}(\infty)\rightarrow \partial Y_{0}(\infty)$ is homeomorphism.
We claim that $Y_{N}(\infty)\Subset Y_0(\infty)$ for some large $N$. 
%
In fact, if $\partial Y_n(\infty)\cap \partial Y_0(\infty)\neq \emptyset$ for all $n\geq 1$, then 
the relation $Y_{n+1}(\infty)\subseteq Y_{n}(\infty)$ implies that 
$$ \partial Y_{n+1}(\infty)\cap \partial Y_0(\infty)\subseteq \partial Y_n(\infty)\cap \partial Y_0(\infty).$$
Therefore 
$$\bigcap \partial Y_n(\infty)\neq \emptyset \text{ and } \bigcap \partial Y_n(\infty)\subseteq \Omega_f \subseteq J(f).$$
Take $p\in \bigcap \partial Y_n(\infty)$ and suppose $f^{n_0}(p)=\infty$. Clearly $p\neq \infty$. 
This contradicts the fact that  $f^{n_0}: \partial Y_{n_0}(\infty)\rightarrow \partial Y_{0}(\infty)$  is a homeomorphism. 

%

By the claim and
applying the Schwarz Lemma to the inverse of $f^N: Y_{N}(\infty)\rightarrow Y_0(\infty)$, we have that
$\mathbf{e}(\infty)=\bigcap_k Y_{Nk}(\infty)=\{\infty\}$.

For those $z\in\olC-(B_f\cup\Omega_f)$, the idea of the proof is same as above. If there is an integer $k_0\geq 0$ such that $\partial P_{k_0}(z)\cap \partial P_{k_0+l}(z)\neq \es$ for all $l>0$,
then the nested property (i.e., $P_{k_0+l+1}(z)\subseteq P_{k_0+l}(z)$) gives that 
	$$\partial P_{k_0}(z)\cap \partial P_{k_0+l+1}(z)\subseteq \partial P_{k_0}(z)\cap \partial P_{k_0+l}(z).$$
	Therefore 
	$$\es\neq \bigcap_{l\geq 1}(\partial P_{k_0}(z)\cap \partial P_{k_0+l}(z))= \bigcap_{l\geq 0}\partial P_{k_0+l}(z)\subseteq\Omega_f\cap J(f).$$
It follows that the puzzle pieces  $\{P_{k_0+l}(z)\}_{l\geq 0}$ have a common boundary point $\xi$ with $f^{m}(\xi)=\infty$ for some $m\geq0$. Applying the 
  $f^{m}$-action on these puzzle pieces, we get
$$\infty\in\ol{P_{k_0-m+l}(f^{m}(z))}\subseteq {Y_{k_0-m+l}(\infty)}, \ \forall  \ l\geq m.$$
This gives that $\infty\in \mathbf{e}(f^{m}(z))\subseteq \mathbf{e}(\infty)$.
By the proven fact 
$\mathbf{e}(\infty)=\{\infty\}$, we have  $f^{m}(z)=\infty$. This contradicts the assumption
$z\in\olC-(B_f\cup\Omega_f)$. 
\end{proof}

We collect some facts about ends as follows:

\begin{itemize}
    \item  ${\mathbf e}(z)$ is either a singleton or a full continuum in $\olC$;
    \item $f(\mathbf  e(z))=\mathbf  e(f(z))$;
	\item For any $z'\in\olC-B_f$ with $z'\neq z$, based on the  proven fact $\ee(q)=\{q\}$ for any $q\in \Omega_f$ (see 
	Proposition \ref{prop:local_connect_infty}), we have that either 
	$${\mathbf e}(z')={\mathbf e}(z) \text{ or }\mathbf e (z')\cap\mathbf e(z)=\emptyset.$$
	\item By Lemma \ref{lem:puzzles}, $\mathbf e (z)=\{z\}$ implies the local connectivity of $J(f)$ at $z$.
	For any immediate root basin $B$ and any $z\in \partial B$,  the fact $\mathbf  e(z)\cap \partial B=\{z\}$ implies the  local connectivity of $\partial B$ at $z$.
%
\end{itemize}

 Let $\mc{E}=\{\ee(z); z\in \olC-B_f\}$ be the collection of all ends.  An end is \emph{trivial} if it is a singleton.
 An end $\ee$ is called \emph{critical} if it contains a critical point of $f$. 
 The {\it orbit} $\tu{orb}(\ee)$ of an end $\ee\in \mc{E}$ is $\tu{orb}(\ee)=\{f^k(\ee)\}_{k\geq 0}$. 

 An end $\ee$ is \emph{pre-periodic} if $f^{m+n}(\ee)=f^m(\ee)$ for some $m\geq 0,n\geq 1$. In particular, $\ee$ is called \emph{periodic} if $m=0$. If there is no such $m,n$, then $\ee$ is called \emph{wandering}. In this case, its {orbit} $\tu{orb}(\ee)$ has infinitely many elements. 
 
 For each end $\ee=\ee(z)$ with $z\in \olC-B_f$, let $P_n(\ee)=P_n(z)$. It follows from Proposition \ref{prop:local_connect_infty} that $P_n(\ee)$  is the puzzle piece of depth $n$ containing $\ee$. 
	
	Let $(\ee_k)_{k\in\mb N}$ be a sequence of wandering ends with distinct entries $\ee_k$'s, the \emph{combinatorial accumulation set}
	$\mc{A}((\ee_k)_{k\in\mb N})$ consists of
	 the ends $\ee'\in\mc{E}$, such that for any integer $n>0$, 
	 the index set $\{k\in \mathbb{N}; \ee_k\subseteq P_n(\ee')\}$ is infinite. 
	 
	  \begin{lemma}\label{comb-accum}
	$\mc{A}((\ee_k)_{k\in\mb N})\neq \emptyset$.
\end{lemma}	

 \begin{proof} For any $n\geq 0$, recall that the collection $\mc{P}_n$ of puzzle pieces of depth $n$ is a finite set.
	 We define the index set $I_n$ and the puzzle piece $P_n\in \mc{P}_n$ inductively as follows. 
	 First, there is a puzzle piece $P_0\in\mc{P}_0$ such that 
	 $$I_0=\{k\in\mb{N}; \ \ee_k\subseteq P_0\}$$
	 is an infinite set. Suppose that we have constructed the infinite index set $I_j$ and the puzzle piece $P_j\in \mc{P}_j$ for 
	 $0\leq j\leq \ell$, 
	 satisfying that
	 $$I_0\supseteq \cdots \supseteq I_\ell,\  P_0\supseteq \cdots \supseteq P_\ell.$$
	 Then one can find $P_{\ell+1}\in \mc{P}_{\ell+1}$  with $P_{\ell+1}\subseteq P_{\ell}$, such that the index set 
	 $$I_{\ell+1}=\{k\in I_{\ell}; \ \ee_k\subseteq P_{\ell+1} \}$$
	 is an infinite set. Now let's define $\ee'=\bigcap_n \ol{P_n}$. 
	 
	 To finish, we show  $\ee'\in\mc{E}$, which implies that
	 $\ee'\in \mc{A}((\ee_k)_{k\in\mb N})$. To  this end, we discuss two cases.
	 If $\ee'\cap \Omega_f\neq \emptyset$, we
	 take $z\in \ee'\cap \Omega_f\neq \emptyset$, then the fact 
	 $\{z\}\subseteq \ee' =\bigcap_n \ol{P_n} \subseteq \bigcap_n Y_n(z)=\{z\}$ (by Proposition \ref{prop:local_connect_infty}) implies that 
	 $\ee'=\{z\}=\ee(z)$.  If $\ee'\cap \Omega_f=\emptyset$, we take $z\in \ee'$, then
	 $ \ee' =\bigcap_n \ol{P_n} =\bigcap_n \ol{P_n(z)}=\ee(z)$. In either case, we have $\ee'\in \mc{E}$, completing the proof. 
\end{proof}
	
	The \emph{combinatorial limit set} $\omega(\ee)$ of  a wandering end $\ee\in \mc{E}$ is  defined by
	$$\omega(\ee)=\mathcal{A}((f^k(\ee))_{k\in\mathbb N}).$$
	
	 One may verify that $\omega(\ee)$ satisfies the following properties:
	 \begin{itemize}
		\item  $\omega(f(\ee))=\omega(\ee)$.
		
		\item $f(\omega(\ee))\subseteq \omega(\ee)$.
		
		\item For any wandering end $\ee'\in \omega(\ee)$, we have $\omega(\ee')\subseteq \omega(\ee)$.
\end{itemize}

The first two follow from the definition of $\omega(\ee)$. 
We only verify the third one.
Let $\ee'\in \omega(\ee)$ be a wandering end, and take $\ee''\in \omega(\ee')$. By definition, for any $n\geq0$, the
	  index set $J_n=\{k\in\mb{N}; f^{k}(\ee')\subseteq P_n(\ee'')\}$ is infinite. 
	  For $k\in J_n$, note that $f^k(\ee')\in\omega(\ee)$, 
	  this implies that the index set $\{t\in\mb{N}; f^{t}(\ee)\subseteq P_n(f^k(\ee'))=P_n(\ee'')\}$ is infinite.
	  Therefore $\ee''\in\omega(\ee)$.


\begin{prop}
\label{lem:near_infinity} Let $L>0$ be an integer with $Y_L(\infty)\Subset Y_0(\infty)$.
Let   $\ee$ be a  wandering end with
 $\ee\subseteq Y_{L}(\infty)$, then
 there is an (minimal) integer $s=s(\ee)\geq 0$ with the following property:   
 $$P_{L+1}(f^{s}(\ee))\Subset P_{0}(f^{s}(\ee))\in \mc P_{0}^\infty$$
 and $f^s: P_{L+s+1}(\ee)\to P_{L+1}(f^{s}(\ee))$ is conformal.
%
%
%
%
\end{prop}

\begin{proof}
Recall that $Y_k(\infty)=\bigcup_{j}\,\ol{P_{k,j}^{\infty}}$ and $P_k(\infty)$ is the interior of $Y_k(\infty)$. By Propositions \ref{map-puzzle}, \ref{prop:local_connect_infty}, 
for any $k\geq 0$, the map $f: P_{k+1}(\infty)\rightarrow P_{k}(\infty)$ is one-to-one.
The assumption 
$$\ee\subseteq Y_{L}(\infty)=\bigcup_{s\geq 0}(Y_{L+s}(\infty)\setminus Y_{L+s+1}(\infty))$$
 implies  that $\ee\subseteq Y_{L+s}(\infty)\setminus Y_{L+s+1}(\infty)$ for some integer $s\geq 0$.
 Then we can find an index $j$ with
  $\ee\subseteq P_{L+s, j}^{\infty}\setminus P_{L+s+1, j}^{\infty}$. Since for any $k\geq0$,
the map $f: P_{k, j}^{\infty}\setminus P_{k+1, j}^{\infty}\rightarrow P_{k-1, j}^{\infty}\setminus P_{k, j}^{\infty}$ is a homeomorphism,
we have that $f^{s}(\ee)\subseteq P_{L, j}^{\infty}\setminus P_{L+1, j}^{\infty}$.
Hence $P_{L+1}(f^{s}(\ee))$ is bounded.

  \begin{figure}[h] 
	 \begin{center}
\includegraphics[height=5cm]{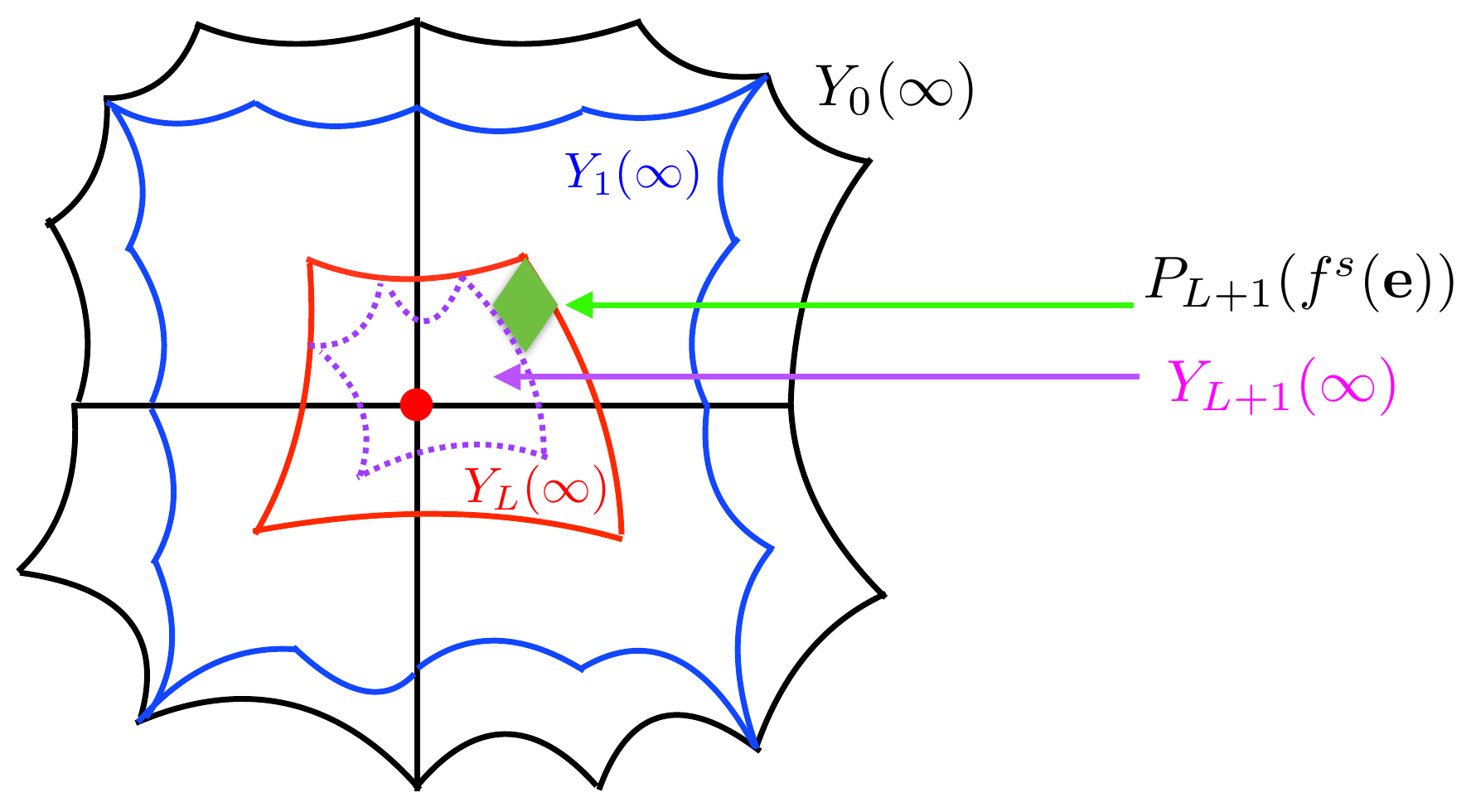} 
 \caption{ The puzzle pieces around $\infty$ (red dot). 
 It may happen that $\partial Y_1(\infty)\cap \partial Y_0(\infty)\neq\emptyset$. 
 Here $Y_L(\infty)\Subset Y_0(\infty)$ for some $L\geq1$.
 Assume $\ee\subseteq Y_L(\infty)$, then
 $P_{L+1}(f^s(\ee))\subseteq Y_L(\infty)\setminus Y_{L+1}(\infty)$ for some minimal integer $s\geq0$.
 Moreover $P_{L+1}(f^s(\ee))\Subset P_{0}(f^s(\ee))$.
 }
\label{puzzle}
\end{center}
\end{figure}

By the structure of unbounded puzzle pieces, we see that  $\overline{P_{L+1}(f^{s}(\ee))}$ is disjoint from 
$\Delta_0\cup \partial Y_{0}(\infty)$ (see Figure \ref{puzzle}). 
This implies that 
$$P_{L+1}(f^{s}(\ee))\Subset P_{0}(f^{s}(\ee))=P_{0, j}^{\infty}\in\mc P_{0}^\infty.$$
\end{proof}

\subsection{Strategy of the proof}

To prove our main Theorem \ref{main}, it suffices to show that for 
any immediate root basin $B$, we have 
$$\ee(z)\cap \partial B=\{z\}, \ \forall z\in \partial B.  \ \ \ \ (*)$$

To this end, we need first classify all ends in $\mc{E}$ into two types: wandering ones and pre-periodic ones, which are denoted by
$\mc{E}_{\rm{w}}$  and $\mc{E}_{\rm{pp}}$, respectively.

%

The  set $\mc{E}_{\rm{w}}$ of
 wandering ends  has a further decomposition:   
$$\mc{E}_{\rm{w}}=
 \mc{E}_{\rm{w}}^{\rm pp}\sqcup \mc{E}_{\rm{w}}^{\rm nr}\sqcup \mc{E}_{\rm{w}}^{\rm r},$$
where
%
%
%

%
%
%

\begin{enumerate}
	\item[] $\mc{E}_{\rm w}^{\rm pp}=\{\ee \in\mc{E}_{\rm{w}};  \mc{E}_{\rm pp} \cap \omega(\ee)\neq \emptyset\}$;
	 
	\item[]  $\mc{E}_{\rm w}^{\rm nr}=\{\ee \in\mc{E}_{\rm{w}}; \mc{E}_{\rm pp} \cap \omega(\ee)=\emptyset
	\text{ and } \omega(\ee)\neq\omega(\ee') \text{ for some  }\ee'\in\omega(\ee)\}$;
	
	\item[]  $\mc{E}_{\rm w}^{\rm r}=\{\ee \in\mc{E}_{\rm{w}}; \mc{E}_{\rm pp} \cap \omega(\ee)=\emptyset
	\text{ and } \omega(\ee)=\omega(\ee') \text{ for all  }\ee'\in\omega(\ee)\}$.
\end{enumerate}

The proof of the statement $(*)$
 will be carried out in the following two sections. 
 In Section \ref{wandering}, we prove a stronger fact that any wandering end is a singleton. 
 In Section \ref{pre-periodic}, we prove that for any pre-periodic end $\ee$,
  the intersection $\ee\cap \partial B$ is either empty or a singleton.
These two cases cover all situations.

In the rest of the paper, let $\mathcal{E}_{\rm crit}\subseteq \mc{E}$ be the collection of all critical ends. Set $\kappa=\# \mathcal{E}_{\rm crit}$. Recall that $d$ is the degree of the Newton map $f$.

\section{Wandering ends are trivial}\label{wandering}

In this section, we show that any wandering end is a singleton. 
The proof is based on the following dichotomy: for any wandering end $\ee$, either

\begin{itemize}
		\item $\ee$ satisfies the {\it bounded degree} property, or 
		\item  $\omega(\ee)$ contains a {\it persistently recurrent} critical end.
	\end{itemize}
 
 The treatments of these two situations are different. 



\subsection{Bounded degree property implies triviality of ends}

\begin{definition}\label{def:star_property}
	An end $\ee$ is said to have bounded degree (BD for short) property, if there exist puzzle pieces $\{P_{n_k}(\ee)\}$, with $n_k\to\infty$ as $k\to\infty$, and an integer $D$, such that
	$$
	\tu{deg}(f^{n_k}:P_{n_k}(\ee)\to P_0(f^{n_k}(\ee)) )\leq D, \ \forall \ k\geq1.   \qquad (\star)$$
\end{definition}
\begin{prop} 
\label{lem:star_property}
	A  wandering end $\ee$ 
	with BD property  is trivial.
\end{prop}
\begin{proof}
	By assumption, there is a sequence of puzzle pieces $\{P_{n_k}(\ee)\}$ satisfying $(\star)$. 
	The combinatorial accumulation set $\mathcal{A}((f^{n_k}(\ee))_{k\in\mathbb N})$ 
	 of the sequence $(f^{n_k}(\ee))_{k\in\mathbb N}$ satisfies 
	$$\emptyset\neq \mathcal{A}((f^{n_k}(\ee))_{k\in\mathbb N})\subseteq \omega(\ee).$$
	Take $\ee_0\in \mathcal{A}((f^{n_k}(\ee))_{k\in\mathbb N})$, 
	 note that for any $n\geq0$, the index set $\{k\in \mathbb{N}; f^{n_k}(\ee)\subseteq P_n(\ee_0)\}$ is infinite.
	 
	 To prove the proposition, we need discuss two cases:
	 
	
\vspace{4 pt}

\textbf{Case 1: $\ee_0\notin  \{\ee(z); z\in \Omega_f\}$}. 

\vspace{4 pt}
	
	
In this case,	by  Proposition \ref{prop:local_connect_infty},
	there is an integer $L_0>0$ such that $P_{L_0}(\ee_0)\Subset P_0(\ee_0)$. 
	By passing to a subsequence, we may assume 
	 $f^{n_k}(\ee)\subseteq P_{L_0}(\ee_0)$ for all $k\geq 1$. By pulling back the triple $\left(f^{n_k}(\ee),P_{L_0}(\ee_0), P_0(\ee_0)\right)$ along the orbit $\ee\mapsto f(\ee)\mapsto\cdots\mapsto f^{n_k}(\ee)$,  we get the non-degenerate annuli $P_{n_k}(\ee)\setminus \overline{P_{L_0+n_k}(\ee)}$'s,
	 whose moduli satisfy
	$${\rm mod}(P_{n_k}(\ee)\setminus \overline{P_{L_0+n_k}(\ee)})\geq \frac{1}{D}{\rm mod }(P_0(\ee_0)\setminus \overline{P_{L_0}(\ee_0)}),
	\forall \ k\geq1.$$
	This implies that $\ee=\bigcap \overline{P_k(\ee)}$ is a singleton.
	
	\vspace{4 pt}

\textbf{Case 2: $\ee_0\in  \{\ee(z); z\in \Omega_f\}$}. 

\vspace{4 pt}
	
In this case,	replacing 
	$(f^{n_k}(\ee))_{k\in\mathbb N}$ by the new sequence $(f^{n_k+l}(\ee))_{k\in\mathbb N}$ (here $l\geq 0$ is some integer) if necessary, we may assume 
	$\ee_0=\ee(\infty)$. 
	Recall that $Y_n(\infty)=\bigcup_{k}\,\ol{P_{n,k}^{\infty}}$ and $P_n(\infty)$ is the interior of $Y_n(\infty)$.
	Let $L>0$ be an integer with $Y_{L}(\infty)\Subset Y_{0}(\infty)$.
	

By choosing subsequence of $\{n_k\}_k$, we may assume  that 
$$f^{n_k}(\ee)\subseteq Y_{L}(\infty) \text{ and  } \tu{deg}(f^{n_k}:P_{n_k}(\ee)\to P_{0,m}^\infty)\leq D$$
with $P_{0,m}^\infty\in \mc P_0^\infty$ and $P_{0,m}^\infty=P_0(f^{n_k}(\ee))$, for all $k\in\mb N$.
%

 For each $k$, the assumption $f^{n_k}(\ee)\subseteq Y_{L}(\infty)=\bigcup_{s\geq 0}(Y_{L+s}(\infty)\setminus Y_{L+s+1}(\infty))$ implies that 
 there is a unique integer $s_k\geq0$ such that  
 $f^{n_k}(\ee)\subseteq Y_{L+s_k}(\infty)\setminus Y_{L+s_k+1}(\infty)$.
 The behavior of $f$ near $\infty$ gives that $P_0(f^{n_k+j}(\ee))\equiv P_{0,m}^\infty$ for all $0\leq j\leq s_k$.
By Proposition \ref{lem:near_infinity}, we have 
$$P_{L+1}(f^{n_k+s_k}(\ee))\Subset P_{0}(f^{n_k+s_k}(\ee))\in \mc P_{0}^\infty.$$
 
%
%
%
 We factor the map $f^{n_k+s_k}: P_{s_k+n_k}(\ee)\rightarrow P_{0}(f^{s_k+n_k}(\ee))$ as 
$$P_{n_k+s_k}(\ee) \xrightarrow{f^{n_k}} P_{s_k}(f^{n_k}(\ee)) \xrightarrow{f^{s_k}}  P_{0}(f^{s_k+n_k}(\ee)) (=P_{0,m}^\infty).$$
The first factor has degree at most $D$. For the second factor,
note that $f^{n_k}(\ee)\subseteq Y_{L+s_k}(\infty)\subseteq Y_{s_k}(\infty)$, this implies that 
$ P_{s_k}(f^{n_k}(\ee))=P_{s_k, m}^\infty$.
Hence the map $f^{s_k}: P_{s_k}(f^{n_k}(\ee))\rightarrow   P_{0}(f^{s_k+n_k}(\ee))$ is conformal. 
So the degree of $f^{n_k+s_k}: P_{s_k+n_k}(\ee)\rightarrow P_{0}(f^{s_k+n_k}(\ee))$ is bounded above by $D$.


By pulling back the pair $\big(P_{L+1}(f^{n_k+s_k}(\ee)), P_{0}(f^{n_k+s_k}(\ee)\big)$ along the orbit $\ee\mapsto f(\ee)\mapsto\cdots\mapsto f^{n_k+s_k}(\ee)$ by $f^{n_k+s_k}$,  we get the annuli $A_k=P_{s_k+n_k}(\ee)\setminus \overline{P_{L+1+n_k+s_k}(\ee)}$'s,
whose moduli have a uniform lower bound 
	\bess
	{\rm mod}(A_k) &\geq& \frac{1}{D}{\rm mod }\big(P_{0}(f^{n_k+s_k}(\ee))\setminus \overline{P_{L+1}(f^{n_k+s_k}(\ee))}\big)\\
	&\geq& \frac{1}{D}\min\big\{{\rm mod }(P_{0,m}^\infty\setminus \overline{Q}); 
	Q\in \mc P_{L+1},  Q\Subset P_{0,m}^\infty\big\}.
	\eess
	This implies that $\ee=\bigcap \overline{P_k(\ee)}$ is a singleton.
%
%
%
%
%
%
\end{proof}

Let $\ee$ be a wandering end and $P$ be a puzzle piece. 
	The \emph{first entry time} of $\ee$ into $P$, denoted by $r_\ee(P)$, is the minimal integer $k\geq 1$ such that $f^k(\ee)\subseteq P$. 
	If no such integer exists, we set $r_\ee(P)=\infty$. If $r_\ee(P)\neq\infty$, we denote by $L_\ee(P)$ the unique puzzle piece containing $\ee$ such that $f^{r_\ee(P)}(L_\ee(P))=P$. 
	Clearly, if $P\in\mc P_k$ for some $k$, then $L_\ee(P)\in \mc P_{k+r_\ee(P)}$.
	
	
\begin{lemma}\label{lem:first_entry} Let $\ee$ be a wandering end and $P$ be a puzzle piece. 
Suppose that the first entry time $r=r_\ee(P)$ is finite, then 
	
1.  the $r$ puzzle pieces $L_\ee(P), \cdots,f^{r-1}(L_\ee(P))$ are pair-wisely disjoint;

2.  the degree of $f^r:L_\ee(P)\to f^r(L_\ee(P))=P$ is at most $d^\kappa$;

3.  any puzzle piece $Q$ containing $\ee$ such that $f^{s}(Q)=P$ for some $s\geq1$ is contained in $L_\ee(P)$;

\end{lemma}
\begin{proof}
 Write $Q_k=f^{k}(L_\ee(P))$ for $0\leq k\leq r-1$. 
 
 1. If  $Q_{k_1}\cap Q_{k_2}\neq \es$ for some $k_1<k_2$, then $Q_{k_1}\subseteq Q_{k_2}$. By pulling back $(Q_{k_1},Q_{k_2})$ along the orbit $Q_0\mapsto\cdots\mapsto Q_{r-1}$, we get the pairs $(Q_{k_1-1},Q_{k_2-1})$,
 $\cdots$, $(Q_0,Q_{k_2-k_1})$. It follows that $\ee\subseteq Q_0\subseteq Q_{k_2-k_1}$ and $f^{r-(k_2-k_1)}(Q_{k_2-k_1})=Q$. This obviously contradicts the definition of first entry time.
	
2. It is a direct consequence of 1, since each critical end appears in the orbit
$Q_0\mapsto\cdots\mapsto Q_{r-1}$ 
 at most once.
	
3. If it is not true, we have $s<r$ and $f^{s}(\ee)\subseteq P$. This contradicts the definition of the first entry time $r$.	
\end{proof}

\begin{proposition}\label{prop:wp}
	Any end $\ee\in \mc{E}_{\rm w}^{\rm pp}$ satisfies the BD property.
\end{proposition}
\begin{proof} Let $\ee\in \mc{E}_{\rm w}^{\rm pp}$.  The fact $f(\omega(\ee))\subseteq \omega(\ee)$ implies that
$\omega(\ee)$ contains at least a periodic end, say $\ee_0$. Let $p$ be the period of $\ee_0$.
Observe that $p=1$ if and only if $\ee_0=\ee(\infty)$.

\vspace{4 pt}

\textbf{Case 1: $\ee_0\neq \ee(\infty)$}. 

  Let $N$ be a large integer so that $P_N(f^k(\ee_0))\setminus f^k(\ee_0)$ contains no critical points of $f$, for all $0\leq k<p$. Let $A_n(\ee_0)=P_n(\ee_0)\setminus\overline{P_{n+1}(\ee_0)}$ for all $n\geq 0$.
 By the choice of $N$, for any $n\geq N$,  any puzzle piece  $Q$ in $A_n(\ee_0)$  will be mapped, by some $f^k$, into 
 a puzzle piece in $A_N(\ee_0)\cup\cdots\cup A_{N+p-1}(\ee_0)$ conformally (because the choice of $N$ guarantees that there is no critical points along the orbit of $Q$). 
%
%
	  
	  For each $n>N$, let $r_n$ be the first entry time of $\ee$ into $P_n(\ee_0)$. Clearly  $r_n\to\infty$ as $n\to \infty$, and the degree of $f^{r_n}:L_\ee(P_n(\ee_0))\to P_n(\ee_0)$ is at most $d^\kappa$ (by Lemma \ref{lem:first_entry}). Note that $f^{r_n}(\ee)\subseteq P_n(\ee_0)$ and $f^{r_n}(\ee)\neq\ee_0$, there is a unique integer $s_n\geq 0$ so that 
	  $f^{r_n}(\ee)\subseteq A_{n+s_n}(\ee_0)$.
	  It follows that $P_{n+s_n+1}(f^{r_n}(\ee))\subseteq  A_{n+s_n}(\ee_0)$. So there is a minimal integer $t_n\geq 0$ satisfying that
	  $f^{t_n}(P_{n+s_n+1}(f^{r_n}(\ee)))=
	  P_{n+s_n-t_n+1}(f^{r_n+t_n}(\ee))\subseteq  A_{n+s_n-t_n}(\ee_0)\in\{A_N(\ee_0), \cdots, A_{N+p-1}(\ee_0)\}$, here
	  $$N< n+s_n-t_n+1\leq N+p, \ \forall n\geq N.$$

	We factor the map $f^{r_n+t_n}: P_{n+s_n+r_n+1}(\ee)\rightarrow 
	  P_{n+s_n-t_n+1}(f^{r_n+t_n}(\ee))$ as 
	  $$ P_{n+s_n+r_n+1}(\ee) \xrightarrow{f^{r_n}} P_{n+s_n+1}(f^{r_n}(\ee)) \xrightarrow{f^{t_n}}  P_{n+s_n-t_n+1}(f^{r_n+t_n}(\ee)).$$
	  The former has degree at most $d^\kappa$, while the latter is conformal. Therefore, by choosing a subsequence of $n$'s so that 
	  $n+s_n-t_n+1$ equals constant, we see that $\ee$ satisfies the BD property.	
	
	\vspace{4 pt}  
	  
	  \textbf{Case 2: $\ee_0= \ee(\infty)$}.
	  
	   In this case, for any $n\geq0$, 
	  the index set $\{k\in \mathbb{N}; f^k(\ee)\subseteq P_n(\infty)\}$ is infinite,
	  implying that for some $j$ independent of $n$,
	  the index set $\{k\in \mathbb{N}; f^k(\ee)\subseteq P_{n,j}^\infty\}$ is infinite.
	  For each $n\geq 1$, let $r_n$ be the first entry time of $\ee$ into $P_{n,j}^\infty$.
	  Then the degree of $f^{r_n}: L_\ee(P_{n,j}^\infty)\rightarrow P_{n,j}^\infty$ has upper bound 
	  $d^\kappa$. By post-composing the conformal map $f^n: P_{n,j}^\infty\rightarrow P_{0,j}^\infty$,
	  we see that the degree of $f^{n+r_n}: L_\ee(P_{n,j}^\infty)\rightarrow P_{0,j}^\infty$ is uniformly bounded by $d^\kappa$. Therefore $\ee$ also satisfies the BD property in this case.  
%
\end{proof}

\begin{proposition}\label{prop:wnr}
	Any end $\ee\in\mc{E}_{\rm w}^{\rm nr}$ satisfies the BD property.
\end{proposition}
\begin{proof}
	By definition, there exists $\ee'\in\omega(\ee)$ with $\omega(\ee')\neq \omega(\ee)$. 
	As is pointed out before, $\omega(\ee')\subseteq \omega(\ee)$, hence there is an end $\ee_0\in \omega(\ee)\setminus \omega(\ee')$. For sufficiently large $N$, we have $\tu{orb}(\ee')\cap P_{N}(\ee_0)=\emptyset$.
	
	For any $n$, let $r_n$ be the first entry time of $\ee$ into $P_{n}(\ee')$. Then
	$L_{\ee}(P_n(\ee'))=P_{n+r_n}(\ee)$ and the degree of $f^{r_n}:L_{\ee}(P_n(\ee'))\to P_n(\ee')$  is bounded above by $d^\kappa$ (by Lemma \ref{lem:first_entry}). Clearly $r_n\to\infty$ as $n\to\infty$. 
	
	Since $\ee_0\in\omega(\ee)$, the orbit of $f^{r_n}(\ee)$ will meet  $P_N(\ee_0)$.
	Let $s_n$ be the first entry time $s_n$ of $f^{r_n}(\ee)$ into $P_N(\ee_0)$. Then
	$L_{f^{r_n}(\ee)}(P_N(\ee_0))=P_{N+s_n}(f^{r_n}(\ee))$ and
	 the map $f^{s_n}:L_{f^{r_n}(\ee)}(P_N(\ee_0))\to P_N(\ee_0)$ has degree at most $d^\kappa$. 
	 
	 Note that both $L_{f^{r_n}(\ee)}(P_N(\ee_0))$ and $P_n(\ee')$ contain $f^{r_n}(\ee)$.
	 We claim that $L_{f^{r_n}(\ee)}(P_N(\ee_0))$ is a proper subset of $P_n(\ee')$. Because, otherwise, one has
	  $ P_n(\ee')\subseteq L_{f^{r_n}(\ee)}(P_N(\ee_0))$. This would imply $\tu{orb}(\ee')\cap P_{N}(\ee_0)\neq\emptyset$,  which contradicts our assumption on $P_N(\ee_0)$.
	  
	
	Then we pull back $L_{f^{r_n}(\ee)}(P_N(\ee_0))=P_{N+s_n}(f^{r_n}(\ee))$ along the orbit $\ee\mapsto\cdots\mapsto f^{r_n}(\ee)$ by $f^{r_n}$, and get the puzzle piece $P_{N+r_n+s_n}(\ee)$
	containing $\ee$. Further, the degree of the map
	$$f^{r_n+s_n}: P_{N+r_n+s_n}(\ee)\to P_N(\ee_0)$$
	is at most $d^{2\kappa}$. This implies that $\ee$ has BD property.
\end{proof}

\subsection{The case $\ee\in \mc{E}_{\rm w}^{\rm r}$}\label{sec:wr}

In this part, we will show that any  $\ee\in \mc{E}_{\rm w}^{\rm r}$ is trivial.
By definition of $\mc{E}_{\rm w}^{\rm r}$,  each end $\ee'\in \omega(\ee)$ is  wandering, satisfying that  
$$\omega(\ee')=\omega(\ee)  \text{ and } \ee'\in \omega(\ee').$$

A wandering end $\ee'$ with the property $\ee'\in \omega(\ee')$ is called {\it combinatorially recurrent}.
Clearly, all  ends in $\omega(\ee)$ are combinatorially recurrent.

We first discuss an easy case (Lemma \ref{lem:exceptions1}). 
 Recall  that  $\mathcal{E}_{\rm crit}$ is the set of all critical ends.
Let $\c\in \mathcal{E}_{\rm crit}\cap \mathcal{E}_{\rm w}$ be a critical wandering end. 
A puzzle piece $P_{n+k}(\c)$ with $k\geq 1$ is called a {\it successor} of $P_n(\c)$ if
	\begin{itemize}
		\item $f^k(P_{n+k}(\c))=P_n(\c)$, and
		\item  Each critical end appears at most once along the orbit
		$$P_{n+k}(\c)\mapsto P_{n+k-1}(f(\c))\mapsto\cdots \mapsto P_{n+1}(f^{k-1}(\c)).$$
	\end{itemize}
	

\vspace{4 pt}	
	By definition, if  $P_{n+k}(\c)$ is a successor of $P_{n}(\c)$, 
	then 
	$$\tu{deg}(f^k:P_{n+k}(\c)\to P_n(\c))\leq d^\kappa,$$
	here, recall that $\kappa=\# \mathcal{E}_{\rm crit}$.

\begin{lemma}\label{lem:exceptions1}
	An end $\ee\in \mc{E}_{\rm w}^{\rm r}$ is trivial, if it satisfies one of the following
	
1. $\omega(\ee)\cap \mathcal{E}_{\rm crit}=\emptyset$. 
		
2. Some piece $P_{n_0}(\c)$ of  $\c\in\omega(\ee)\cap \mathcal{E}_{\rm crit}$ has infinitely many successors.
\end{lemma}
\begin{proof} We will show that $\ee$ satisfies the BD property, then the triviality of $\ee$
follows from Proposition \ref{lem:star_property}.

1.  The assumption  $\omega(\ee)\cap \mathcal{E}_{\rm crit}=\emptyset$  implies that 
there is an
integer $N>0$ such that the 
index set
$$\Big\{k\geq 0; f^{k}(\ee)\subseteq\bigcup_{\c\in \mathcal{E}_{\rm crit}} P_N(\c)\Big\}$$
is finite. 
 Let $m$ be the cardinality of this index set.  
 One sees that for each $k\geq 1$, the degree of $f^k: P_{N+k}(\ee)\rightarrow P_{N}(f^k(\ee))$ is bounded above by $d^m$.

2. Let $\{P_{n_k}(\c)\}_{k\geq 1}$ be all successors of $P_{n_0}(\c)$ with $n_1<n_2<\cdots \rightarrow \infty$. By the assumption that $\c\in \omega(\ee)$, for each $k\geq 1$, 
there is a well-defined  first entry time $r_k$ of $\ee$ into  $P_{n_k}(\c)$. Then $L_{\ee}(P_{n_k}(\c))=P_{n_k+r_k}(\ee)$ and the degree of $f^{n_k+r_k-n_0}: P_{n_k+r_k}(\ee)\rightarrow P_{n_0}(\c)$ is 
bounded above by
$$\tu{deg}(f^{r_k}: P_{n_k+r_k}(\ee)\rightarrow P_{n_k}(\c))\cdot
\tu{deg}(f^{n_k-n_0}: P_{n_k}(\c)\rightarrow P_{n_0}(\c)) \leq  d^{2\kappa}.$$
%
%
%
%

We see that $\ee$ satisfies the BD property in both cases.
\end{proof}

By Lemma \ref{lem:exceptions1}, we only need  discuss the ends $\ee\in \mc{E}_{\rm w}^{\rm r}$ satisfying that 
	 $\omega(\ee)\cap \mathcal{E}_{\rm crit}\neq \emptyset$ 
	 and that for any  $\c\in\omega(\ee)\cap \mathcal{E}_{\rm crit}$, and any $n\geq0$, the puzzle piece
 $P_{n}(\c)$ 
 has finitely many successors.  To show the triviality of ends, we first discuss the critical case.
		
A
 critical end $\c\in \mathcal{E}_{\rm crit}\cap \mc{E}_{\rm w}^{\rm r}$ is
 called {\it persistently recurrent} in the combinatorial sense, if it satisfies 
	\begin{itemize}
		\item $\c\in \omega(\c)$, and 
		\item For any $\c'\in \omega(\c)\cap  \mathcal{E}_{\rm crit}$ and any $k\geq1$, the puzzle piece $P_k(\c')$
		  has only finitely many successors.
	\end{itemize}

\vspace{4pt}

We first choose an large integer $L_0>0$ so that 

\begin{itemize}
		\item  For any different $\c_1,\c_2\in \mathcal{E}_{\rm crit}\cap \omega(\c)$, one has 
		$P_{L_0}(\c_1)\cap P_{L_0}(\c_2)\neq\emptyset$.
		\item For any $\c_1,\c_2\in \mc E_{\rm crit}$,  we have the implication
		$$\c_1\notin \omega(\c_2)
		\Longrightarrow \c_1\cap \bigcup_{k\geq 1} P_{L_0}(f^k(\c_2)) =\emptyset.$$
	\end{itemize}

Let $[\c]=\omega(\c)\cap  \mathcal{E}_{\rm crit}$ and ${\rm orb}([\c])=\bigcup_{\c'\in[\c]}\bigcup_{k\geq 0}f^{k}(\c')$.
The persistent recurrence of $\c$ 
allows one to construct the {\it principal nest}, whose significant properties are summarized  as follows
%
%
%
%

%
%
%

\begin{theorem} \label{p-nest} Assume $\c$ is persistently recurrent and $L_0>0$ is
chosen as above. 
  Then there is a nest of $\c$-puzzle pieces 
$$Q_{0}(\c)\supset Q_{1}(\c)\supset Q'_{1}(\c)\supset Q_{2}(\c)\supset Q'_{2}(\c)\supset\cdots,$$
each puzzle piece is a suitable pull back of $Q_{0}(\c)=P_{L_0}(\c)$ by some iterate of $f$,
satisfying the following properties:

(1). There exist integers $D_0>0$, $n_j>m_j\geq 1$ for all $j\geq 1$, so that
$$f^{m_j}: Q'_{j}(\c)\rightarrow Q_{j}(\c),  \ f^{n_j}: Q_{j+1}(\c)\rightarrow Q_{j}(\c)$$
are proper maps of degree $\leq D_0$, and $f^{n_j}(Q'_{j+1}(\c))\subseteq Q'_{j}(\c)$.

(2). The gap $d_j$ of the depths between $Q_j$ and $Q'_j$ satisfies 
$$d_j\rightarrow+\infty \text{ as } j\rightarrow+\infty.$$
%

%


(3). For  all $j\geq 1$,
$$(Q_{j}(\c)- \overline{Q'_{j}}(\c))\cap {\rm orb}([\c])=\emptyset.$$

(4).  We have the following asymptotic lower bound of moduli,
 $$\liminf_{j\rightarrow+\infty}{\rm mod}(Q_{j}(\c)- \overline{Q'_{j}}(\c))>0.$$
%
%

\end{theorem}


The construction of the principal nest is attributed to Kahn-Lyubich \cite{KL1} 
in the unicritical case, Kozlovski-Shen-van Strien \cite{KSS07} in the multicritical case.  The complex bounds are proven by Kahn-Lyubich \cite{KL1, KL09} (unicritical case), Kozlovski-van Strien \cite{KS} and Qiu-Yin \cite{QY} independently (multicritical case). 
The interested readers may see these references for  a detail construction of the nest and the proof of its properties.
We remark that in our setting, the annuli $Q_{j}(\c)- \overline{Q'_{j}}(\c)$ might be degenerate for the first few indices $j$'s. But because of the growth of the gaps $d_j$, the annuli $Q_{j}(\c)- \overline{Q'_{j}}(\c)$ will be non-degenerate when $j$ is large enough. That's the reason why we use the term `asymptotic lower bound' instead of 
`uniform lower bound' in Theorem \ref{p-nest}(4).

\begin{prop}
The end $\ee\in \mc{E}_{\rm w}^{\rm r}$ is trivial, if
 $\omega(\ee)\cap \mathcal{E}_{\rm crit}$ contains a persistently recurrent end $\c$.
%
			\end{prop}
		
		\begin{proof} 
		Let $(Q_{j}(\c),Q'_{j}(\c))$'s be the puzzle pieces of principal nest given by Theorem \ref{p-nest}. For each $j\geq 1$, let $r_j$ be the first entry time of $\ee$ into $Q'_{j}(\c)$.
		Let $T'_j(\ee)=L_{\ee}(Q'_{j}(\c))$, and  $T_j(\ee)$ be the component of $f^{-r_j}(Q_{j}(\c))$ containing $\ee$. Then Theorem \ref{p-nest}(3) implies that 
		$$\tu{deg}(f^{r_j}|_{T'_j(\ee)})=\tu{deg}(f^{r_j}|_{T_j(\ee)})\leq d^\kappa.$$
		Hence, by Theorem \ref{p-nest}(4) and let $\mu$ be the asymptotic lower bound of moduli, for all large $j$, we have
		$$\tu{mod}(T_j(\ee)\setminus\ol{T'_j(\ee)})\geq \tu{mod}(Q_k(\c)\setminus\ol{Q'_j(\c)})/d^\kappa\geq \mu/d^\kappa.$$
		It follows that $\ee$ is trivial.
		\end{proof}

\section{The renormalizable case}\label{pre-periodic}

As we have seen in the previous section,  wandering ends  are always trivial. 
However,  pre-periodic ends can be non-trivial (see Lemma \ref{s-r}).  Even though, the intersection 
of such end and the boundary of an immediate root basin, is trivial. The aim of this section is 
to prove this statement.

We first introduce  the {\it renormalization} of Newton maps. 
We say that the Newton map $f$ is {\it renormalizable} if there exist an integer $p\geq 1$, two multi-connected domains $U, V$ with $U\Subset V\subseteq \mathbb C$ such that $f^p:U\to V$ is a proper mapping with a 
connected {\it filled Julia set} $K(f^p|_U)=\bigcap_{k\geq 0} f^{-kp}(U)$.
The triple $(f^p,U,V)$ is called a {\it renormalization} of $f$.  Note that in the definition, we assume $\infty\notin V$ to exclude the existence of $f$-fixed point in $K(f^p|_U)$. 

In the definition, 
one may further require that $U, V$ are topological disks, and this kind of   renormalization is called  {\it P-renormalization} (here `P' refers to `polynomial-like').
For Newton maps, we have 
$$f\text{ is renormalizable }\Longleftrightarrow f \text{ is P-renormalizable}.$$

To see this, we only need show the `$\Longrightarrow$' part.   Suppose that 
$(f^p,U,V)$ is a renormalization of $f$, with $U, V$ multi-connected and $K(f^p|_U)$ connected.
The assumption $\infty \notin K(f^p|_U)$ implies that 
$K(f^p|_U)$ is disjoint from the boundary of puzzle pieces, hence contained in a periodic end $\ee\in\mc E$, which satisfies $f^p(\ee)=\ee$. 
Consider the map $f^{\ell p}: P_{\ell p}(\ee)\rightarrow P_{0}(\ee)$, choose an integer 
$\ell>0$ so that $P_{\ell p}(\ee)\Subset P_{0}(\ee)$, we see that
$(f^{\ell p}, P_{\ell p}(\ee), P_{0}(\ee))$ is a P-renormalization of $f$.

Because of this equivalence, when we are discussing the renormalizations of Newton maps, we always require that $U,V$ are topological disks.

Periodic ends are closely related to renormalizations: 

\begin{lemma} \label{s-r} Let $\ee$ be a periodic end, with period $p\geq 1$.

1. If none of $\ee, \cdots, f^{p-1}(\ee)$ is critical, then $\ee$ is a singleton.

2. If some end of $\ee, \cdots, f^{p-1}(\ee)$  is critical, then $f$ is renormalizable. In this case, $\ee$ is the filled Julia set of 
a renormalization.
\end{lemma}
\begin{proof} Choose a large integer $N>0$ so that 
$$(P_{N}(\ee)\cup \cdots\cup P_N(f^{p-1}(\ee)))\setminus (\ee\cup\cdots\cup f^{p-1}(\ee))$$
contains no critical point of $f$. By Proposition \ref{prop:local_connect_infty}, there is an integer $\ell>0$ so that
$P_{N+\ell p}(\ee)\Subset P_{N}(\ee)$. If none of $\ee, \cdots, f^{p-1}(\ee)$ is  critical, then $f^{\ell p}: P_{N+\ell p}(\ee)\to P_{N}(\ee)$ is conformal. Applying the Schwarz Lemma to its inverse, we see that $\ee$ is singleton.
If some end of $\ee, \cdots, f^{p-1}(\ee)$  is critical, then $(f^{\ell p}, P_{N+\ell p}(\ee), P_{N}(\ee))$ is a renormalization of $f$.
In this case, the filled Julia set $K(f^{\ell p}|_{P_{N+\ell p}(\ee)})=\bigcap_{k\geq 1} P_{N+k\ell p}(\ee)=\ee$.
\end{proof}

The main result of this section is the following

\begin{proposition}\label{prop:per} For any pre-periodic end $\ee\in  \mc{E}_{\rm pp}$  and any immediate root basin $B\in {\rm Comp}(B_f) $, 
	the intersection $\ee\cap\overline{B}$ is either empty or a singleton. 
\end{proposition}
\begin{proof}
It suffices to treat the periodic case. We may assume $\ee$ is non-trivial,  of period $p>1$ (note that
$p=1$ iff $\ee=\ee(\infty)=\{\infty\}$), and  
$\ee\cap \overline{B}\neq \emptyset$ for some immediate root basin $B$.
The idea of the proof is to construct a Jordan curve 
separating $\ee$ from $B$.

%
	
	By  Proposition \ref{prop:local_connect_infty}, one can find two  puzzle pieces $Q_1$ and $Q_0=f^{n_0p}(Q_1)$, such that	
	$\ee\Subset Q_1\Subset Q_0$. Assume the depths of $Q_1,Q_0$ are large enough so that all critical points of  $g:=f^{n_0p}:Q_1\to Q_0$ are contained in $\ee$. Let $d_\ee={\rm deg}(g|_{Q_1})$, then  $d_\ee\geq 2$, otherwise, $g$ is conformal and the Schwarz Lemma
	 would imply that $\ee$ is trivial. 
	
%
	
	Write $Q_k=g^{-k}(Q_0)$ for $k\geq 1$. By Proposition \ref{cor:non-cutpoints_graph}, 
	for all $k\geq 0$, there exist $\alpha_k,\beta_k\in \mb R/\mb Z$ with $\alpha_k<\beta_k$, and $r_k\in(0,1)$ such that 
	$$\ol{Q}_k\cap \ol{B}=\ol{Q_k\cap B}=\ol{S_B(\alpha_k,\beta_k;r_k)}.$$ 
	Since $f|_B$ is conjugate to $z\mapsto z^{d_B}$ on $\mb{D}$, we have
	$$\alpha_k\leq \alpha_{k+1}< \cdots<\beta_{k+1}\leq \beta_k,\  |\beta_{k+1}-\alpha_{k+1}|=|\beta_k-\alpha_k|/{d_B^{n_0p}}.$$
	Therefore the sequences $\{\alpha_k\}$ and 
	$\{\beta_k\}$ have a common limit $\theta=\tu{lim }\alpha_k=\tu{lim }\beta_{k}$. The internal ray $R_B(\theta)$ of $B$ is invariant under $g$, hence land at a $g$-fixed point $q\in \ee\cap \partial B$. 
		
	 In the following, we show $\ee\cap \partial B=\{q\}$. To this end, let $\eta_\varepsilon={R_B(\theta)\cap  Q_\varepsilon}$ with $\varepsilon\in\{0,1\}$.  
Let $\phi:\olC\setminus \ee\to \olC\setminus\olD$ be a Riemann mapping, and denote
$$(\wh{\eta}_\varepsilon,\wh{B},\wh{Q}_\varepsilon)=(\phi(\eta_\varepsilon),\phi(B),\phi(Q_\varepsilon\setminus \ee)).$$
  Then 
   $\wh{g}=\phi\circ g\circ \phi^{-1}:\wh{Q}_1\to\wh{Q}_0$ is  a covering map between annuli, of  degree $d_\ee$. By Schwarz reflection principle, we may assume that $\wh{g}$ is holomorphic in a neighborhood of $\partial\mb{D}$. The arc $\wh{\eta}_\varepsilon$ is $\wh{g}$-invariant, hence lands at a $\wh{g}$-fixed point, say  $\wh q$, on $\partial\mb{D}$.

	Let $\Omega_+, \Omega_-$ be the two components of $\wh{g}^{-1}(\wh{Q}_0\setminus\wh{\eta}_0)$
	such that
	 $\wh{\eta}_1\subseteq \partial \Omega_+\cap \partial \Omega_-$. 
	 Clearly,  $\Omega_+,\Omega_-$  are Jordan disks.

  \begin{figure}[h] 
	 \begin{center}
\includegraphics[height=5.6cm]{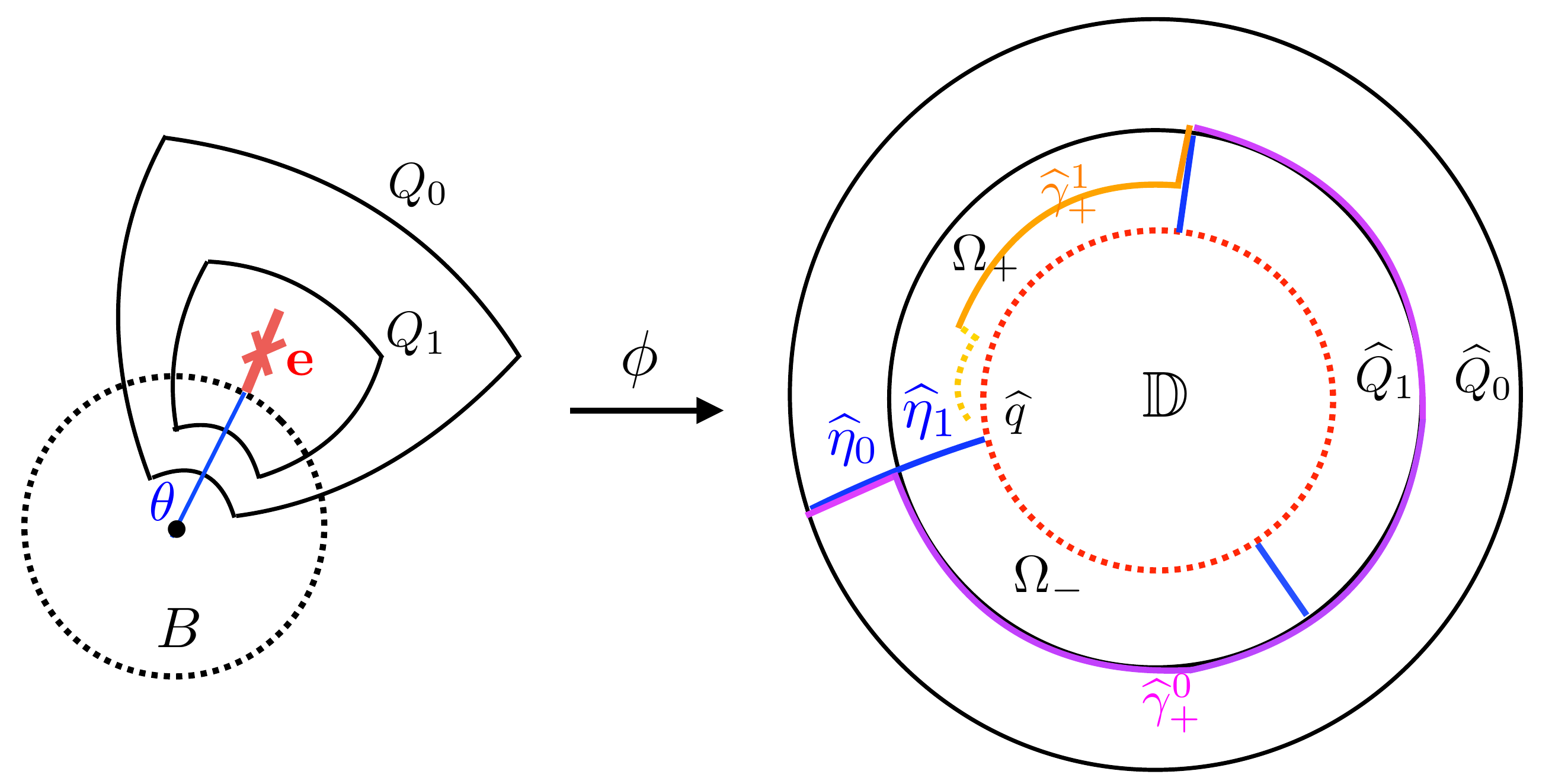} 
 \caption{Some domains and construction of curves converging to $\wh{q}$.}
\label{renormalizable}
\end{center}
\end{figure}

\vspace{3pt}
	
	\tb{Claim 1:} {\it The map $\wh{g}$ has exactly one fixed point on $\ol{\Omega}_+$ (or $\ol{\Omega}_-$). This fixed point is $\wh q \in \partial {\Omega}_+\cap  \partial {\Omega}_-$.}
	
	\begin{proof}
Let $\Omega_+^*, \Omega_-^*, \wh{\eta}_1^*$ be the reflection part of $\Omega_+,\Omega_-,\wh{\eta}_1$ with respect to the circle $\partial \mb D$.	
Let $$Y=\Omega_+^*\cup\Omega_-^*\cup\wh{\eta}_1^*\cup \Omega_+\cup\Omega_-\cup\wh{\eta}_1\cup \{\wh{q}\}.$$
Clearly, $Y$ is an open topological disk.
 The Schwarz reflection principle guarantees  that 
$\wh{g}$ can be defined in $Y$, and $Y\Subset \wh{g}(Y)$. 	Let $X$ be the component of $\wh{g}^{-1}(Y)$ containing $\wh{q}$.
One may verify that $X\Subset Y$ and $\wh{g}:X\rightarrow Y$ is conformal. Applying Schwarz Lemma to $\wh{g}|_X^{-1}: Y\rightarrow X$,
 we conclude that $\wh{g}$ has exactly one repelling fixed point on $X$.
 This fixed point is $\wh{q}\in \partial\Omega_+\cap\partial\Omega_-$.
 \end{proof}

%
%
%
%

	Let $\Omega_\varepsilon^0=\ol{\Omega}_\varepsilon\setminus  \overline{\wh{\eta}_1}$ for $\varepsilon\in\{\pm\}$. 
We consider the bijections 
	$$\wh{g}_\varepsilon=\wh{g}|_{\Omega_\varepsilon^0}: \Omega_\varepsilon^0\to \ol{\wh{Q}}_0, \ \ 
	\varepsilon\in\{\pm\}.$$

One may verify that $\wh{g}_\varepsilon$ is  conformal in the interior of $\Omega_\varepsilon^0$.
	
	\vspace{5pt}
	
	\tb{Claim 2:} {\it For each $\varepsilon\in\{\pm\}$,  let's define a sequence of closed Jordan arcs
	$$\wh{\gamma}^0_\varepsilon=\overline{(\phi(\partial Q_1)\setminus \Omega_\varepsilon^0)\cup (\wh{\eta}_0\setminus\wh{\eta}_1)}, \text{ and } \ \wh{\gamma}_\varepsilon^k=\wh{g}_\varepsilon^{-k}(\wh{\gamma}^0_\varepsilon), \ k\geq 1.$$
	Then $\wh{\gamma}_\varepsilon=\bigcup_{k\geq 1}\wh{\gamma}_\varepsilon^k$ is a Jordan arc in $\Omega_\varepsilon^0$, satisfying that
	
	1. $\wh{\gamma}_\varepsilon$ is disjoint from $\olD$;

    2. $\wh{\gamma}_\varepsilon$ is disjoint from the closure of $\wh{B}$;

    3. $\wh{\gamma}_\varepsilon$ converges to the $\wh{g}$-fixed point $\wh{q}$.}
		
%
	\begin{proof} We only prove the case $\varepsilon=+$, the other is similar.
	
	1. It suffices to note that $\ee$ has no intersection with 
	$\partial Q_1\cup\overline{(\eta_0\setminus\eta_1)}$.
	
	2.  Note that
	$$\wh{\gamma}_+\cap \ol{\wh{B}}=\emptyset \Longleftrightarrow \wh{\gamma}_+^1\cap \ol{\wh{B}}=\emptyset
	\Longleftrightarrow \phi^{-1}(\wh{\gamma}_+^1)\cap \ol{B}=\emptyset.$$
By	Proposition \ref{cor:non-cutpoints_graph},
$$\phi^{-1}(\wh{\gamma}_+^1) \subseteq \ol{Q}_1\setminus \ol{B}=\ol{Q}_1\setminus \ol{S_B(\alpha_1,\beta_1;r_1)}
\Longrightarrow \phi^{-1}(\wh{\gamma}_+^1)\cap \ol{B}=\emptyset.$$	
%
%
%
		

	3. Note that $\wh{\gamma}_+^2\Subset Y$ and $\wh{g}^{-1}: Y\rightarrow X$ is strictly contracting,
	we conclude that	$\wh{\gamma}_+$ converges to the $\wh{g}$-fixed point $\wh{q}$.
	\end{proof}	
		
		\vspace{5pt}
	
	 \tb{Claim 3:} {\it  For each $\varepsilon\in\{\pm\}$, the curve 	
		 $\gamma_\varepsilon=\phi^{-1}({\wh\gamma}_\varepsilon)$ satisfies that $\gamma_\varepsilon\cap
		(\ee\cup \ol{B})=\emptyset$ and  converges to  the  $g$-fixed point  ${q}$.}
		\vspace{4pt}
\begin{proof} 
By Claim 2, we see that $\gamma_\varepsilon$ is disjoint from $\ee\cup \ol{B}$.
Let $$V=\phi^{-1}\big(Y\setminus \ol{\mb D}), \ U=\phi^{-1}\big(X\setminus \ol{\mb D}\big).$$
Clearly, $V$ is a topological disk, $V\subseteq Q_1$ and $q\in \partial V$, and $g: U\rightarrow V$ is conformal.
Let $h=g|_{U}^{-1}: V\rightarrow U$. 
Since the ray $R_B(\theta)$ converges to $q$,  the family of maps $\{h^k\}_{k\in \mb N}$ converge uniformly on 
$\ol{R_B(\theta)\cap(Q_0\setminus Q_1)}$ to the boundary point $q$.
By Denjoy-Wolff's Theorem (see \cite{D},\cite{W}), the maps $\{h^k\}_{k\in \mb N}$ converge uniformly on any compact subset of $V$, in particular on $\gamma_+^2=\phi^{-1}(\wh{\gamma}_+^2)\Subset V$, 
to the boundary point $q$. 
Hence $\gamma_+$ converges to ${q}$.

Similar argument works for $\gamma_-$.
\end{proof}

Now we define the Jordan curve by
	\begin{equation*}
{\gamma}=
\begin{cases}
{\gamma}_+\cup {\gamma}_-\cup\{q\}\cup \left(\partial Q_1\setminus \partial V\right),\ \  & \text{ if } d_\ee\geq 3,\\
({\gamma}_+\cup{\gamma}_-\cup\{q\})\setminus g^{-1}(\eta_0\setminus \ol{\eta_1}),\ \ &\text{ if } d_\ee=2.
\end{cases}
\end{equation*}

%
%

Then the  sets $\ol{{B}}\setminus\{q\}$ and $\ee\setminus\{q\}$ are in different components of $\olC-{\gamma}$. It follows that  $\ee\cap \ol{B}=\{q\}$, completing the proof.
\end{proof}


%
%

\section{Proof of the main theorem}\label{proof-main}

In this section, we will complete the proof of  Theorem \ref{main}.
At the end, we give some concluding remarks.

\subsection{Proof of Theorem \ref{main}}

To prove the local connectivity of $\partial B$, it is equivalent to show that for any immediate root basin $B\in {\rm Comp}(B_f)$, and any $z\in \partial B$, the intersection $\ee(z)\cap \partial B$ is a singleton. 

This actually follows from the decomposition 
$$\mc{E}=\mc{E}_{\rm pp}\sqcup \mc{E}_{\rm w}^{\rm pp}\sqcup \mc{E}_{\rm w}^{\rm nr}\sqcup \mc{E}_{\rm w}^{\rm r}$$
and
Sections \ref{wandering} and \ref{pre-periodic}.


It remains to show that $\partial B$ is a Jordan curve iff $d_B={\rm deg}(f|_B)=2$. In fact, if $d_B\geq 3$, then there are $d_B-1\geq 2$ internal rays in $B$, landing at $\infty$, so $\partial B$ is not  
a Jordan curve.  If $d_B=2$, it follows from 
Lemma \ref{num-lands} and Corollaries \ref{num-land}, \ref{Jordan-2} (see below) that
$\partial B$ is a Jordan curve.

\begin{lemma}\label{num-lands}
Let $B\in {\rm Comp}(B_f)$. 
 If two different internal rays
$R_B(\theta_1), R_B(\theta_2)$  land at the same point, then
$$f(R_B(\theta_1))\neq f(R_B(\theta_2)).$$
\end{lemma}
\begin{proof}
   We need discuss two cases: $f(B)=B$ and $f(B)\neq B$.

\vspace{4 pt}

\textbf{Case 1: $f(B)=B$}. 

\vspace{4 pt}

 In this case, $f|_B$ is conjugate to the map $z^{d_B}|_{\mathbb D}$. To discuss the relative position of the internal rays, we need consider the  angle tupling map on the circle.
    Let $m_{d_B}: t\mapsto d_Bt \ ({\rm mod }\  \mathbb Z)$ be the angle tupling map on $\mb{R}/\mb{Z}$.
    Note that  $S_0:=\big\{\frac{0}{d_B-1},\cdots,\frac{d_B-2}{d_B-1}\big\}$ is the set of fixed points of $m_{d_B}$. The components of $\mb{R}/\mb{Z}\setminus S_0$ are denoted by $I_k=\big(\frac{k}{d_B-1}, \frac{k+1}{d_B-1}\big), 0\leq k\leq d_B-2$. 
    
   First, note that the statement is true when one of $\theta_1,\theta_2$ is in $S_0$. 
    In the following, we assume $\theta_1,\theta_2 \notin S_0$.
      We will prove by contradiction. 
    If $f(R_B(\theta_1))=f(R_B(\theta_2))$, then fact that 
    $\bigcup_{\theta\in S_0}\overline{R_B(\theta)}$ divides $\ol{B}$ into $d_B-1$ parts, implies that one of them  
    contains  $R_B(\theta_1),R_B(\theta_2)$, together with their common landing point $z$.
    Without loss of generality, 
   we assume 
   $$0<\theta_1<\theta_2<{1}/(d_B-1).$$

%
%
	
The assumption implies that $\theta_1,\theta_2\in I_0$.
	Consider the action of $m_{d_B}$ on the open arc $I_0$. Let $S_1=f^{-1}(S_0)\cap I_0$. Then $S_1=\Big\{\frac{1}{d_B(d_B-1)},\cdots,\frac{d_B-1}{d_B(d_B-1)}\Big\}$. 
	Since $m_{d_B}$ is injective on $S_1$, the assumption $f(R_B(\theta_1))=f(R_B(\theta_2))$ implies that  $\theta_1,\theta_2
	\notin S_1.$
%
	
	The set $I_0\setminus S_1$ consists of $d_B$ components: 
	$$J_k=\Bigg(\frac{k-1}{d_B(d_B-1)},\frac{k}{d_B(d_B-1)}\Bigg), \ 1\leq k\leq d_B.$$
	 
	  Note that on each $J_k$, the map $m_{d_B}$ is one-to-one. Thus $\theta_1,\theta_2$ belong to distinct $J_k$'s. Since $m_{d_B}(J_1)=m_{d_B}(J_{d_B})=I_0$,
	  we conclude that 
	  $\theta_1\in J_1, \theta_2=\theta_1+\frac{1}{d_B}\in J_{d_B}$. 
	 For $k\in\{1, d_B\}$, we denote by $\theta_{1,k},\theta_{2,k}\in J_k$ such  that $m_{d_B}(\theta_{1,k})=\theta_1,m_{d_B}(\theta_{2,k})=\theta_2$, then we have 
	   $$\theta_{1,1}=\frac{\theta_1}{d_B}, \theta_{2,1}=\frac{1}{d_B}\Big(\theta_1+\frac{1}{d_B}\Big), \theta_{1,d_B}=\theta_{1,1}+\frac{1}{d_B}, \theta_{2, d_B}=\theta_{2,1}+\frac{1}{d_B}.$$ 
	   It's easy to see that $\theta_{1,1}<\theta_1<\theta_{2,1}<\theta_{1,d_B}<\theta_{2}<\theta_{2, d_B}$.
%
	 It follows that $R_B(\theta_{2,1})\cup R_B(\theta_{1,d_B})\subseteq S_B(\theta_1,\theta_2;0)$. 
	
	Let $W$ be the component of  $\olC-\overline{R_B(\theta_1)\cup R_B(\theta_2)}$ such that $\infty\notin W$.
	Clearly,   $W$ contains no fixed point,  because $W$ is disjoint from the channel graph $\Delta_0$ which contains all fixed points of $f$.
	 By above discussion, there is a component $V$ of $f^{-1}(W)$, 
	such that $V$ contains $S_B(\theta_{1,1},\theta_{2,1};0)$ (or $S_B(\theta_{1,d_B},\theta_{2,d_B};0)$).
	The facts 
	$$R_B(\theta_1)\subseteq S_B(\theta_{1,1},\theta_{2,1};0) \text{ and } \partial V\cap J(f)\subseteq f^{-1}(q)$$
 imply that 
	  $\partial V$ contains the common landing point $q$ of $R_B(\theta_1),R_B(\theta_2)$.
	Since $f(\partial V\cap J(f))\subseteq \partial W\cap J(f)=\{q\}$, we see that $q$ is a fixed point of $f$, which is necessarily $\infty$. This contradicts the assumption $\theta_1,\theta_2 \notin S_0$.

	\vspace{4 pt}

\textbf{Case 2: $f(B)\neq B$}. 

\vspace{4 pt}

 
    Assume $f(R_B(\theta_1))=f(R_B(\theta_2))$. Let $U\subseteq f(B)$ be a Jordan disk, whose boundary passes through two endpoints of $f(R_B(\theta_1))$. Let $D=\olC\setminus \overline{U}$.
    
    Let $W$ be the component of  $\olC-\overline{R_B(\theta_1)\cup R_B(\theta_2)}$ such that $\overline{W}\cap \Delta_0=\emptyset$. Then $\overline{W}$  contains no fixed points of $f$, because all fixed points of $f$ are contained in the channel graph $\Delta_0$. 
    Clearly $\olC\setminus f(R_B(\theta_1))\subseteq f(W)$ 
     and $W\subseteq D$. There is a component $V$ of $f^{-1}(D)$ contained in $W$.
    In particular, $\overline{V}$ contains no fixed point of $f$.
    By Corollary \ref{cor:counting_number}, there is at least one fixed point in $\overline{V}$.
     This is a contradiction.
%
%
  \end{proof}

%


\begin{corollary}\label{num-land} 
For any $B\in \tu{Comp}(B_f)$ and any $z\in \partial B$, let $\mu_B(z)$ be the number of internal rays in $B$ landing at $z$. Then we have 
$$\mu_B(z)\leq \mu_{f(B)}(f(z)), \ \forall z\in \partial B.$$
In particular,  
\begin{equation*}
\mu_B(z)
\begin{cases}
=1,\ \  & \text{ if }z\in \partial B\setminus \Omega_f,\\
\leq d_{f^\ell(B)}-1,\ \ &\text{ if }z\in \partial B\cap \Omega_f,
\end{cases}
\end{equation*}
where $\ell\in\mathbb{N}$ is chosen so that ${f^\ell(B)}$ is  fixed.
%
\end{corollary}
\begin{proof} By Lemma \ref{num-lands}, one has 
$$\mu_B(z)\leq \mu_{f(B)}(f(z)), \ \forall B\in \tu{Comp}(B_f), \forall z\in \partial B.$$

For $z\in \partial B\cap \Omega_f$, 
let $\ell\in \mb N$ be chosen so that $f^\ell(z)=\infty$ and $f^\ell(B)$  fixed, then 
 $$\mu_B(z)\leq \mu_{f^\ell(B)}(\infty)=d_{f^\ell(B)}-1.$$

To prove $\mu_B(z)=1$ for $z\in \partial B\setminus \Omega_f$, it suffices to consider the fixed case: $f(B)=B$. 
In this case, for any $ z\in \partial B\setminus \Omega_f$, if $\mu_B(z)\geq 2$, then there are two 
 internal rays $R_B(t_1), R_B(t_2)$, with $t_1<t_2$, landing at $z$.
 It follows that $R_B(t_1), R_B(t_2)$ are contained in the same component of $\olC-\Gamma_B$, where 
 $\Gamma_B=\bigcup_{0\leq k\leq d_B-2} \overline{R_B(k/(d_B-1))}$. 
 This implies that 
 $$0<t_2-t_1<{1}/{(d_B-1)}.$$

  It follows that for all $k\geq 0$,  
the two rays $R_B(d_B^kt_1), R_B(d_B^kt_2)$ land at the common point $f^k(z)$. On the other hand,  the assumption $ z\in \partial B\setminus \Omega_f$ implies for  $k_0\geq 1$,
satisfying that
$$d_B^{k_0}(t_2-t_1)> {1}/{(d_B-1)}\geq d_B^{k_0-1}(t_2-t_1),$$ 
 the rays
$R_B(d_B^{k_0}t_1), R_B(d_B^{k_0}t_2)$ are contained in different components of 
$\olC-\Gamma_B$,
hence can not land at the same point.   This is a contradiction.
\end{proof}

As consequence of Corollary \ref{num-land}, if $d_{f^\ell(B)}=2$, we have $\mu_B(z)=1$ for all $z\in \partial B$. This fact can be stated in the following form:
\begin{corollary}\label{Jordan-2} 
For any ${B}\in\tu{Comp}(B_f)$ 
which is eventually iterated to
 an immediate root basin $B_0$ with $d_{B_0}=2$, the boundary $\partial {B}$
is a Jordan curve. 
\end{corollary}

We remark that for Corollary \ref{num-land}, 
when $f(B)=B$ and  $d_B\geq 3$, it can happen that for some $z\in \partial B\cap \Omega_f$, the strict 
inequality $$\mu_B(z)<d_B-1$$
holds.
Figure \ref{fig:newton2} provides such an  example.
In fact, we have an even more interesting example.
\begin{example} It can also happen that for some ${B}\in\tu{Comp}(B_f)$ 
which is eventually iterated to
 an immediate root basin $B_0$ with $d_{B_0}>2$, and such that $B\neq B_0$, the boundary $\partial {B}$
is a Jordan curve.

Figure \ref{New1} gives an example of degree five Newton map $f$, with an immediate root basin $B_0$ such that $d_{B_0}=3$.  For this example,  the boundary of any ${B}\in\tu{Comp}(B_f) \setminus\{B_0\}$  is a Jordan curve. 
\end{example}

%

%
%
%
%
%

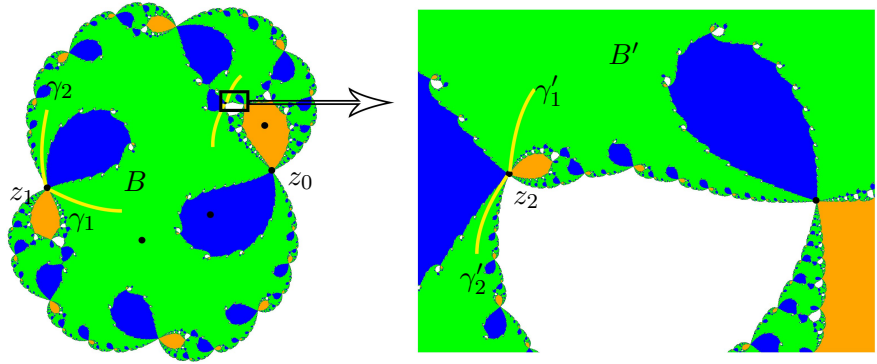
\begin{figure}
	\center
	\begin{tikzpicture}
	\node at (0,0) {\includegraphics[width=12cm]{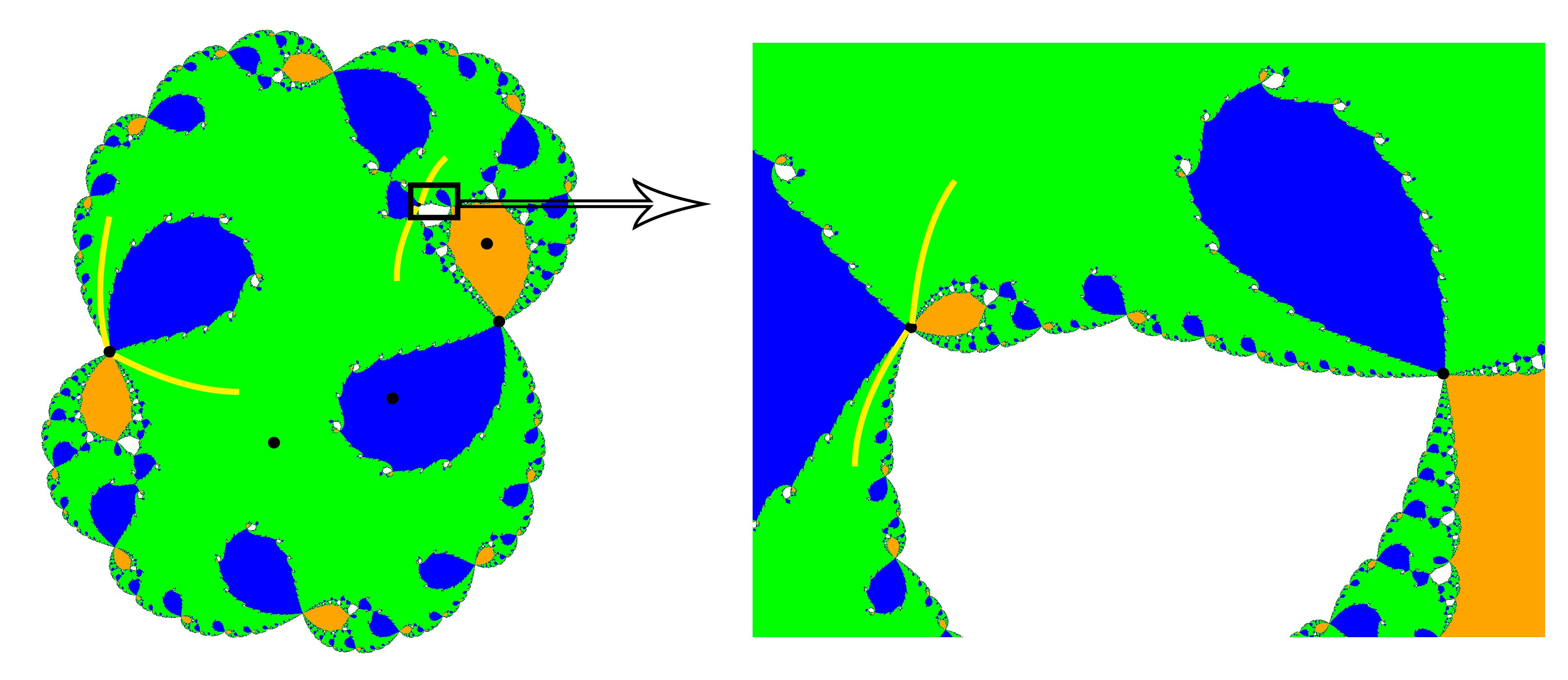}};
	\node at (-1.8,0){$z_0$};
	\node at (-5.5,-0.2){$z_1$};
	\node at (1.2,-0.25){$z_2$};
	\node at (-4,0){$B$};
	\node at (2.5,1.7){$B'$};
	\node at (-4.7,-0.55){$\gamma_1$};
	\node at (-5,1.2){$\gamma_2$};
	\node at (1.5,1.2){$\gamma_1'$};
	\node at (0.5,-1.3){$\gamma_2'$};
	\end{tikzpicture}
		\caption{This degree four Newton map $f$ sends points $z_2\mapsto z_1\mapsto z_0=\infty$ and Fatou components $B'\mapsto B\mapsto B$. As shown above, $z_1$ has two non-homotopic accesses $\gamma_1,\gamma_2$ from $B$, while $z_2$ has one access $\gamma_2'$ from $B$ and another access $\gamma_1'$ from $B'$, here $f(\gamma_k')=\gamma_k, k\in\{1,2\}$.
	}
	\label{fig:newton2}
\end{figure}

  \begin{figure}[h] 
	 \begin{center}
\includegraphics[height=5.6cm]{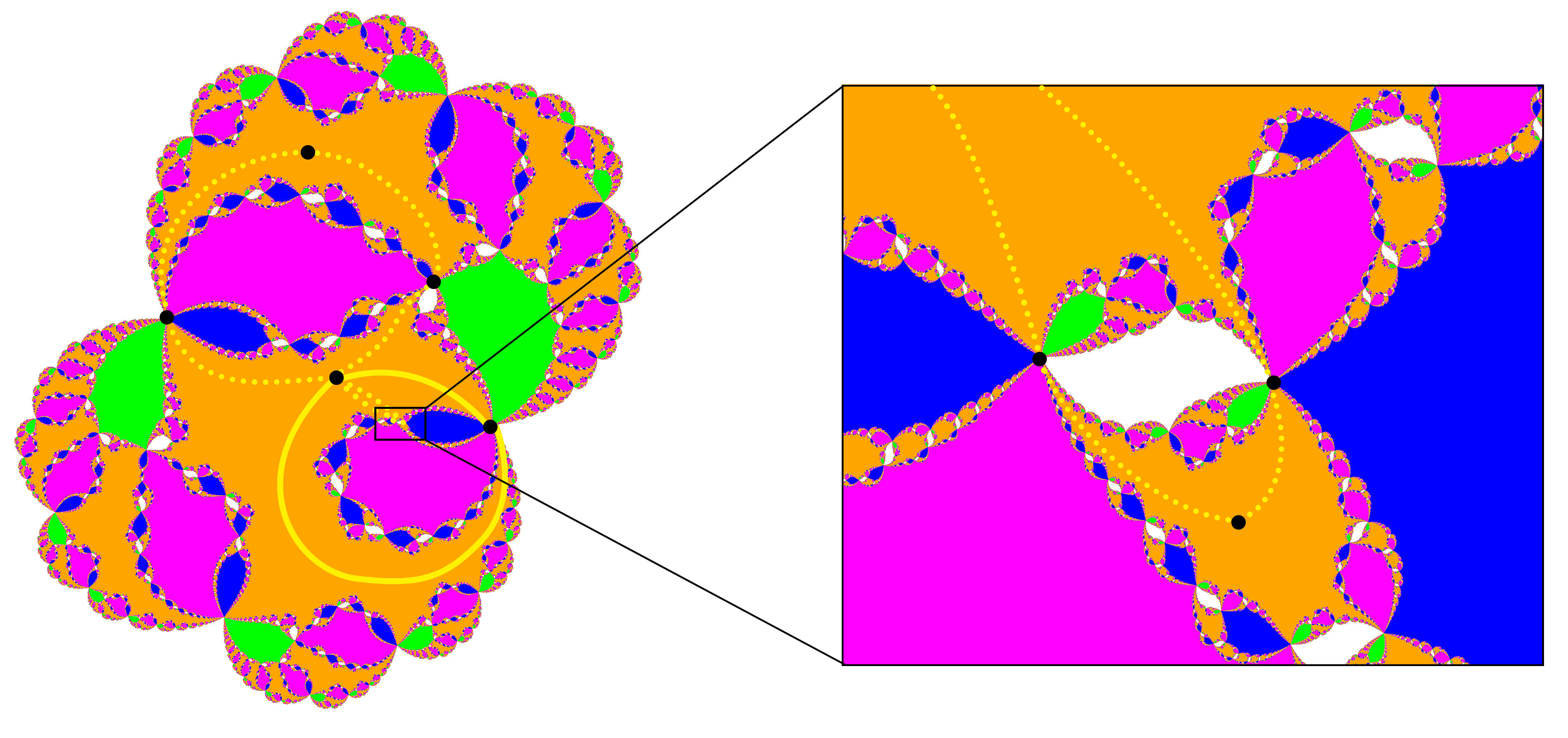} 
 \put(-295,65){$B_0$}  \put(-230,63){$\infty$}  \put(-280,33){$\gamma$} 
 \caption{There are two internal rays in $B_0$ converging to $\infty$. The union of their closures gives a Jordan curve $\gamma$. Its preimage $f^{-1}(\gamma)$ consists of three Jordan curves: one is $\gamma$; the other two are mapped onto $\gamma$ by degree two, hence each encloses a critical point.
 }
\label{New1}
\end{center}
\end{figure}

\subsection{Concluding remarks}


There are two by-products of our whole proof:

%
%
%
%
%

\begin{enumerate}
		\item {\it The Julia set $J(f)$ of a non-renormalizable Newton map $f$ is locally connected.}
		\item  {\it A wandering continuum\footnote{ A continuum  (compact set, which is connected and non-singleton) $E$ is called {\it wandering} under $f$, if $f^m(E)\cap f^n(E)=\es$ for all $0\leq m<n$.} $E\subseteq J(f)$ of the Newton map $f$ will eventually be iterated into the filled Julia set of a renormalization.}	\end{enumerate} 
		
		To see (1), it suffices to observe that for a  non-renormalizable Newton map $f$, each periodic end is a singleton (by Lemma \ref{s-r}). Combining Section \ref{wandering}, we see that all possible type of ends are trivial.

		To see (2), note that $\infty\notin E$, which implies that $E$ is contained in some end $\ee$.
		If $\ee$ is  wandering, then it is trivial by  Section \ref{wandering}. This is impossible because $E$ is a continuum. So $\ee$ is pre-periodic. By Lemma \ref{s-r}, for some $k\geq0$, the end $f^k(\ee)$ is periodic and equal to a filled Julia set of a renormalization.

\end{document}